\documentclass[11pt, a4paper]{article}

\usepackage{framed,multirow}

\usepackage{amssymb}
\usepackage{amsmath,bm}
\usepackage{latexsym}
\usepackage[left=1.7cm, right=1.7cm]{geometry}

\usepackage{url}
\usepackage{xcolor}
\definecolor{newcolor}{rgb}{.8,.349,.1}

\usepackage{float}
\usepackage{tikz}
\usepackage{listings}
\usetikzlibrary{decorations.pathreplacing,calc}
\newcommand{\tikzmark}[1]{\tikz[overlay,remember picture] \node (#1) {};}
\usepackage{hyperref}
\usepackage{amsthm}

\usepackage{caption}
\usepackage{subfig}
\usepackage{hyperref}
\usepackage{thmtools}
\usepackage{cleveref}
\newcommand*{\AddNote}[4]{%
    \begin{tikzpicture}[overlay, remember picture]
        \draw [decoration={brace,amplitude=.5em},decorate,ultra thick,black]
            ($(#3)!([yshift=1.5ex]#1)!($(#3)-(0,1)$)$) --  
            ($(#3)!(#2)!($(#3)-(-0,1)$)$)
                node [align=center, text width=2.5cm, pos=.5, anchor=west] {#4};
    \end{tikzpicture}}

\usepackage{algorithm}
\usepackage{adjustbox}	

\def\coloneqq{{:=}}

\newtheorem{theorem}[subsection]{Theorem}

\newtheorem{example}[subsection]{Example}
\newtheorem{lemma}[subsection]{Lemma}

\newtheorem{proposition}[subsection]{Proposition}
\theoremstyle{remark}

\begin{document}

\title{Numerical Discretisation of Hyperbolic Systems of Moment Equations Describing Sedimentation in Suspensions of Rod-Like Particles }%

\author{Sina Dahm$^\dagger$, Jan Giesselmann\thanks{Technical University of Darmstadt, Department of Mathematics, 64293 Darmstadt, Germany} , Christiane Helzel\thanks{Heinrich Heine
  University Düsseldorf, Faculty of Mathematics and Natural Sciences,
  Institut of Mathematics, 40225 D\"usseldorf, Germany} , 
}

\maketitle

\begin{abstract}
We present a numerical discretisation of the coupled moment systems,
previously introduced in Dahm and Helzel \cite{Dahm}, which
approximate the kinetic multi-scale model by Helzel and Tzavaras
\cite{Helzel1} for sedimentation in suspensions of rod-like particles
for a two-dimensional flow problem and a shear flow problem. We use a
splitting ansatz which, during each time step, separately computes the
update of the macroscopic flow equation and of the moment system. The
proof of the hyperbolicity of the moment systems in \cite{Dahm}
suggests solving the moment systems with standard numerical methods
for hyperbolic problems, like LeVeque's Wave Propagation Algorithm
\cite{LeV}. The number of moment equations used in the hyperbolic
moment system can be adapted to locally varying flow features.
An error analysis is proposed, which compares the approximation with
$2N+1$ moment equations to an approximation with $2N+3$ moment
equations. This analysis suggests an error indicator which can be
computed from the numerical approximation of the moment system with
$2N+1$ moment equations.
  In order to use moment approximations with a different number of
  moment equations in different parts of the computational domain, 
we consider an interface coupling of moment systems with different
resolution.
Finally, we derive a conservative high-resolution Wave Propagation Algorithm
for solving moment systems with different numbers of moment equations.

\end{abstract}

\section{Introduction}
We are interested in the development of numerical methods for solving
the coupled moment systems, introduced in Dahm and Helzel \cite{Dahm},
which approximate the kinetic multi-scale model by Helzel and Tzavaras
\cite{Helzel1} for sedimentation in suspensions of rod-like
particles. A typical phenomenon during the sedimentation process in
initially well-stirred suspensions of rod-like particles under the
influence of gravity is the formation of concentration
instabilities. Guazzelli and coworkers observed experimentally in
\cite{Guazzelli}, \cite{Guazzelli1} that after some time and under the
influence of gravity, the interplay of the particle orientation and
the flow field generated by the sedimenting rods leads to a
destruction of the spatially homogeneous distribution of the rods and
structural instabilities like cluster formations. While the rods are
nearly isotropic in regions with low particle densities, they are
strongly oriented in the direction of gravity in regions of particle
packages.

First numerical simulations in \cite{Dahm} have shown that different
levels of detail are required to accurately approximate the spatially
varying behaviour of the particles. While a high number of moment
equations is needed to resolve the complex flow structure in spatial
regions of the domain with clusters, few moment equations are
sufficient in spatial regions of the domain with low particle
densities. Thus, for deriving an accurate and efficient approximation
of the concentration instabilities observable during the process of
sedimentation in suspensions of rod-like particles, the number of
moment equations used in the hyperbolic moment system should be
adapted to locally varying flow features and accuracy requirements.

The central goal of this paper is to describe numerical
discretisations for the coupled moment systems which can adaptively
adjust the level of detail. Alternatively, approximations of coupled
kinetic-fluid problems have been considered which directly approximate the high
dimensional kinetic equation, see for example \cite{WegenerKuzminTurek+2023}.

In \autoref{sec:model}, we introduce the multiscale model by Helzel
and Tzavaras \cite{Helzel1} for sedimentation in suspensions of
rod-like particles and its approximation by hyperbolic systems of
moment equations derived by Dahm and Helzel \cite{Dahm}. As in
\cite{Dahm}, we restrict our considerations to shear flow and
two-dimensional flow and more importantly allow the particles to
orient only on $S^1$, i.e. the plane spanned by the direction of shear
and the direction of gravity.
In \autoref{sec:estimate} an 
  error estimate for the one-dimensional moment system coupled to the
  flow equation is proved, which motivates  an error indicator that can
  be used in practical computations. 
In \autoref{sec:numerical}, the numerical discretisation of the
one-and two-dimensional homogenous moment system is presented. Since
it was shown in \cite{Dahm} that the one- and two-dimensional moment
systems are hyperbolic, we can solve the moment systems with the
high-resolution Wave Propagation Algorithm by LeVeque \cite{LeV}. We
distinguish between a uniform approximation, in which the number of
moment equations is fixed globally for the entire domain and a
non-uniform approximation, in which the number of moment equations is
adapted locally. We derive a conservative high-resolution Wave
Propagation Algorithm for solving moment systems with different
resolution. Bulk-coupling of the moment equations with the flow
equations is challenging. For shear flow, an inhomogeneous,
one-dimensional hyperbolic system is coupled to the flow
equation, which in the simplest case reduces to the diffusion
equation. For two-dimensional flow, an inhomogeneous, two-dimensional
hyperbolic system  is coupled to the Navier-Stokes
equation. In \autoref{sec:Bulk}, we provide a splitting algorithm for
solving the coupled moment system for shear flow and two-dimensional
flow. We provide accuracy studies for several test
problems and  illustrate
that the error indicator can efficiently be used to predict regions
with larger errors. 
\section{A Kinetic Model for Sedimentation in Suspensions of Rod-Like Particles and its Approximation by Hyperbolic Systems of Moment Equations}
\label{sec:model}
In this section, we briefly introduce the general multiscale model by Helzel and Tzavaras \cite{Helzel1} for sedimentation in dilute suspensions of rod-like particles under the influence of gravity as well as the simpler models for shear flow and two-dimensional flow with director $f$ on $S^1$. Moreover, we present the hyperbolic systems of moment equations derived by Dahm and Helzel \cite{Dahm}, which represent a lower-dimensional approximation of the kinetic equation. The reader is referred to \cite{Helzel1} and \cite{Dahm} for a detailed derivation of the models presented in this section. 
\subsection{Multiscale Models for Sedimentation in Suspensions of Rod-Like Particles}
In \cite{Helzel1}, Helzel and Tzavaras describe sedimentation in
dilute suspensions of inflexible rod-like particles with a
high-dimensional multiscale model which couples a kinetic Smoluchowski
equation for the rod orientation to a Navier-Stokes equation for the
macroscopic flow. Kinetic models of this type were established by Doi
and Edwards \cite{Doi}.

The mathematical model considers rigid rod-like particles in a dilute
suspension under the influence of gravity. Let $l$ denote the constant
length of the molecules and $b$ their constant width. As we consider
slender rods, we assume $b\ll l$. Let $\nu$ denote the constant number
density of the rod-like molecules. The characteristic feature of a
dilute suspension is that the rods are well separated, as expressed by
$\nu \ll l^{-3}$. Further, we assume that the density of particles is not constant in time and space so that clusters are allowed to form. Let $d$ be the spatial dimension. In a physical space $\Omega \subset \mathbb{R}^d$, the probability distribution function $f=f(t,\pmb{x},\pmb{n})$ models the time-dependent probability that a particle with orientation $\pmb{n} \in S^{d-1}$, where $S^{d-1}$ is the unit sphere embedded in $\mathbb{R}^d$, has a center of mass at position $\pmb{x} \in \mathbb{R}^d$. Moreover, $\pmb{u}(t,\pmb{x})$ describes the macroscopic velocity field and $p=p(t,\pmb{x})$ the pressure of the solvent. The model accounts for the effects of gravity which acts in the direction of $\pmb{e}_3$, where $\pmb{e}_3$ is the unit vector in the upward direction. In non-dimensional form, the multiscale model is given as
\begin{equation}
\label{eqn:allgemein}
\begin{aligned}
\partial_t f+ \nabla_x\cdot (\pmb{u}f)&+\nabla_n\cdot (P_{\pmb{n}^{\perp}}\nabla_{\pmb{x}}\pmb{un}f)-\nabla_x\cdot ((I+\pmb{n}\otimes \pmb{n})\pmb{e}_3f) \\
&=D_r\Delta_nf+\gamma\nabla_x\cdot (I+\pmb{n}\otimes \pmb{n})\nabla_xf, \\
\sigma&=\int_{S^{d-1}}(d\pmb{n}\otimes \pmb{n}-I)f\text{d}\pmb{n}, \\
Re\left(\partial_t\pmb{u}+(\pmb{u}\cdot \nabla_x)\pmb{u}\right)&= \Delta_x\pmb{u}-\nabla_xp+\delta \gamma \nabla_x\cdot \sigma-\delta \displaystyle\int_{S^{d-1}}f\text{d}\pmb{n}\pmb{e}_3,\\
\nabla_x\cdot \pmb{u}&=0. 
\end{aligned}
\end{equation}
We give a short explanation of the different terms in model (\ref{eqn:allgemein}). 
The transport of the center of mass of the rods due to the macroscopic
velocity $\pmb{u}$ and gravity is described with the second and fourth
term in the first line.  The third term models the rotation of the
axis of the particles due to a macroscopic velocity gradient $\nabla_x
\pmb{u}$, where $P_{n^{\perp}}\nabla_x\pmb{un}$ is the orthogonal
projection of the vector $\nabla_x \pmb{u} \pmb{n}$ onto the tangent space in
$\pmb{n}$. On the right hand side of the first equation, rotation and
translation of the rod-like particles due to Brownian motion is
modeled. The dynamic of an incompressible fluid is described by a
Navier-Stokes equation which is extended by an additional elastic
stress tensor $\sigma$ and a buoyancy term. Thermodynamic consistency
justifies the form of $\sigma$ as shown in \cite{Helzel1}. Moreover,
four non-dimensional parameters are used in the full model
(\ref{eqn:allgemein}): $Re$, a Reynolds number based on the
sedimentation velocity, $D_r$, the rotational diffusion coefficient,
$\delta$, which measures the relative importance of buoyancy versus
viscous stresses and $\gamma$, which measures the relative importance
of elastic forces over buoyancy forces. We will restrict to the case
$\gamma=0$.

Physical applications of the model assume a three-dimensional physical space $\Omega \subset \mathbb{R}^3$ in which the orientation of the particles is characterised by a director $\pmb{n} \in S^2$. In this case, model (\ref{eqn:allgemein}) is a time-dependent five dimensional system of coupled partial differential equations. 
\subsubsection{Simplified Model for Two-Dimensional Flow}
\label{sec:Twoflow}
A simplification of the general model (\ref{eqn:allgemein}) can be
achieved by restricting to a two-dimensional flow and more
importantly, to restrict the orientation of the rod like particles to
take values only on $S^1$. In this case, we consider a velocity field of the form $\pmb{u}=(u(t,x,z),0,w(t,x,z))^T$. The director $\pmb{n} \in S^1$, which characterises the orientation of the rod-like particles, is restricted to take non-zero values only in the sphere embedded in the $(x,z)$-plane. We set $\pmb{n}=(\cos\theta,0,\sin\theta)$, with the angle $\theta \in [0,2\pi]$ measured counter-clockwise from the positive $x$-axis. For $\gamma=0$, the general model (\ref{eqn:allgemein}) reduces to 
\begin{equation}
\label{eqn:general}
\begin{aligned}
\partial_tf +\partial_\theta\left(\left(\left(\partial_z w- \partial_x
      u \right)\cos\theta\sin\theta- \partial_z u\sin^2\theta+
    \partial_x w \cos^2\theta\right)f\right)&& \\[6pt]
+ \partial_x\left(\left(u-\cos\theta\sin\theta\right)f\right)+\partial_z\left(w-\left(1+\sin^2(2\theta)\right)f\right)&=D_r\partial_{\theta \theta}f, \\
Re\left(\partial_tu+u\partial_xu+w\partial_zu\right)+\partial_x p
&=\partial_{xx}u
+\partial_{zz}u,\\
Re\left(\partial_tw+u\partial_xw+w\partial_zw\right)+\partial_z
p&=\partial_{xx}w
+\partial_{zz}w -\delta\displaystyle\int_0^{2\pi}f d\theta,\\
\partial_x u +\partial_z w&=0, 
\end{aligned}
\end{equation}
where $f=f(t,x,z,\theta)$ describes the distribution of the particles as a function of time $t$, space $(x,z) \in \mathbb{R}^2$ and orientation $\theta \in [0,2\pi]$. 
\subsubsection{Simplified Model for Shear Flow}
\label{sec:Shearflow}
Considering shear flow and orientations of particles restricted to
$S^1$ further simplifies the general model
(\ref{eqn:allgemein}). We assume $\pmb{u}=(0,0,w(t,x))^T$ and
$f=f(t,x,\theta)$. The most general form of the pressure which is
consistent with the ansatz of shear flow is $p=-\kappa(t) z$,
where $\kappa(t)$ can account for an externally imposed pressure
gradient. Here, see also \cite{Helzel1}, we use $\kappa = \delta \bar{\rho}$ where $\bar{\rho}$
is the total mass of suspended rods to describe an equilibrated flow.
For $\gamma = 0$, the coupled system for shear flow is given as 
\begin{equation}
\label{eqn:shear}
\begin{aligned}
\partial_tf+\partial_\theta(\partial_x w \cos^2\theta f)-\partial_x(\sin\theta\cos\theta f)&= D_r\partial_{\theta \theta}f, \\
Re\partial_t w&=\partial_{xx}w+\delta\left(\bar{\rho}-\displaystyle\int_0^{2\pi}fd\theta\right).
\end{aligned}
\end{equation}
For periodic boundary conditions the average density is constant in time, i.e.
\begin{equation*}
\bar{\rho}=\frac{1}{|\Omega|}\int_{\Omega}\int_0^{2\pi}f(t,x,\theta)d\theta=\frac{1}{|\Omega|}\int_{\Omega}\int_0^{2\pi}f(0,x,\theta)d\theta.
\end{equation*}

Note that the evolution of $f$ in (\ref{eqn:general}) and
  (\ref{eqn:shear}) is described by a conservation law. Furthermore,
  integration of $f$ over $S^1$ leads in both cases to a conservation
  law for the density as a function of space and time.
\subsection{Hyperbolic Moment System for Shear Flow}
As in Dahm and Helzel \cite{Dahm}, the dimension of the multi-scale model
(\ref{eqn:shear}) is reduced by replacing the distribution function
$f$ in the kinetic model (\ref{eqn:shear}) by a hierarchy of moment
equations.

Using the quantities 
\begin{equation}
\begin{aligned}
\rho({x},t)&\coloneqq\displaystyle\int_0^{2\pi}f({x},t,\theta)d\theta,\\
C_l({x},t)&\coloneqq\dfrac{1}{2}\displaystyle\int_0^{2\pi}\cos(2l\theta)f({x},t,\theta)d\theta, \quad l=1,2,...\\
S_l({x},t)&\coloneqq\dfrac{1}{2}\displaystyle\int_0^{2\pi}\sin(2l\theta)f({x},t,\theta)d\theta, \quad l=1,2,...\\
\end{aligned}
\label{eqn:moments}
\end{equation}
and setting $C_0=\frac{1}{2}\rho$, $S_0=0$, the infinite system of partial differential equations for shear flow is given as
\begin{equation*}
\begin{aligned}
\partial_t\rho &= \partial_xS_1, \\
\partial_tC_l &= \dfrac{1}{4}\partial_x(S_{l+1}-S_{l-1})-\dfrac{l}{2}\partial_x w (S_{l-1}+2S_l+S_{l+1})-4l^2D_rC_l, \quad l=1,2,... \\
\partial_tS_l &=\dfrac{1}{4}\partial_x\left(C_{l-1}-C_{l+1}\right)+\dfrac{l}{2}\partial_x w\left(C_{l-1}+2C_l+C_{l+1}\right)-4l^2D_rS_l, \quad l=1,2,...
\end{aligned}
\end{equation*}
The system is closed with $C_{N+1}=S_{N+1}=0$ which is based on the assumption that higher order moments decay faster than lower order moments as they correspond to a larger eigenvalue of the Laplace Operator on $S^1$. The closed moment system can be written in the form 
\begin{equation}
\label{eqn:inhomogen}
\partial_tQ(x,t)+A\partial_xQ(x,t)=\phi(Q(x,t)),
\end{equation}
where $Q(x,t)=(\rho,C_1,S_1,...,C_N,S_N)$ represents the vector of moments. The coefficient matrix $A\in \mathbb{R}^{(2N+1)\times(2N+1)}$ has the components
\begin{equation*}
\begin{aligned}
A_{1,3}:=-1, &\quad A_{3,1}:=-\frac{1}{8},\\
\left(\begin{matrix}
a_{2(N-j)-2,2(N-j)-2}&\cdots & a_{2(N-j)-2,2(N-j)+1}\\[0.5pt]
\vdots & &\vdots \\[0.5pt]
a_{2(N-j)+1,2(N-j)-2}&\cdots & a_{2(N-j)+1,2(N-j)+1}
\end{matrix}\right) &:=\left(\begin{matrix}
0 & 0 & 0 & -\frac{1}{4}\\[0.5pt]
0 & 0 & \frac{1}{4} & 0 \\[0.5pt]
0 & \frac{1}{4} & 0 & 0\\[0.5pt]
-\frac{1}{4} & 0 & 0 & 0 
\end{matrix}\right),\quad j=0,...,N-2.
\end{aligned}
\label{eqn:coefficient matrix}
\end{equation*}
Note that our definition of $A$ defines some of the components
twice. However, all those values are zero.
While the kinetic equation in (\ref{eqn:shear}) is a time-dependent partial differential equation in space and orientation, the system of moment equations (\ref{eqn:inhomogen}) depends only on space and time.
The moment system (\ref{eqn:inhomogen}) has to be considered with the diffusion equation
\begin{equation}
\label{eqn:diffusion}
Re\partial_tw=\partial_{xx}w+\delta \left(\bar{\rho}-\int_0^{2\pi}fd\theta \right).
\end{equation}
In \cite{Dahm}, we showed that the moment system (\ref{eqn:inhomogen}) is hyperbolic. Moreover, we showed that the update
\begin{equation} \label{eqn:source}
\partial_tQ(x,t)=\phi(Q(x,t))
\end{equation}
resulting from the source term of the moment system is equivalent with a spectral method for the drift diffusion equation 
\begin{equation}\label{eqn:driftDiffusion}
\partial_tf+\partial_\theta(\partial_x w \cos^2\theta f)=D_r\partial_{\theta \theta}f, 
\end{equation}
which is also described in more detail in \cite{Dahm} and
\cite{Helzel2}.

Note that the density is a conserved quantity of the moment system,
since the source term in (\ref{eqn:inhomogen}) only acts on the
higher order moments.
Furthermore,  note that the density distribution function $f(t,x,\theta)$ can
  be reconstructed from the moments using an expansion of the form
  \begin{equation}\label{eqn:expansionF}
f(t,x,\theta) = \frac{1}{2\pi} \rho(x,t) + \sum_{i=1}^\infty
 \left( \frac{2}{\pi} C_i(x,t) \cos(2 i \theta) + \frac{2}{\pi}
   S_i(x,t) \sin(2 i \theta) \right).  \end{equation}
In practical computations a finite number of moments will be used in
order to approximate $f$. 
  
\subsection{Hyperbolic Moment Systems for Two-Dimensional Flow}
\label{subsec:HypMomentSystem2d}
For two-dimensional flow, the infinite system of moment equations is given as 
 \begin{equation*}
\begin{aligned}
\partial_t \rho &= -u\partial_x\rho+\partial_xS_1-\left(w-\frac{3}{2}\right)\partial_z\rho-\partial_zC_1, \\
\partial_t C_l&= -u\partial_xC_l+\frac{1}{4}\partial_x(S_{l+1}-S_{l-1})-\partial_z\left(\frac{1}{4}C_{l-1}+\left(w-\frac{3}{2}\right)C_l+\frac{1}{4}C_{l+1}\right) \\
&-\frac{l}{2}(\partial_z w- \partial_x u)C_{l-1} +
\frac{l}{2}(\partial_z w- \partial_x u)C_{l+1}\\
&-\frac{l}{2}(\partial_z u+ \partial_x w)S_{l-1}+l(\partial_z
u-\partial_x w)S_l-\frac{l}{2}(\partial_z u+\partial_x w)S_{l+1}\\
&-4l^2D_rC_l, \quad l=1,...,N \\
\partial_t S_l&= -u\partial_xS_l+\frac{1}{4}\partial_x(C_{l-1}-C_{l+1})-\partial_z\left(\frac{1}{4}S_{l-1}+\left(w-\frac{3}{2}\right)C_l+\frac{1}{4}S_{l+1}\right) \\
&-\frac{l}{2}(\partial_z w- \partial_x
u)S_{l-1}+\frac{l}{2}(\partial_z w- \partial_x u)S_{l+1}\\
&+\frac{l}{2}(\partial_z u+ \partial_x w)C_{l-1}-l(\partial_z
u-\partial_x w)C_l+\frac{l}{2}(\partial_z u+\partial_x w)C_{l+1}\\
&-4l^2D_rS_l, \quad l=1,...,N. 
\end{aligned}
\label{eqn:moments2D}
\end{equation*}
Again, the system is closed with $C_{N+1}=S_{N+1}=0$. The moment equations can be rewritten in the form 
\begin{equation}
\partial_tQ(x,z,t)+A\partial_xQ+B\partial_zQ=\phi(Q). 
\label{eqn:system2d}
\end{equation}
The coefficient matrix $A\in \mathbb{R}^{(2N+1)\times (2N+1)}$ has the entries
\begin{equation*}
\begin{aligned}
 a_{1,3}&:=-1, \\
  a_{3,1}&:=-\frac{1}{8}, \\
a_{i,i}&:=u,&  i=1,..,2N+1, \\
\left(\begin{matrix}
a_{2(N-j)-2,2(N-j)-2} & \cdots &  a_{2(N-j)-2,2(N-j)+1}\\
\vdots & & \vdots \\
a_{2(N-j)+1,2(N-j)-2} & \cdots &  a_{2(N-j)+1,2(N-j)+1}
\end{matrix}\right)&:=&\left(\begin{matrix}
0 & 0 & 0 & -\frac{1}{4}\\
0 & 0 & \frac{1}{4} & 0\\
0 & \frac{1}{4} & 0 & 0 \\
-\frac{1}{4} & 0 & 0 & 0 
\end{matrix}\right), &j=0,...,N-2. \\
\end{aligned}
\label{eqn:A2d}
\end{equation*}
All other components of $A$ are equal to zero. The coefficient matrix $B\in \mathbb{R}^{(2N+1)\times (2N+1)}$ has the form
\begin{equation*}
\begin{aligned}
b_{1,2}&:= 1, \quad b_{2,1}:=\frac{1}{8},& \\
b_{j,j+2}&:= \frac{1}{4}, \quad b_{j+2,j}:=\frac{1}{4},& \quad j=2,...,2N-1, \\
b_{j,j}&:= w-\frac{3}{2}, &\quad j=1,...,2N+1.
\end{aligned}
\label{eqn:B2d}
\end{equation*}
All other components of $B$ are equal to zero. The two-dimensional moment system (\ref{eqn:system2d}) has to be considered with the flow equation 
\begin{equation}
\begin{aligned}
Re\left(\partial_tu+u\partial_xu+w\partial_zu\right)+\partial_x p&=u_{xx}+u_{zz},\\
Re\left(\partial_tw+u\partial_xw+w\partial_zw\right)+\partial_z p&=w_{xx}+w_{zz}-\delta\displaystyle\int_0^{2\pi}f d\theta,\\
\partial_x u+\partial_z w&=0. 
\end{aligned}
\label{eqn:Navier}
\end{equation}
In \cite{Dahm}, we showed that also the moment system (\ref{eqn:system2d}) is hyperbolic. 

\section{Estimating modelling errors}
\label{sec:estimate}
Our goal is to control the difference between the solution to the $(2N+1)$-moments system and the solution to the $(2N+3)$ moments system based on information that can be computed from the solution of the $(2N+1)$-moments system.
We will prove such an estimate in the case of one space dimension and  periodic boundary conditions.
The rationale is that we plan to solve the $(2N+1)$-moments system numerically and would like to assert whether its solution also provides a good approximation of the $(2N+3)$ moments system. In particular, we avoid dependence of constants in our estimate on the $(2N+3)$-moments solution.
We denote the flat, $1$-dimensional torus by $\mathbb{T}$.

For any $N \in \mathbb{N}$ the $(2N+1)$-moments system is endowed with an
energy, energy flux pair. Indeed, let\linebreak
$Q=(\rho, C_1, S_1, .., C_N, S_N)$  we define entropy $\eta$ and entropy flux $q$ by
\begin{gather} \eta(Q) := \frac12 C_0^2 + \sum_{i=1}^N C_i^2 + S_i^2  ,\\
   q(Q) :=  \frac14 S_1 C_0 + \frac14 \sum_{i=2}^N (C_i S_{i-1} + C_{i-1} S_i)
\end{gather}

\begin{lemma}
Any solution $(Q,w)$ of the $2N+1$-moments system  with
\[Q \in L^\infty(0,T;L^2(\mathbb{T})), w \in L^\infty(0,T;H^1(\mathbb{T})) \cap L^2(0,T; H^2(\mathbb{T})) \]
  satisfies:
\begin{multline*} \partial_t \big(\eta(Q)+ \frac{Re}2 (\partial_x w)^2\big) + \partial_x \big( q(Q)-\partial_x w \partial_{xx}w\big) 
 =\\
 \partial_x w \frac12 \Big(\sum_{\ell=1}^{N-1}C_\ell S_{\ell+1} - S_\ell C_{\ell+1}  
  \Big)
 - 4 D_r \sum_{\ell=1}^N \ell^2 (C_\ell^2 + S_\ell^2) - ( \partial_{xx} w)^2 - \delta \partial_x w \partial_x \rho
\end{multline*}
in the sense of distributions.
\end{lemma}
This can be seen by smoothing the solution $Q$ in space and time.

In order, to bound the difference between the solutions to the
$(2N+3)$ and $(2N+1)$
moments systems, we will need the following 
 generalised Gronwall lemma:
\begin{proposition} \cite[Prop 6.2, Generalised Gronwall lemma]{Bartels} \label{prop:GeneralizedGronwall}
    Suppose that the nonnegative functions $y_1 \in C([0, T])$, $y_2, y_3 \in L^1([0, T]), a \in L^{\infty}([0, T])$, and the real number $A \geq 0$ satisfy
    \begin{equation*}
    y_1\left(T^{\prime}\right)+\int_0^{T^{\prime}} y_2(t) \mathrm{d} t \leq A+\int_0^{T^{\prime}} a(t) y_1(t) \mathrm{d} t+\int_0^{T^{\prime}} y_3(t) \mathrm{d} t
    \end{equation*}
    for all $T^{\prime} \in[0, T]$. Assume that for $B \geq 0, \beta>0$, and every $T^{\prime} \in[0, T]$, we have
    \begin{equation*}
    \int_0^{T^{\prime}} y_3(t) \mathrm{d} t \leq B\left(\sup _{t \in\left[0, T^{\prime}\right]} y_1^\beta(t)\right) \int_0^{T^{\prime}}\left(y_1(t)+y_2(t)\right) \mathrm{d} t .
    \end{equation*}
    Set $E=\exp \left(\int_0^T a(t) \mathrm{d} t\right)$. Provided $8 A E \leq(8 B(1+T) E)^{-1 / \beta}$ holds, then
    \begin{equation*}
        \sup _{t \in[0, T]} y_1(t)+\int_0^T y_2(t) \mathrm{d} t \leq 8 A E.
    \end{equation*}
\end{proposition}

Note that if $\hat \rho, \hat C_1, \hat S_1, .., \hat C_N, \hat S_N$
solve the $(2N+1)$-moments
system and we insert \\$\hat Q:= (\hat \rho, \hat C_1, \hat S_1, ..,
\hat C_N, \hat S_N, 0, 0 )^T$
into the $(2N+3)$-moments system then the evolution equation for $w$ and all but the last two evolution equations for the $C_i, S_i$ are satisfied.
In the evolution equation for $C_{N+1}$ we have (due to $\hat C_{N+1}=\hat S_{N+1}=0$)
\[\partial_t 0 + \frac14 \partial_x\hat  S_N + \frac{N+1}{2} \partial_x \hat w (\hat S_N + 2\cdot 0) + 4(N+1)^2 D_r 0     \]
which is, in general, not zero.
Thus, we can understand $\hat Q$ as the solution of a perturbed $(2N+3)$-moments system with perturbations
\begin{eqnarray}
-\frac14 \partial_x \hat S_N - \frac{N+1}{2}   \hat S_N \partial_x \hat w&=:& \hat R_{2N+2},
\\
+\frac14 \partial_x \hat C_N + \frac{N+1}{2}   \hat C_N\partial_x \hat w &=:& \hat R_{2N+3},
\end{eqnarray}
in the evolution equation for $C_{N+1}, \, S_{N+1}$ respectively.

If we define the vector $\hat R:=(0,...,0,\hat R_{2N+2},\hat R_{2N+3}  )^T$, then the homogeneously extended $(2N+1)$-moments solution $\hat Q$ satisfies the perturbed $(2N+3)$-moments system:
\begin{eqnarray}\label{eq:perturbedm}
 \partial_t \hat Q + A\partial_x \hat Q &=& \hat \phi(\hat Q)  + \hat R\\
 \label{eq:perturbedv}
 \operatorname{Re} \partial_t \hat w &=& \partial_{xx} \hat w + \delta (\bar \rho - \hat \rho),
\end{eqnarray}
where we write $\hat \phi$ to emphasise the dependence on $\hat w$.

Now, we plan to show that the difference between $(\hat Q, \hat w)$
and the exact solution of the
$(2N+3)$ moments system $(Q, w)$ can be bounded in terms of $\hat R$ and norms of $(\hat Q, \hat w)$.  It turns out that this can only be done rigorously if $\hat R$ is small enough and certain norms of $\hat Q$ and $\partial_x \hat {w}$ are not too large. This is a reflection of the fact that the equations allow for the development of clusters that are associated with instabilities.

In the following, let $\eta$ be the entropy of the $(2N+3)$-moments
system  and note that $\eta'(Q)$ is given by\linebreak
$(C_0, 2C_1, 2S_1, \dots , 2C_{N+1}, 2S_{N+1}).$

\begin{theorem}\label{theorem:errorEstimate}
 Let $(\hat Q, \hat w)$ be a solution to \eqref{eq:perturbedm},
 \eqref{eq:perturbedv} with $\hat R \in L^2((0,T) \times \mathbb{T})$
 and let $(Q,w)$ solve
 \eqref{eqn:inhomogen}, \eqref{eqn:diffusion}.
 Assume that
 \begin{eqnarray*}
  Q &\in& L^\infty(0,T; L^2(\mathbb{T}))\\
  \hat Q &\in& L^\infty((0,T) \times \mathbb{T})\\
  w&\in& L^\infty(0,T;H^1(\mathbb{T})) \cap L^2(0,T; H^2(\mathbb{T}))\\
  \hat w &\in& L^\infty(0,T;W^{1,\infty}(\mathbb{T})) \cap L^2(0,T; H^2(\mathbb{T}))\\
 \end{eqnarray*}
Set 
\[ y_1 :=\int_{\mathbb{T}} \frac12 (C_0 - \hat C_0)^2 + \sum_{i=1}^{N+1}( (S_i - \hat S_i)^2+ (C_i -\hat C_i)^2 )  + \frac{Re}{2} (\partial_{x}w -\partial_{x} \hat w)^2 dx\]
then, provided 
\begin{multline}\label{eqn:inequTheorem}
 16^3 \left(y_1(0) + \int_0^{T} \| \hat R \|_{L^2(\mathbb{T})}^2 dt \right)\left( (1+T) \frac{1}{2\sqrt{D_r}} \right)^{2}\\
\leq  \exp\left(-3 \left(\int_0^{T} (\|\partial_x \hat w \|_{\infty}  +2\delta^2 + 1  2 \max_{ 1 \leq \ell \leq N} (\ell+1) ( \|\hat S_\ell\|_\infty + \| \hat C_\ell \|_\infty))   dt  \right)\right)
\end{multline}
is satisfied, the following estimate holds
\begin{multline}
\sup_{1 \leq t \leq T } y_1(t) + \int_0^T\frac12 \|\partial_{xx} w - \partial_{xx} \hat w\|_{L^2(\mathbb{T})}^2 
+  D_r \sum_{\ell=1}^{N+1} \ell^2 (\|S_\ell - \hat S_\ell\|_{L^2(\mathbb{T})}^2 + \|C_\ell - \hat C_\ell\|_{L^2(\mathbb{T})}^2) dt\\
\leq
8 \left(y_1(0) + \int_0^{T} \| \hat R \|_{L^2(\mathbb{T})}^2 dt \right)\\
 \times \exp\left(3 \left(\int_0^{T} \| \partial_x \hat w \|_\infty +2\delta^2 + 1  2 \max_{ 1 \leq \ell \leq N} (\ell+1) ( \|\hat S_\ell\|_\infty + \| \hat C_\ell \|_\infty)   dt  \right)\right)
\end{multline}
\end{theorem}

\begin{proof}
We observe that 
\[\eta(Q) - \eta(\hat Q) - \eta'(\hat Q)(Q- \hat Q) = \frac12 (C_0 - \hat C_0)^2 + \sum_{i=1}^{N+1}( (S_i - \hat S_i)^2+ (C_i -\hat C_i)^2 ) \]
and $\eta'(Q)-\eta'(\hat Q)= \eta'(Q-\hat Q)$.
Thus, we have, in an almost everywhere sense,
\begin{eqnarray*}
&&\frac{d}{dt}  \int_{\mathbb{T}} \frac12 (C_0 - \hat C_0)^2 + \sum_{i=1}^{N+1}( (S_i - \hat S_i)^2+ (C_i -\hat C_i)^2 )  + \frac{Re}{2} (\partial_x w -\partial_x \hat w)^2 dx \\
&=&\int_{\mathbb{T}}  (\eta'(Q) -\eta'(\hat Q))\cdot(\partial_t Q - \partial_t \hat Q)  + Re (\partial_x w -\partial_x \hat w) (\partial_{xt} w - \partial_{xt} \hat w) dx\\
&=& \int_{\mathbb{T}}  - (\eta'(Q) -\eta'(\hat Q))\cdot A(\partial_x Q - \partial_x \hat Q)  + (\eta'(Q) -\eta'(\hat Q))\cdot (\phi(Q)  - \hat \phi( \hat Q) ) \\
&& -  (\partial_{xx} w - \partial_{xx} \hat w)^2 + \delta (\partial_{xx} w- \partial_{xx} \hat w) ( \rho - \hat \rho) - (\eta'(Q) -\eta'(\hat Q)) \hat R dx\\
&=&\int_{\mathbb{T}} \sum_{\ell=1}^{N+1} \frac{-\ell}2 (C_\ell - \hat C_\ell) \left[  \partial_x w (S_{\ell-1} + 2 S_\ell + S_{\ell +1})
-\partial_x  \hat w (\hat S_{\ell-1} + 2 \hat S_\ell + \hat S_{\ell +1})\right]\\
&& + \sum_{\ell=1}^{N+1} \frac{\ell}2 (S_\ell - \hat S_\ell) \left[ \partial_x w (C_{\ell-1} + 2 C_\ell + C_{\ell +1})
- \partial_x \hat w (\hat C_{\ell-1} + 2 \hat C_\ell + \hat C_{\ell +1})\right]\\
&& -  (\partial_{xx} w - \partial_{xx} \hat w)^2 + \delta (\partial_{xx} w- \partial_{xx} \hat w) (\rho - \hat \rho) - (\eta'(Q) -\eta'(\hat Q)) \hat R \\
&& - D_r \sum_{\ell=1}^{N+1} \ell^2 ((S_\ell - \hat S_\ell)^2 + (C_\ell - \hat C_\ell)^2) dx \\
&=&\int_{\mathbb{T}} \partial_x\hat w  \sum_{\ell=1}^{N+1}  (S_\ell - \hat S_\ell)(C_{\ell+1} - \hat C_{\ell+1}   )- (C_\ell - \hat C_\ell)(S_{\ell+1} - \hat S_{\ell+1}   )\\
&& + \sum_{\ell=1}^{N+1} \frac{\ell}2 (S_\ell - \hat S_\ell)  (\partial_x w - \partial_x \hat w)  (C_{\ell-1} + 2 C_\ell + C_{\ell +1})
- (C_\ell - \hat C_\ell)  (\partial_x w -\partial_x \hat w)  (S_{\ell-1} + 2 S_\ell + S_{\ell +1})\\
&& -  (\partial_{xx} w - \partial_{xx} \hat w)^2 + \delta (\partial_{xx} w- \partial_{xx}\hat w) (\rho - \hat \rho)  - (\eta'(Q) -\eta'(\hat Q)) \hat R \\
&&- D_r \sum_{\ell=1}^{N+1} \ell^2 ((S_\ell - \hat S_\ell)^2 + (C_\ell - \hat C_\ell)^2) dx 
\end{eqnarray*}

Introducing the abbreviations
\begin{eqnarray*}
 y_1 &:=&\int_{\mathbb{T}} \frac12 (C_0 - \hat C_0)^2 + \sum_{i=1}^N( (S_i - \hat S_i)^2+ (C_i -\hat C_i)^2 )  + \frac{Re}{2} (\partial_{x}w -\partial_{x} \hat w)^2 dx,\\
 y_2 &:=& \int_{\mathbb{T}} (\partial_{xx} w - \partial_{xx} \hat w) ^2dx
 +  D_r \sum_{\ell=1}^{N+1} \ell^2 ((S_\ell - \hat S_\ell)^2 + (C_\ell - \hat C_\ell)^2) dx,\\
 y_3 &:=&\int_{\mathbb{T}}\partial_x \hat w  \sum_{\ell=1}^{N+1} \left(  (S_\ell - \hat S_\ell)
          (C_{\ell+1} - \hat C_{\ell+1}   ) -
 (C_\ell - \hat C_\ell)(S_{\ell+1} - \hat S_{\ell+1}   ) \right) dx,\\
 y_4 &:=&\int_{\mathbb{T}} \sum_{\ell=1}^{N+1} \Big( \frac{\ell}2 (S_\ell -
          \hat S_\ell)  (\partial_x w - \partial_x\hat w)  (\hat C_{\ell-1} + 2 \hat
          C_\ell + \hat C_{\ell +1}) \\
 && \qquad 
- (C_\ell - \hat C_\ell)  (\partial_x w - \partial_x\hat w)  (\hat S_{\ell-1} + 2 \hat
    S_\ell + \hat S_{\ell +1}) \Big) dx,
\end{eqnarray*}
\begin{eqnarray*}
y_5&:=&
 \int_{\mathbb{T}} \sum_{\ell=1}^{N+1} \Big( \frac{\ell}2 (S_\ell - \hat S_\ell)  (\partial_x w - \partial_x\hat w)  (C_{\ell-1} + 2 C_\ell + C_{\ell +1}
- \hat C_{\ell-1} - 2 \hat C_\ell - \hat C_{\ell +1})\\
&&\quad
- \frac{\ell}2(C_\ell - \hat C_\ell)  (\partial_x w -\partial_x \hat w)  (S_{\ell-1} +
   2 S_\ell + S_{\ell +1}- \hat S_{\ell-1} - 2 \hat S_\ell - \hat
   S_{\ell +1})\Big) dx, \\
y_6 &:=&  \int_{\mathbb{T}} \delta (\partial_{xx} w- \partial_{xx} \hat w) (\rho - \hat \rho) dx,\\
 y_7 &:=& - \int_{\mathbb{T}} (\eta'(Q) -\eta'(\hat Q)) \hat R dx.
\end{eqnarray*}

We can summarise our computation by
\begin{equation}\label{eq:sum}
 y_1' + y_2 \leq |y_3| + |y_4| + |y_5| + |y_6|+ |y_7|
\end{equation}
and,  in order to apply Proposition \ref{prop:GeneralizedGronwall}, we need to bound $|y_3|,\dots,|y_7|$ in terms of $y_1, y_2$ and norms of $(\hat Q, \hat w)$.

Young's inequality implies
\[ |y_3| \leq \|\partial_x \hat w \|_{\infty }  y_1,\]
and
\[|y_4| \leq \frac14 y_2 + 2 \max_{ 1 \leq \ell \leq N} (\ell+1) ( \|\hat S_\ell\|_\infty + \| \hat C_\ell \|_\infty) y_1. \]
We can use the embedding of $H^1(0,1)$ into $L^\infty(0,1)$ with Lipschitz constant $2$ to obtain
\begin{equation*}  |y_5|^2  \leq \|\partial_x (w - \hat w)\|_\infty^2 
   4\left( \sum_{\ell} \ell^2 (\|S_\ell - \hat S_\ell\|_{L^2}^2+\| C_\ell - \hat C_\ell\|_{L^2}^2) \right)
   \left( \sum_{\ell}  \|S_\ell - \hat S_\ell\|_{L^2}^2+\| C_\ell - \hat C_\ell\|_{L^2}^2 \right)
 \leq  \frac{1}{4 D_r } y_2^2 y_1.
\end{equation*}
Using Young's inequality twice more implies
\[|y_6| \leq \frac14 y_2 + 2\delta^2 y_1  \]
and
\[| y_7| \leq y_1 + \| \hat R \|_{L^2(\mathbb{T})}^2.  \]

Thus, inserting the inequalities that we just derived into \eqref{eq:sum} and integrating in time from $0$ to some $T'$ we obtain for any $0 \leq T' \leq T$

\begin{multline}
 y_1 (T') + \int_0^{T'}\frac12  y_2(t) dt \leq y_1(0) + \int_0^{T'} \| \hat R \|_{L^2(\mathbb{T})}^2 dt + \frac{1}{2\sqrt{D_r}} \sup_{ 0 \leq t \leq T'} \sqrt{y_1} \int_0^{T'} C_S   (y_1+y_2) dt \\
 + \int_0^{T'} (\|\partial_x   \hat w \|_{\infty}  +2\delta^2 + 1  2 \max_{ 1 \leq \ell \leq N} (\ell+1) ( \|\hat S_\ell\|_\infty + \| \hat C_\ell \|_\infty))  y_1  dt
\end{multline}
where we write $\| \cdot\|_\infty$ instead of $\| \cdot \|_{L^\infty(0,1)}$ for brevity.

Thus, by invoking \eqref{prop:GeneralizedGronwall}, we conclude that for any $T$ such that
\begin{multline}
 16^3 \left(y_1(0) + \int_0^{T} \| \hat R \|_{L^2(\mathbb{T})}^2 dt \right) ((1+T) \frac{1}{2\sqrt{D_r}} )^{2}\\
\leq  \exp\left(-3 \left(\int_0^{T} (\|\partial_x \hat w \|_{\infty}  +2\delta^2 + 1  2 \max_{ 1 \leq \ell \leq N} (\ell+1) ( \|\hat S_\ell\|_\infty + \| \hat C_\ell \|_\infty))   dt  \right)\right)
\end{multline}
the following bound for the difference between the solutions to both systems holds:

\begin{multline}
\sup_{0 \leq t \leq T } y_1(t) + \int_0^T \frac12 y_2(t) dt
\leq
8 \left(y_1(0) + \int_0^{T} \| \hat R \|_{L^2(\mathbb{T})}^2 dt \right)\\
 \times \exp\left(3 \left(\int_0^{T} (\| \partial_x \hat w \|_\infty +2\delta^2 + 1  2 \max_{ 1 \leq \ell \leq N} (\ell+1) ( \|\hat S_\ell\|_\infty + \| \hat C_\ell \|_\infty))   dt  \right)\right).
\end{multline}
\end{proof}

\section{Numerical Discretisation of the 1D and 2D Homogenous Moment System}
\label{sec:numerical}
In this section, we present a numerical discretisation of the
homogenous moment systems for shear flow (\ref{eqn:inhomogen}) and
two-dimensional flow
(\ref{eqn:system2d}) with $\phi(Q) = 0$.
The update described by the source term is in each grid cell
equivalent to a spectral method described in \cite{Helzel2}.
The source term will be added via a straight
forward splitting approach and will not be discussed further. 

As it could be shown in \cite{Dahm} that the one- and two-dimensional
moment systems are hyperbolic, they can be solved with the
high-resolution Wave Propagation Algorithm by LeVeque \cite{LeV}, a
finite volume method for hyperbolic problems.
We distinguish between a uniform approximation and a non-uniform
approximation,
which corresponds to a constant number of moment equations throughout
the domain or a varying number of moment equations. 
\subsection{Uniform Approximation}
For the uniform approximation, the number of moment equations is fixed
globally for the entire domain. Thus, we consider
\begin{equation}\label{eqn:homSystem}
\partial_t Q + A \partial_x Q = 0, \quad A \in
\mathbb{R}^{(2N+1)\times (2N+1)}
\end{equation}
where the matrix $A$ is diagonalisable with real eigenvalues. We use the
notation $A=R \Lambda R^{-1}$, where $\Lambda$ is the diagonal matrix
of eigenvalues $\lambda_1\le \ldots \le \lambda_{2N+1}$ of $A$ and $R$
is the matrix whose columns are the
corresponding linear independent eigenvectors $r_1,\ldots,r_{2N+1}$.
\subsubsection{Wave Propagation Algorithm for 1D Moment System}
\label{subsec:wave1d}
The spatial domain $\Omega:=[x_l,x_r]$ is discretised with an equidistant numerical grid $x_l=x_{\frac{1}{2}},\dots,x_{M+\frac{1}{2}}=x_r$ with grid cells
\begin{equation*}
C_i\coloneqq\left[x_{i-\frac{1}{2}},x_{i+\frac{1}{2}}\right], \quad i=1,\ldots,M,
\label{eqn:grid1d}
\end{equation*}
of length $\Delta x:=x_{i+\frac{1}{2}}-x_{i-\frac{1}{2}}$. For the discretisation of the time variable, we consider  $0=t^0<t^1<t^2<\dots$ and define the length of the time step as $\Delta t\coloneqq t^{n+1}-t^n$, $\forall n \in \mathbb{N}_0$. The discrete values of $Q(x,t)$ in (\ref{eqn:homSystem}) at time $t^n$ are stored at the midpoints of the grid cell, i.e.
\begin{equation}
\label{eqn:Q1d}
Q_i^n\approx\frac{1}{\Delta x}\int_{x_{i-\frac{1}{2}}}^{x_{i+\frac{1}{2}}}Q\left(x,t^n\right)dx\approx Q\left(x_i,t^n\right), \quad i=1,\ldots,M
\end{equation}
approximates the cell averages in cell $C_i$ at time $t^n$. For each
time step, the cell averages are updated with LeVeque's high-resolution Wave
Propagation Algorithm, which can be described in the general form 
\begin{equation}
\label{eqn:highres}
Q_{i}^{n+1}=Q_i-\frac{\Delta t}{\Delta x}\left(\mathcal{A}^+\Delta Q_{i-1/2}+\mathcal{A}^-\Delta Q_{i+1/2}\right)-\frac{\Delta t}{\Delta x}\left(\tilde{F}_{i+1/2}-\tilde{F}_{i-1/2}\right).
\end{equation}
The fluctuations $\mathcal{A}^\pm$ are defined as 
\begin{equation}
\label{eqn:fluc}
\mathcal{A}^\pm_{i+1/2}:=\sum_{p=1}^{2N+1}\left(\lambda^p\right)^{\pm}\mathcal{W}^p_{i+\frac{1}{2}},
\end{equation}
with waves $\mathcal{W}^p_{i-\frac{1}{2}} = \alpha_p r_p$.
The coefficients describe the eigenvector decomposition of the jump
in $Q$ at the grid cell interface, i.e.
$(\alpha_1,\ldots,\alpha_{2N+1})^T = R^{-1} (Q_i^n - Q_{i-1}^n)$.
The second-order correction terms $\tilde{F}$ are for all $i$ given as 
\begin{equation*}
\tilde{F}_{i-1/2}=\frac{1}{2}\displaystyle\sum_{p=1}^{2N+1}|\lambda^p|\left(1-\frac{\Delta t}{\Delta x}|\lambda^p|\right)\tilde{\mathcal{W}}_{i-1/2}^p.
\end{equation*}
The tilde indicates that
limited versions of the waves are
used to suppress unphysical oscillations near discontinuities or steep
gradients as described in detail in \cite{LeV}.

While the Wave Propagation Algorithm
  (\ref{eqn:highres}) is not written in flux difference form, one
  obtains a conservative method for the homogeneous hyperbolic system
  (\ref{eqn:homSystem}) if
  \begin{equation}\label{eqn:conservationCondition}
  A Q_i - A Q_{i-1} = {\cal A}^- \Delta Q_{i-\frac{1}{2}} + {\cal A}^+
  \Delta Q_{i-\frac{1}{2}}.
  \end{equation}
  Computing the waves by an eigenvector decomposition of
  $Q_i-Q_{i-1}$ using the eigenvectors of $A$ as outlined above
  provides a conservative method.

\subsubsection{Wave Propagation Algorithm for 2D Moment System}
\label{subsubsec:Wave2d}
For the approximation of the two-dimensional homogenous moment system
\begin{equation}\label{eqn:homSystem2d}
  \partial_t Q + A \partial_x Q + B \partial_y Q = 0,
\end{equation}
with $A, B \in \mathbb{R}^{(2N+1)\times(2N+1)}$ as described in
\autoref{subsec:HypMomentSystem2d},
we assume that the velocity field $(u(t,x,z),0,w(t,x,z))^T$
is constant in time over a time step and externally imposed.
The two-dimensional spatial domain $\Omega:=[x_l,x_r]\times [z_l,z_r] $ is discretised on an equidistant numerical grid with grid cells
\begin{equation*}
C_{i,j}\coloneqq\left[x_{i-\frac{1}{2}},x_{i+\frac{1}{2}}\right] \times\left[z_{j-\frac{1}{2}},z_{j+\frac{1}{2}}\right] , \quad i=1,\ldots,M, j=1,\ldots,N,
\label{eqn:gridcell}
\end{equation*}
of length $\Delta x:=x_{i+\frac{1}{2}}-x_{i-\frac{1}{2}}$ and $\Delta z:=z_{j+\frac{1}{2}}-z_{j-\frac{1}{2}}$. 
The average value of $Q(x,z,t)$ over the $(i,j)$-th grid cell at time $t^n$ 
\begin{equation}
Q_{i,j}^n\approx \frac{1}{\Delta x\Delta z}\int_{z_{j-1/2}}^{z_{j+1/2}}\int_{x_{i-1/2}}^{x_{i+1/2}}Q\left(x,z,t^n\right)dxdz=\frac{1}{\Delta x\Delta z}\int_{C_{i,j}}Q\left(x,z,t^n\right)dxdz,
\label{eqn:Q2d}
\end{equation}
is updated with a method of the form 
\begin{equation*}
\begin{aligned}
Q_{i,j}^{n+1}=Q_{i,j}^n&-\dfrac{\Delta t}{\Delta x}\left(\mathcal{A}^+\Delta Q_{i-1/2,j}+\mathcal{A}^-\Delta Q_{i+1/2,j}\right) -\dfrac{\Delta t}{\Delta z}\left(\mathcal{B}^+\Delta Q_{i,j-1/2}+\mathcal{B}^-\Delta Q_{i,j+1/2}\right) \\
&-\dfrac{\Delta t}{\Delta x}\left(\tilde{F}_{i+1/2,j}-\tilde{F}_{i-1/2,j}\right)-\dfrac{\Delta t}{\Delta z}\left(\tilde{G}_{i,j+1/2}-\tilde{G}_{i,j-1/2}\right), 
\end{aligned}
\end{equation*}
$\mathcal{A}^\pm$ and $\mathcal{B}^\pm$ are the fluctuations resulting
from solving Riemann problems in the $x$- and $z$-direction. The
fluxes $\tilde{F}$ and $\tilde{G}$ perform second order corrections.
The details can again be found in \cite{LeV}.
\subsection{Non-Uniform Approximation with Interface Coupling}
For the non-uniform approximation, the number of moment equations in the moment system is adjusted adaptively. Depending on locally varying flow structures or accuracy requirements, the domain of interest is divided into intervals in which moment systems with different numbers of moment equations are considered. This leads to cell interfaces at which moment systems with different numbers of moment equations have to be coupled. 
\subsubsection{Generalised Riemann Problems for Moment Systems with
  Different Resolution}
\label{subsec:GeneralisedRP}
At interfaces between two cells in which moment systems with different
resolutions are used, generalised Riemann problems of the following
form are considered
 \begin{equation}
\label{eqn:generaliesedRP}
\begin{array}{rlrrrrrll}
\partial_tQ^{2N+1}+A^{2N+1}\partial_xQ^{2N+1}&=&0, &&&& \partial_tQ^{2M+1}+A^{2M+1}\partial_xQ^{2M+1}&=&0, \\[8pt]
Q^{2N+1}(x,0)&=&Q^{2N+1}_L, &&&& Q^{2M+1}(x,0)&=&Q^{2M+1}_R, \\[8pt]
x<0, &&&&&& x>0,
\end{array}
\end{equation}
where $A^{2N+1}\in \mathbb{R}^{(2N+1)\times (2N+1)}$ and $A^{2M+1}\in
\mathbb{R}^{(2M+1)\times (2M+1)}$ correspond to the coefficient matrix
(\ref{eqn:coefficient matrix}) of the one dimensional moment system,
$M \neq N$ and $Q^{2N+1}\in \mathbb{R}^{2N+1}$, $Q^{2M+1}\in
\mathbb{R}^{2M+1}$.
Without loss of generality we assume $N \le M$.
The change of the number of moment equations
leads to a change of
the eigenvalues and eigenvectors of the matrix $A$ and thus influences
the waves.



We approximate a solution of (\ref{eqn:generaliesedRP}) by a
piecewise constant function of the form
\begin{equation}
\begin{aligned}
\label{eqn:generalsolution}
Q(x,t)&=\begin{cases}Q_L^{2N+1}, \hspace{3mm} x-\lambda_1^{2N+1}t<0,
  \\[8pt]
Q_L^{2N+1} + \sum_{p=1}^i {\cal W}^{p,2N+1}, \hspace{3mm}    x \in
(\lambda_i^{2N+1}t,\lambda_{i+1}^{2N+1}t), \hspace{3mm} i=1,\dots, N,
\\[8pt]
Q_R^{2M+1} + \sum_{p=2M+1}^k {\cal W}^{p,2M+1}, \hspace{3mm}    x \in
(\lambda_{k-1}^{2M+1}t,\lambda_{k}^{2M+1}t), \hspace{3mm} k=2M+1,\dots, M+2,
\\[8pt]
Q_R^{2M+1},\hspace{3mm} x -\lambda_{2M+1}^{2M+1}t>0.
\end{cases}
\end{aligned}
\end{equation}
Here $\lambda_i^{2N+1}$, $i=1, \ldots, N$ are the negative
eigenvalues of the coefficient matrix
$A^{2N+1}$ and $\lambda_j^{2M+1}$, $j=M+2, \ldots, 2M+1$  are the
positive eigenvalues of $A^{2M+1}$.
Note that  $\lambda_{N+1}^{2N+1}=\lambda_{M+1}^{2M+1}=0$. Thus, for $x<0$
the piecewise constant solution, with $2N+1$ components, is computed by adding the left moving
waves ${\cal W}^{p,2N+1}=\alpha_p r_p^{N+1}$, $p=1, \ldots, N$, with
$\alpha = (R^{2N+1})^{-1} (Q_R^{2N+1}-Q_L^{2N+1})$ and eigenvectors
of $A^{2N+1}$ which correspond to negative eigenvalues, to the left
initial state $Q_L^{2N+1}$.
Here $Q_R^{2N+1}$ consists of the first $2N+1$ components of $Q_R^{2M+1}$.
Analogously,
the piecewise constant solution for $x>0$ is computed by adding the
right moving waves  
${\cal W}^{p,2M+1} = \alpha_p r_p^{2M+1}$, $p=2M+1,\ldots, M+2$,
with 
$\alpha = (R^{2M+1})^{-1}(Q_R^{2M+1} - Q_L^{2M+1})$, 
to the initial state $Q_R^{2M+1}$. Now the left state $Q_L^{2M+1} \in
\mathbb{R}^{2M+1}$  is obtained from
$Q_L^{2N+1}$ by adding zeros at the components $2N+2, \ldots, 2M+1$.

In order to visualise the solution of the generalised Riemann problem
for the homogeneous system of moment equations, we start with
piecewise constant initial values which are obtained from 
steady state solutions of (\ref{eqn:driftDiffusion}) with constant
externally imposed velocity gradient $\partial_x w$ and initial values
$f(0,\theta) = 1/2\pi$ using the spectral method from \cite{Helzel2}.
The spectral method for (\ref{eqn:driftDiffusion}) is based on an
expansion of $f$ of the form (\ref{eqn:expansionF}) but with a finite
number of moments. Thus, the spectral method directly provides the
initial values for the moments.

\begin{example}
\label{exa:Interface1}
We consider the generalised Riemann problem (\ref{eqn:generaliesedRP})
for different values of $N$ and $M$. For $x<0$ we use $w_x/D_r=1$, for
$x>0$ we use $w_x/D_r=4$ in order to compute the initial values for
the moments using a spectral method for the computation of steady
states of (\ref{eqn:driftDiffusion}).

As a reference solution for the generalised Riemann problem we compute the solution of the detailed model 
\begin{equation}
\partial_tf(x,t,\theta)+\partial_x(-\cos\theta\sin\theta f)=0
\label{eqn:fvergleich}
\end{equation}
using the steady state solutions of (\ref{eqn:driftDiffusion})
as initial values in $f$.
We compute the numerical solution $f$ of (\ref{eqn:fvergleich}) at
time $t=5$ using the two-dimensional Wave Propagation Algorithm
adapted to this scalar transport equation. We then numerically
integrate this solution over $\theta$ to compute the reference
solution $\rho(x,t)$.
\end{example}
\begin{figure}[H]%
  \centering
    \scalebox{1.5}{
\begin{tikzpicture}[decoration=brace]
	\draw[] (0,0) -- (5,0) node[right] {};
	
        \draw[] (2,0.0) -- (0.8,2) node[below]{}; 
	\draw[] (2,-0.2) -- (2,2) node[below]{};
	\draw[] (2,0.0) -- (3.2,2) node[below]{};
	\draw[] (2,0.0) -- (4.4,2) node[below]{};
	
	\node at (2,-0.5) {$\scriptscriptstyle 0$};
	\node at (0.6,1) {$\scriptscriptstyle Q_L$};
	\node at (0.8,2.2) {$\scriptscriptstyle x=\lambda_1^3t$};
	\node at (1.5,1.6) {$\scriptscriptstyle \tilde{Q}_1$};
	\node at (2,2.2) {$\scriptscriptstyle  x=\lambda_3^5t$};
	\node at (2.5,1.6) {$\scriptscriptstyle \tilde{Q}_2$};
	\node at (3.4,2.2) {$\scriptscriptstyle x=\lambda_4^5t$};
	\node at (3.5,1.6) {$\scriptscriptstyle \tilde{Q}_3$};
	\node at (4.2,2.2) {$\scriptscriptstyle x=\lambda_5^5t$};
	\node at (4.2,1) {$\scriptscriptstyle Q_R$};
\end{tikzpicture}}
  \caption{Structure of the solution $Q$ of the generalised Riemann Problem (\ref{eqn:generaliesedRP}) for $N=1$ and $M=2$. The value of $Q$ is constant in each wedge of the $x-t$ plane. }.%
  \label{fig:Godunov}
\end{figure}
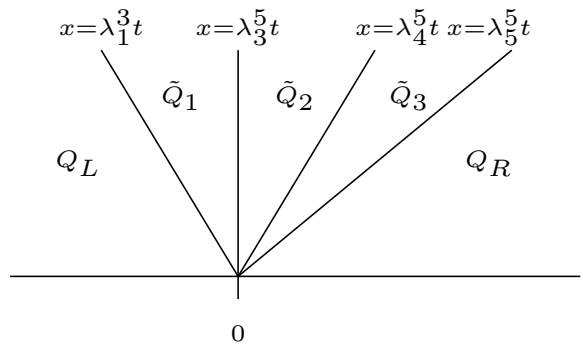
In \autoref{fig:Godunov}, the solution of the generalised Riemann
Problem (\ref{eqn:generaliesedRP}) is
visualised in the $x$-$t$ plane for $N=1$ and $M=2$. For $x<0$ we
consider the coefficient matrix $A^3$ and for $x\ge 0$ the matrix
$A^5$ to compute the waves.
While the negative eigenvalue $\lambda_1^{3}$ of $A^{3}$ gives
the wave speed of the left-going wave $\mathcal{W}^{1,3}$, the
positive eigenvalues $\lambda_4^{5}$ and $\lambda_5^5$ of
$A^{5}$ describe the wave speeds of the right-going waves.
Both matrices $A^3$ and $A^5$ have the eigenvalue
$\lambda=0$ as centered eigenvalue. The jump in $Q$ across this
centered wave follows from $\tilde{Q}_1$ and $\tilde{Q}_2$.
\begin{figure}[H]%
  \centering
\includegraphics[width=0.5\linewidth]{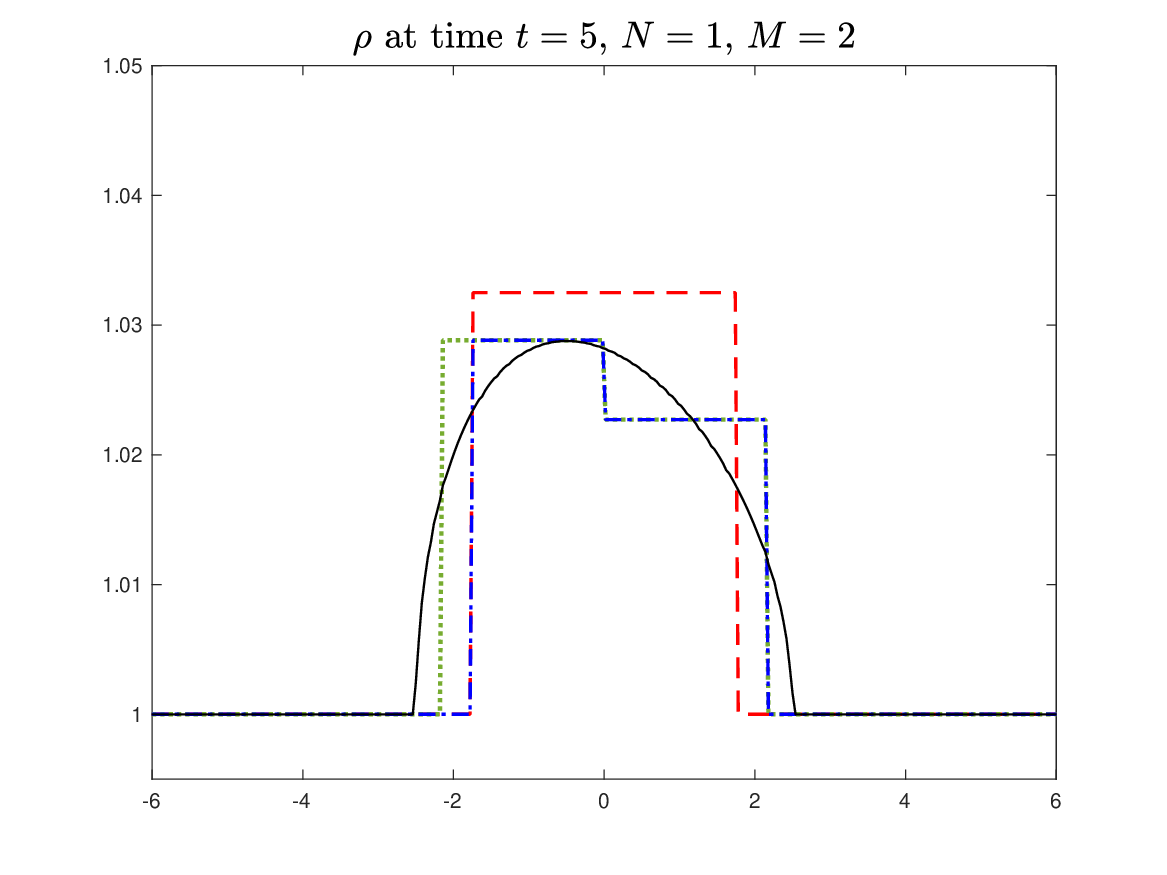}
     \caption{Solution $\rho$ of the generalised Riemann problem (\ref{eqn:generaliesedRP}) described in \cref{exa:Interface1} at time $t=5$. For $x<0$ we used $w_x/D_r=1$, for $x>0$ we used $w_x/D_r=4$. The blue dashed dotted curve is the solution for $N=1$ and $M=2$. The red dashed curve uses $N=M=1$ and the green dotted curve uses $N=M=2$ moment equations. The black solid curve is a highly resolved reference solution.} 
     \label{fig:generalisedRPN=1M=2}
\end{figure}
In \autoref{fig:generalisedRPN=1M=2}, the constructed solution of the
generalised Riemann problem in \cref{exa:Interface1} is visualised at
time $t=5$. The blue dashed-dotted curve shows the first component
$\rho$ of the solution vector $Q$ of the generalised Riemann problem
described
in \cref{exa:Interface1} coupling moment systems of order $N=1$ and
$M=2$. The red dashed curve uses $2N+1=2M+1=3$ moment equations
throughout the domain. The green dotted curve uses $2N+1=2M+1=5$
moment equations throughout the domain. The black solid curve is a
highly resolved reference solution. Note that $\rho$ is a Riemann
invariant of the centered wave for the $3 \times 3$ moment system and
of the second and fourth wave of the $5 \times 5$ moment system.
The jump in the moments at $x=0$ is neither an eigenvector of $A^3$ nor of $A^5$ but
instead follows from the coupling of the two different solutions.
\begin{figure}[H]%
\captionsetup[subfloat]{labelformat=empty}
  \centering
 \subfloat[][]{\includegraphics[width=0.3\linewidth]{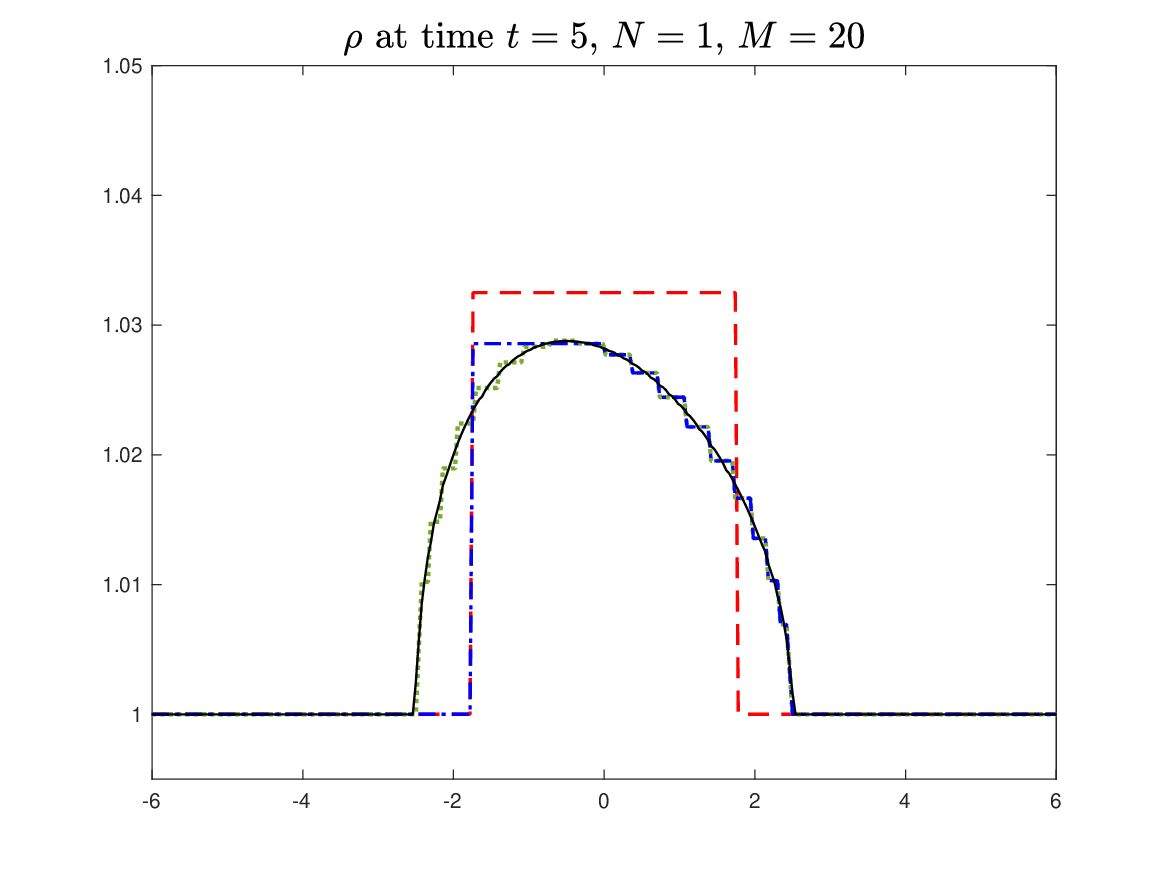}}
 \quad
 \subfloat[][]{\includegraphics[width=0.3\linewidth]{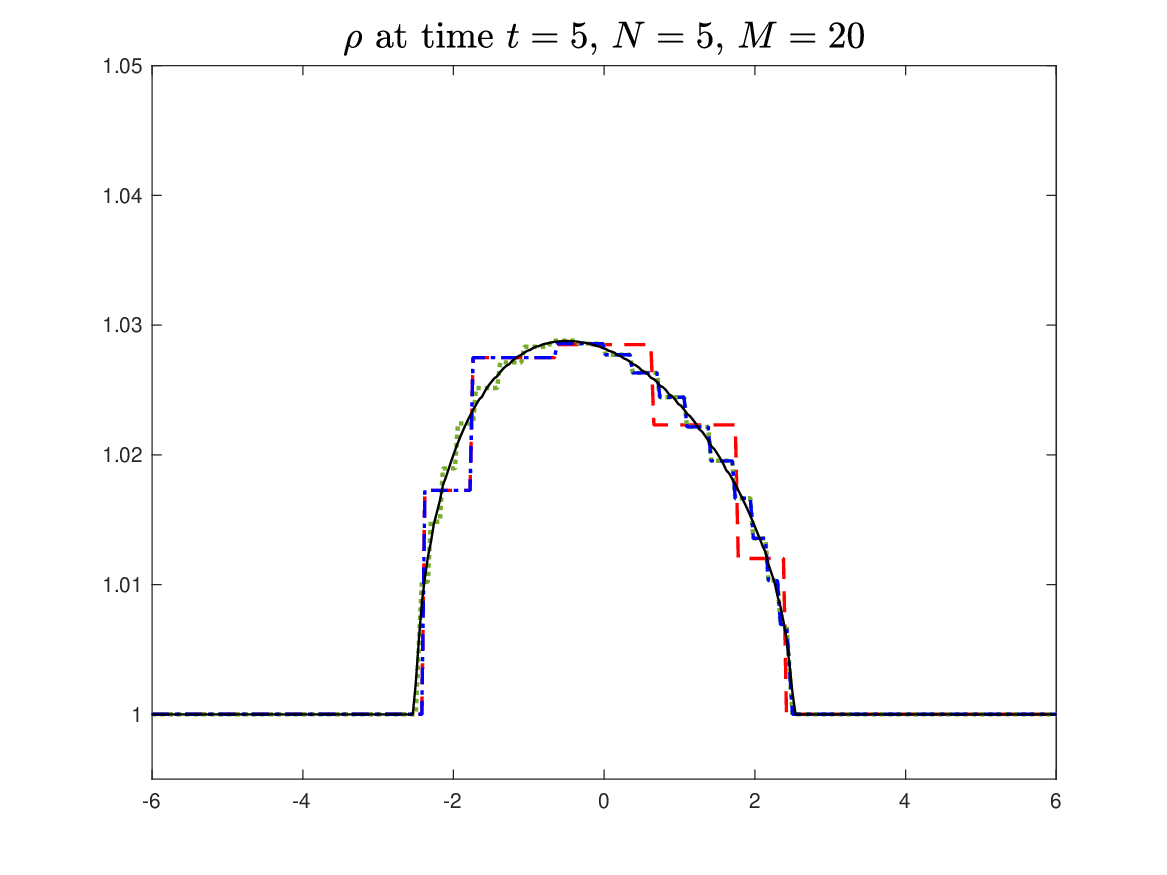}} 
 \quad \subfloat[][]{\includegraphics[width=0.3\linewidth]{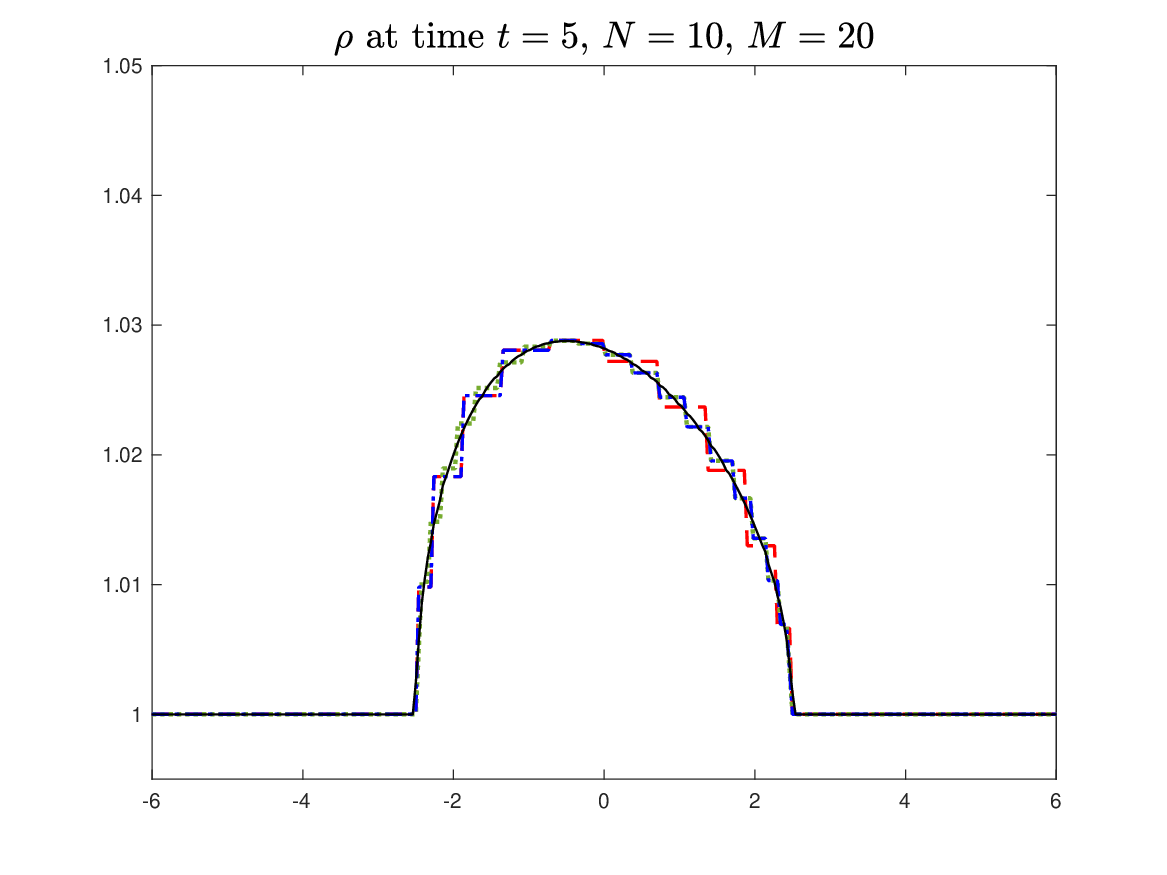}} 
 \caption{Solution $\rho$ of the generalised Riemann problem described in \cref{exa:Interface1} at time $t=5$. For $x<0$ we use $w_x/D_r=1$, for $x>0$ we use $w_x/D_r=4$. The blue dashed dotted curve uses different values of $N$ on the left hand side of the interface and $M=20$ on the right hand side of the interface. This solution is compared with a rough solution (red dashed curve) and a detailed solution (green dotted curve) which use $2N+1=2M+1$ moment equations throughout the domain. The black solid curve is a highly resolved reference solution. 
 }%
 \label{fig:InterfaceRP}
\end{figure}
\autoref{fig:InterfaceRP} shows that a spatial coupling of moment systems with different resolution leads to an accurate approximation of the reference solution once the resolution of the moment systems is large enough on both sides of the interface. The red dashed curve is a rough solution using moment systems of order $N=1$, $N=5$, and $N=10$ throughout the domain. For $N=1$ and $M=20$, the solution in $\rho$ (blue dashed dotted curve) roughly approximates the highly resolved reference solution (black solid curve). At lower computational costs, the approximation using $N=10$ and $M=20$ compares well with the detailed solution using $N=M=20$ (green dotted curve). 
\begin{figure}[H]%
\captionsetup[subfloat]{labelformat=empty}
  \centering
 \subfloat[][]{\includegraphics[width=0.3\linewidth]{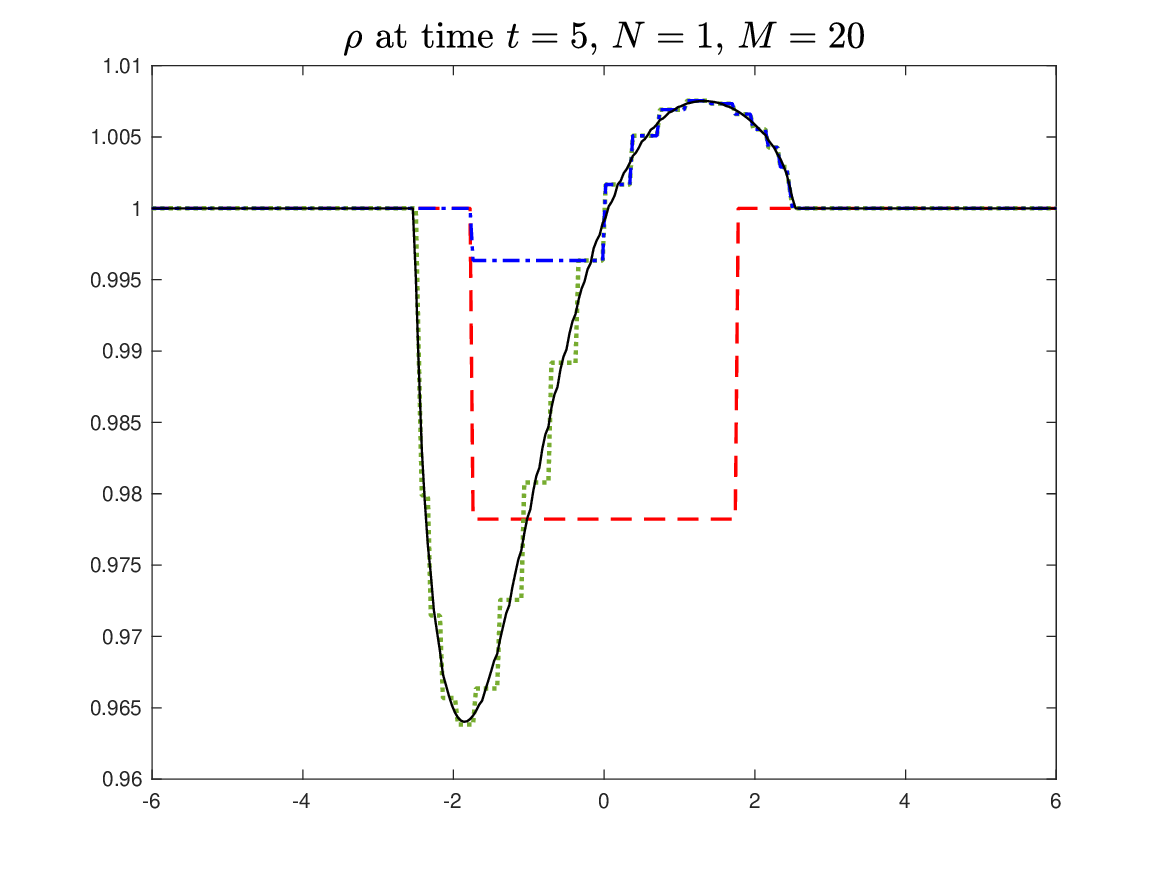}}
 \quad
 \subfloat[][]{\includegraphics[width=0.3\linewidth]{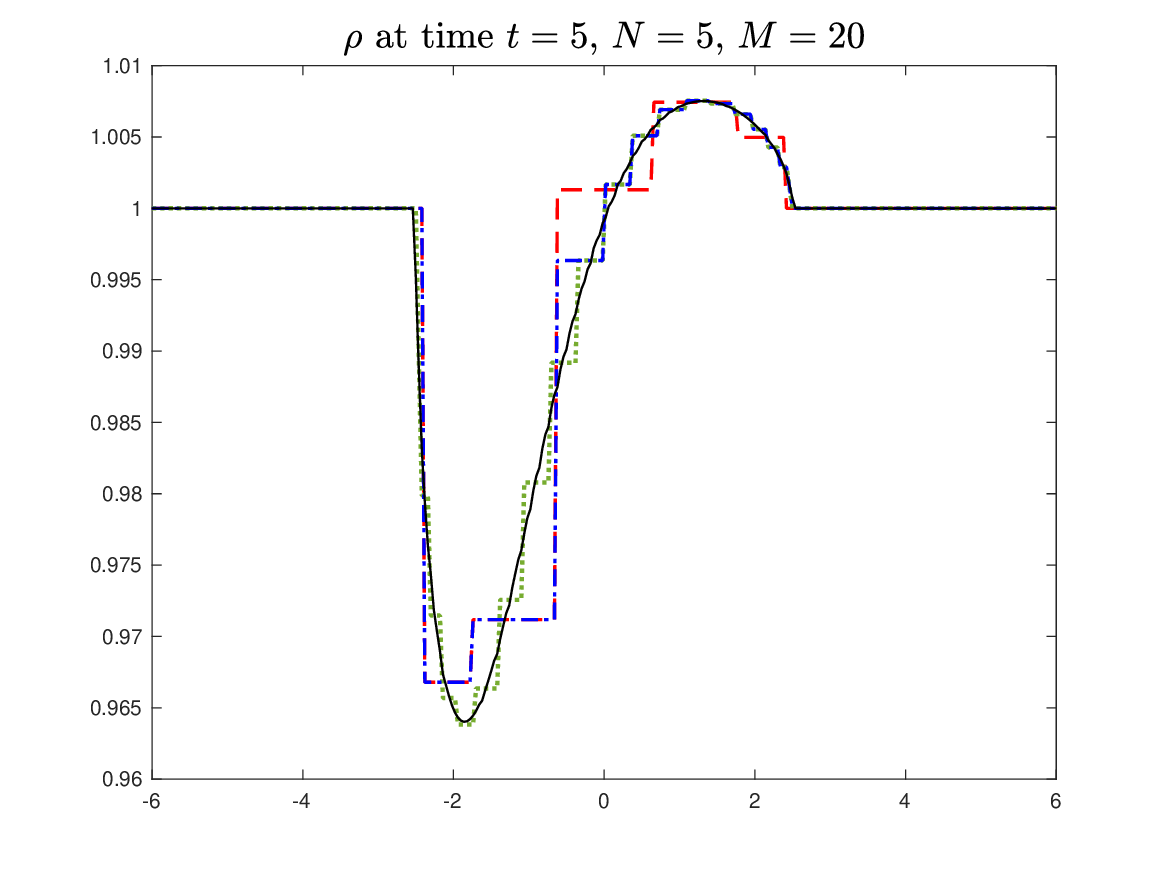}} 
 \quad
 \subfloat[][]{\includegraphics[width=0.3\linewidth]{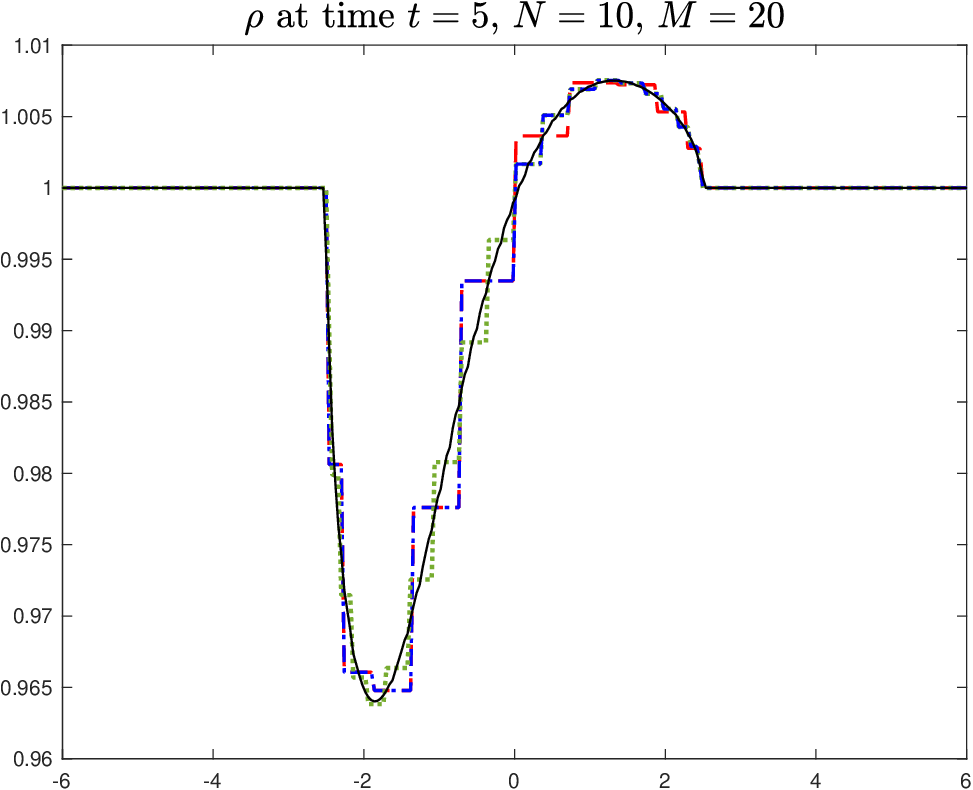}}
 \caption{Solution $\rho$ of the generalised Riemann problem described in \cref{exa:Interface1} at time $t=5$. For $x<0$ we used $w_x/D_r=10$, for $x>0$ we used $w_x/D_r=40$. The blue dashed dotted curve uses different values of $N$ on the left hand side of the interface and $M=20$ on the right hand side of the interface. This solution is compared with a rough solution (red dashed curve) and a detailed solution (green dotted curve) which use $2N+1=2M+1$ moment equations throughout the domain. The black solid curve is a highly resolved reference solution. 
 }%
 \label{fig:InterfaceRP2}
\end{figure}
In \autoref{fig:InterfaceRP2}, we consider an analogous test problem
as in \cref{exa:Interface1} but we used $w_x/D_r=10$ for $x<0$ and
$w_x/D_r=40$ for $x>0$ for the computation of the piecewise constant
initial values. Using $N=10$ and $M=20$ produces accurate
results. In the next section we will describe a wave propagation
algorithms which uses this wave decomposition of the generalised
Riemann problem. 
\subsubsection{Conservative Wave Propagation Algorithm for 1D Moment System with Different Resolution}
\label{subsubsec:conservativeWave}
We now develop a Wave Propagation Algorithm for the moment system with
different resolution. At grid cell interfaces with different 
numbers of moment equation the wave decomposition used in the numerical
method is based on the generalised Riemann problem
(\ref{eqn:generaliesedRP}) discussed in the
previous section.
As the flux function of the moment system on
the left hand side of the interface differs from the flux function on
the right hand side, fluctuations defined in analogy to the standard
form (\ref{eqn:fluc}) do not lead to a conservative method. In fact,
the condition (\ref{eqn:conservationCondition}), which guaranties
conservation in the standard case, is not even well defined if vectors
and matrices with different dimension are used across a grid cell
interface where the number of moment equations changes.

To derive a conservative Wave Propagation Algorithm for solving moment
systems with different numbers of moment equations, we use an
alternative procedure to define the fluctuations 
$\mathcal{A}^{\pm}\Delta Q_{i-\frac{1}{2}}$.  For homogeneous
linear hyperbolic systems (\ref{eqn:homSystem}) the fluctuations can
alternatively be defined using
\begin{equation}
\label{eqn:fluccon}
\mathcal{A}^+\Delta Q_{i-\frac{1}{2}}=A Q_i - A Q^*_{i-\frac{1}{2}},
\quad
\mathcal{A}^-\Delta Q_{i-\frac{1}{2}}=A Q^*_{i-\frac{1}{2}} - A Q_{i-1}.
\end{equation}
Here $Q_{i-\frac{1}{2}}^*$ is the solution of the Riemann problems
with piecewise constant initial values given by $Q_{i-1}$ and $Q_i$ at
the interface $x_{i-\frac{1}{2}}$.
This value can be computed using
\begin{equation}
\label{eqn:formel1}
Q^*_{i-\frac{1}{2}}=Q_{i-1}+\sum_{p:\lambda^{p}<0}\mathcal{W}^{p}_{i-1/2}
\end{equation}
or alternatively by using 
\begin{equation}
\label{eqn:formel2}
Q^*_{i-\frac{1}{2}}=Q_{i}-\sum_{p:\lambda^{p}>0}\mathcal{W}^{p}_{i-1/2}. 
\end{equation}
The stationary wave of the homogeneous Riemann problem with
constant number of moment equations can be ignored since
${\cal W}^{N+1,2N+1}$ is an eigenvector with eigenvalue zero and
therefore this wave does not contribute to the flux $A Q^*$ at the interface.

If the same number of moment equations is used in adjacent cells, we have the same flux function on both sides of the interface $x=x_{i-\frac{1}{2}}$. For $2N+1$ moment equations, formula (\ref{eqn:fluccon}) leads to 
\begin{equation*}
\begin{aligned}
\mathcal{A}^+\Delta
Q_{i-\frac{1}{2}}&=A^{2N+1}Q_i^{2N+1}-A^{2N+1}Q_i^{2N+1}
+\sum_{p=1}^{2N+1}\left(\lambda^{p,2N+1}\right)^+\mathcal{W}^{p,2N+1}_{i-\frac{1}{2}} \\
&= \sum_{p=1}^{2N+1}\left(\lambda^{p,2N+1}\right)^+\mathcal{W}^{p,2N+1}_{i-1/2},
\end{aligned}
\end{equation*}
\begin{equation*}
\begin{aligned}
\mathcal{A}^-\Delta Q_{i-\frac{1}{2}}&=A^{2N+1}Q_{i-1}^{2N+1}+\sum_{p=1}^{2N+1}\left(\lambda^{p,2N+1}\right)^-\mathcal{W}^{p,2N+1}_{i-\frac{1}{2}}-A^{2N+1}Q_{i-1}^{2N+1} \\
&=\sum_{p=1}^{2N+1}\left(\lambda^{p,2N+1}\right)^-\mathcal{W}^{p,2N+1}_{i-1/2}.
\end{aligned}
\end{equation*}
Analogously for adjacent cells in which $2M+1$ moments are used. 

At the interface between cells in which different numbers of moment equations are used, we solve Riemann problems between states $Q_{i-1}^{2N+1}$ and $Q_i^{2M+1}$ (or $Q_{i-1}^{2M+1}$ and $Q_i^{2N+1}$). To construct a conservative method, we assign both states at the interface to the flux function of the moment system of higher order. Again, we assume $2M+1>2N+1$. We extend the vector 
$Q_{i-1}^{2N+1}\in \mathbb{R}^{2N+1}$ to a vector of length $2M+1$ by adding $2M-2N$ zeros, i.e., we define 
\begin{equation*}
\tilde{Q}_{i-1}^{2N+1}:=\begin{pmatrix}
 Q_{i-1}^{2N+1} \\
0 \\
\vdots \\
0 
\end{pmatrix}
\in \mathbb{R}^{2M+1}.
\end{equation*} 
Moreover, we use (\ref{eqn:formel2}) for both fluctuations
$\mathcal{A}^{\pm}\Delta Q_{i-\frac{1}{2}}$, i.e.\ we set
$$Q_{i-\frac{1}{2}}^* = Q_{i}^{2M+1}-\sum_{p:\lambda^{p,2M+1}>0} {\cal W}_{i-\frac{1}{2}}^{p,2M+1}.$$  Then, the fluctuations are given as
\begin{align*}
\mathcal{A}^+\Delta Q_{i-\frac{1}{2}}&=A^{2M+1}Q_i^{2M+1}-A^{2M+1}Q^*_{i-\frac{1}{2}}  \nonumber\\ 
&=A^{2M+1}Q_i^{2M+1}-A^{2M+1}Q_i^{2M+1}+\sum_{p=1}^{2M+1}\left(\lambda^{p,2M+1}\right)^+\mathcal{W}^{p,2M+1}_{i-\frac{1}{2}}\nonumber \\
&=\sum_{p=1}^{2M+1}\left(\lambda^{p,2M+1}\right)^+\mathcal{W}^{p,2M+1}_{i-\frac{1}{2}}, \\
  \mathcal{A}^-\Delta Q_{i-\frac{1}{2}}&=A^{2M+1} Q^*_{i-\frac{1}{2}}
                                         -A^{2M+1} \tilde{Q}_{i-1}^{2N+1} \nonumber \\
&=A^{2M+1}Q_i^{2M+1}-\sum_{p=1}^{2M+1}\left(\lambda^{p,2M+1}\right)^+\mathcal{W}^{p,2M+1}_{i-\frac{1}{2}}-A^{2M+1}\tilde{Q}_{i-1}^{2N+1} \nonumber \\
&=A^{2M+1}\left(Q_i^{2M+1}-\tilde{Q}_{i-1}^{2N+1}\right)-\sum_{p=1}^{2M+1}\left(\lambda^{p,2M+1}\right)^+\mathcal{W}^{p,2M+1}_{i-\frac{1}{2}}. \nonumber
\end{align*}
Since ${\cal A}^- \Delta Q_{i-\frac{1}{2}}$ is the fluctuation due to
the left moving waves, this term updates the cell average values in
cell $(i-1)$ where we only use $2N+1$ moment equations. Therefore, we
only use the first $2N+1$ components of ${\cal A}^- \Delta
Q_{i-\frac{1}{2}}$ to update the cell averages in cell $(i-1)$. For
the second order correction terms at the interface $i-\frac{1}{2}$ we
use the $N+M+1$ waves and corresponding wave speeds discussed in
\autoref{subsec:GeneralisedRP}, i.e.\ we compute
$$
\tilde{F}_{i-\frac{1}{2}} = \frac{1}{2} \sum_{p=1}^N
|\lambda^{p,2N+1}|\left( 1 - \frac{\Delta t}{\Delta x}
  |\lambda^{p,2N+1}| \right) \tilde{\cal W}_{i-\frac{1}{2}}^{p,2N+1}
+
\frac{1}{2} \sum_{p=M+2}^{2M+1}
|\lambda^{p,2M+1}|\left( 1 - \frac{\Delta t}{\Delta x}
  |\lambda^{p,2M+1}| \right) \tilde{\cal W}_{i-\frac{1}{2}}^{p,2M+1}.
$$
To obtain vectors of the same length, we add zeros as components
$2N+2, \ldots, 2M+1$ to $\tilde{\cal W}_{i-\frac{1}{2}}^{p,2N+1}$ and
only use the first $2N+1$ components of the correction flux for the
update of the moments in cell $i-1$ but the whole vector for the
update in cell $i$.
The wave limiter described in \cite{LeV} limits waves based on a
comparison with neighbouring waves of the same family, i.e.\
neighbouring waves which
correspond to the same eigenvector are compared. In order to apply limiting for the
waves at the interface $i-\frac{1}{2}$ one needs to compute two
additional wave decompositions at the interfaces $i-\frac{3}{2}$ and
$i+\frac{1}{2}$. 

We summarise our results in the following theorem. 
\begin{theorem}
Let 
\begin{equation}
\begin{aligned}
\label{eqn:flucconneu1}
\mathcal{A}^+\Delta Q_{i-\frac{1}{2}}&=\sum_{p=1}^{2M+1}\left(\lambda^{p,2M+1}\right)^+\mathcal{W}^{p,2M+1}_{i-\frac{1}{2}},  \\
\mathcal{A}^-\Delta Q_{i-\frac{1}{2}}&=A^{2M+1}\left(Q_i^{2M+1}-\tilde{Q}_{i-1}^{2N+1}\right)-\sum_{p=1}^{2M+1}\left(\lambda^{p,2M+1}\right)^+\mathcal{W}^{p,2M+1}_{i-\frac{1}{2}}
\end{aligned}
\end{equation}
at interfaces between cells in which moment systems with $2N+1$ and $2M+1$, $M>N$, moment equations are used and 
\begin{equation}
\begin{aligned}
\label{eqn:fluct1}
\mathcal{A}^+\Delta Q_{i-\frac{1}{2}}&=\sum_{p=1}^{2N+1}\left(\lambda^{p,2N+1}\right)^+\mathcal{W}^{p,2N+1}_{i-1/2},  \\
\mathcal{A}^-\Delta Q_{i-\frac{1}{2}}&=\sum_{p=1}^{2N+1}\left(\lambda^{p,2N+1}\right)^-\mathcal{W}^{p,2N+1}_{i-1/2}
\end{aligned}
\end{equation}
at interfaces between cells in which moment systems with $2N+1=2M+1$ moment equations are used. Then the high-resolution Wave Propagation Algorithm 
\begin{equation}
\label{eqn:high}
Q_{i}^{n+1}=Q_{i}-\frac{\Delta t}{\Delta x}\left(\mathcal{A}^{+} \Delta Q_{i-1 / 2}+\mathcal{A}^{-} \Delta Q_{i+1 / 2}\right)-\frac{\Delta t}{\Delta x}\left(\widetilde{F}_{i+1 / 2}-\widetilde{F}_{i-1 / 2}\right)
\end{equation}
is a conservative method in the first $2N+1$ components for solving moment systems with different resolution in different spatial regions of the domain. 
\end{theorem}
\begin{proof}
At interfaces between cells in which moment systems with $2N+1$ and $2M+1$, $M>N$ are used, we have
\begin{equation*}
\begin{aligned}
&\mathcal{A}^-\Delta Q_{i-\frac{1}{2}}+\mathcal{A}^+\Delta Q_{i-\frac{1}{2}} \\
&=\sum_{p=1}^{2M+1}\left(\lambda^{p,2M+1}\right)^+\mathcal{W}^{p,2M+1}_{i-\frac{1}{2}}+A^{2M+1}\left(Q_i^{2M+1}-\tilde{Q}_{i-1}^{2N+1}\right)-\sum_{p=1}^{2M+1}\left(\lambda^{p,2M+1}\right)^+\mathcal{W}^{p,2M+1}_{i-\frac{1}{2}} \\
&= A^{2M+1}Q_i^{2M+1}-A^{2M+1}\tilde{Q}_{i-1}^{2N+1} \\
&= A^{2M+1} Q_i^{2M+1}- A^{2M+1} \tilde{Q}_{i-1}^{2N+1}.
\end{aligned}
\end{equation*}
At interfaces between cells in which moment systems with $2N+1=2M+1$ moment equations are used, the fluctuations (\ref{eqn:fluct1}) are defined in the standard form and obviously fulfil the conservation condition. As the second order correction terms are defined in flux difference form, (\ref{eqn:high}) leads to a conservative update. 
\end{proof}
The form of our moment equations (\ref{eqn:inhomogen})
shows that only $\rho$ is a conserved quantity. Our approximation of
the homogeneous moment system obtained by ignoring the source term
conserves the minimal number of moments used anywhere in the computational
domain by defining a unique flux at each grid cell interface. After
applying the source term update to the moment system only $\rho$ will
be conserved.  

\begin{figure}[H] 
     \centering
\begin{tikzpicture}[decoration=brace]
	\draw[] (0,0) -- (14.5,0) node[right] {};

	 \draw[] (2,0.0) -- (0.8,2) node[below]{}; 
	\draw[dashed] (2,-0.2) -- (2,2.2) node[below]{};
	\draw[] (2,0.0) -- (3.2,2) node[below]{};
	
        \draw[] (7,0.0) -- (5.8,2) node[below]{}; 
	\draw[dashed] (7,-0.2) -- (7,2.2) node[below]{};
	\draw[] (7,0.0) -- (8.2,2) node[below]{};
	\draw[] (7,0.0) -- (8.9,2) node[below]{};
	
	\draw[] (12,0.0) -- (10.1,2) node[below]{}; 
	\draw[] (12,0.0) -- (10.8,2) node[below]{}; 
	\draw[dashed] (12,-0.2) -- (12,2.2) node[below]{};
	\draw[] (12,0.0) -- (13.2,2) node[below]{};
	\draw[] (12,0.0) -- (13.9,2) node[below]{};
	
	\node at (2,-0.5) {$\scriptscriptstyle x_{i-\frac{3}{2}}$};
	\node at (7,-0.5) {$\scriptscriptstyle x_{i-\frac{1}{2}}$};
	\node at (7,-1.0) {$\scriptscriptstyle 0$};
	\node at (12,-0.5) {$\scriptscriptstyle x_{i+\frac{1}{2}}$};
	\node at (0.5,-0.5) {$\scriptscriptstyle Q_{i-2}$};
	\node at (4.5,-0.5) {$\scriptscriptstyle Q_{i-1}$};
	\node at (9.5,-0.5) {$\scriptscriptstyle Q_{i}$};
	\node at (13.5,-0.5) {$\scriptscriptstyle Q_{i+1}$};
		
	\node at (1.5,1.7) {$\scriptscriptstyle\mathcal{W}_{i-1/2}^{1,3}$};
	\node at (2.5,1.7) {$\scriptscriptstyle\mathcal{W}_{i-1/2}^{3,3}$};
	
	\node at (6.5,1.7) {$\scriptscriptstyle\mathcal{W}_{i-1/2}^{1,3}$};
	\node at (7.5,1.7) {$\scriptscriptstyle\mathcal{W}_{i-1/2}^{4,5}$};
	\node at (8.8,1.2) {$\scriptscriptstyle\mathcal{W}_{i+1/2}^{5,5}$};
	
	\node at (10.4,1.2) {$\scriptscriptstyle\mathcal{W}_{i-1/2}^{1,5}$};
	\node at (11.5,1.7) {$\scriptscriptstyle\mathcal{W}_{i+1/2}^{2,5}$};
	\node at (12.5,1.7) {$\scriptscriptstyle\mathcal{W}_{i+1/2}^{4,5}$};
	\node at (13.8,1.2) {$\scriptscriptstyle\mathcal{W}_{i+1/2}^{5,5}$};
	
\end{tikzpicture}
     \caption{Schematic diagram of Godunov's method for solving the generalised Riemann Problem (\ref{eqn:generaliesedRP}) for $N=1$ and $M=2$. The Riemann problem is solved at each cell interface.} 
     \label{fig:generalisedRP}
\end{figure}
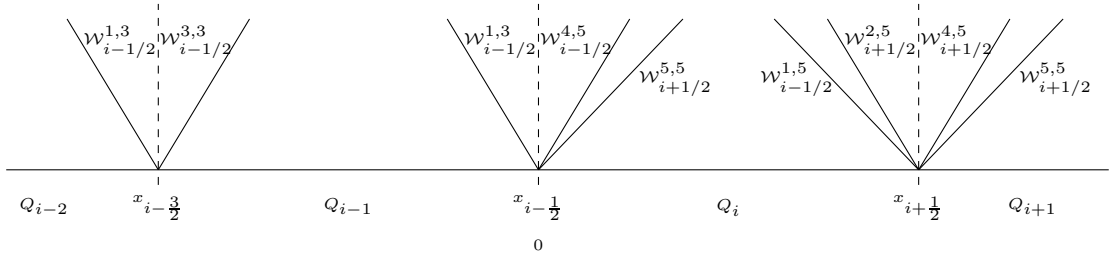
\autoref{fig:generalisedRP} gives a schematic diagram of Godunov's
method for solving the generalised Riemann problem
(\ref{eqn:generaliesedRP}) for $N=1$ and $M=2$. At the interface
$x_{i-\frac{3}{2}}$, the Riemann problem between the states
$Q_{i-2}^3$ and $Q_{i-1}^3$ has to be solved. The flux function at
this interface is given as $A^3Q_{i-\frac{3}{2}}^*$. Analogously, the
flux function at the interface $x_{i+\frac{1}{2}}$ is given as
$A^5Q_{i-\frac{1}{2}}^*$. At the interface $x_{i-\frac{1}{2}}$, three
moment equations are used on the left and five moment equations on the
right hand side of the interface. To get a method which is
conservative in the first three components, we choose the coefficient
matrix $A^5$ to compute the numerical flux function. 
\section{Bulk-Coupling of Moment Equations with Flow Equations}	
\label{sec:Bulk}
In this section, we study the numerical discretisation of the one- and two-dimensional hyperbolic moment systems coupled to the diffusion equation (\ref{eqn:diffusion}) or the two-dimensional Navier-Stokes equation (\ref{eqn:Navier}). 
\subsection{Bulk-Coupling for Shear Flow}	
\label{subsec:shear}
We consider the one-dimensional moment system (\ref{eqn:inhomogen}) coupled to (\ref{eqn:diffusion}). We discretise the spatial domain $\Omega=[x_l,x_r]$ and the time variable $t$ in the same way as described in \autoref{subsec:wave1d}. We define the discrete values of the velocity $w(x,t)$ and of the vector of moments $Q(x,t)$ on a staggered grid as visualised in \autoref{fig:grid}.
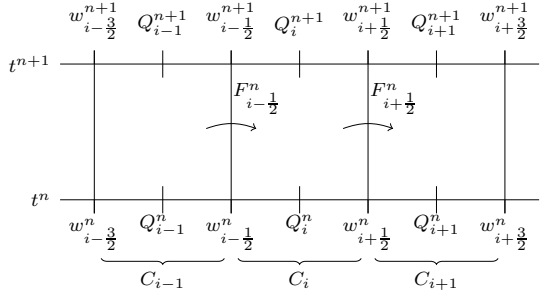
\begin{figure}[H]%
  \centering
  \scalebox{.9}{
\begin{tikzpicture}[decoration=brace]
	\draw[] (8.5,0) -- (1.5,0) node[left] {$\scriptstyle t^n$};
         \draw[] (8.5,2) -- (1.5,2) node[left] {$\scriptstyle t^{n+1}$};
	\draw[] (2,-0.2) -- (2,2.2) node[below]{};
	\draw[] (4,-0.2) -- (4,2.2) node[below]{};
	\draw[] (6,-0.2) -- (6,2.2) node[below]{};
	\draw[] (8,-0.2) -- (8,2.2) node[below]{};
\node (A) at (3.5,1) {}; \node (B) at (4.5,1) {};
\draw[->] (A) to[bend right=-20] node[above] {$\quad\quad  \scriptstyle F^{n}_{i-\frac{1}{2}}$} (B);
\node (C) at (5.5,1) {}; \node (D) at (6.5,1) {};
\draw[->] (C) to[bend right=-20] node[above] {$\quad\quad  \scriptstyle F^{n}_{i+\frac{1}{2}}$} (D);
	
	\draw (2,-.2) -- (2,0) node[below=4pt] {$\scriptstyle w_{i-\frac{3}{2}}^n$};
	\draw (4,-.2) -- (4,0) node[below=4pt] {$\scriptstyle w_{i-\frac{1}{2}}^n$};
	\draw[decorate, yshift=-4ex] (3.9,-.2) -- node[below=0.4ex] {$\scriptstyle C_{i-1}$} (2.1,-0.2);
	\draw (6,-.2) -- (6,0) node[below=4pt] {$\scriptstyle w_{i+\frac{1}{2}}^n$};
	\draw[decorate, yshift=-4ex] (5.9,-0.2) -- node[below=0.4ex] {$\scriptstyle C_i$} (4.1,-0.2);
	\draw (8,-.2) -- (8,0) node[below=4pt] {$\scriptstyle w_{i+\frac{3}{2}}^n$};
	\draw[decorate, yshift=-4ex] (7.9,-0.2) -- node[below=0.4ex] {$\scriptstyle C_{i+1}$} (6.1,-0.2);
	
	\draw (3,-.2) -- (3,0.2) node[below=8pt] {$\scriptstyle Q_{i-1}^n$};
	\draw (5,-.2) -- (5,0.2) node[below=8pt] {$\scriptstyle Q_{i}^n$};
	\draw (7,-.2) -- (7,0.2) node[below=8pt] {$\scriptstyle Q_{i+1}^n$};
	
	\draw (2,2) -- (2,2.2) node[above=1pt] {$\scriptstyle w_{i-\frac{3}{2}}^{n+1}$};
	
	\draw (4,2) -- (4,2.2) node[above=1pt] {$\scriptstyle w_{i-\frac{1}{2}}^{n+1}$};
	\draw (6,2) -- (6,2.2) node[above=1pt] {$\scriptstyle w_{i+\frac{1}{2}}^{n+1}$};
	\draw (8,2) -- (8,2.2) node[above=1pt] {$\scriptstyle w_{i+\frac{3}{2}}^{n+1}$};
	
	\draw (3,1.8) -- (3,2.2) node[above=2pt] {$\scriptstyle Q_{i-1}^{n+1}$};
	\draw (5,1.8) -- (5,2.2) node[above=2pt] {$\scriptstyle Q_{i}^{n+1}$};
	\draw (7,1.8) -- (7,2.2) node[above=2pt] {$\scriptstyle Q_{i+1}^{n+1}$};
\end{tikzpicture}}
\caption{Illustration of the staggered grid used for the discretisation of the coupled moment system for shear flow.}
\label{fig:grid}
\end{figure}
The discrete value of the velocity at time $t^n$ is stored at the nodes of the grid, i.e.
\begin{equation*}
w^n_{i+\frac{1}{2}}\approx w\left(x_{i+\frac{1}{2}},t^n\right), \quad i=1,\ldots,m+1
\end{equation*}
approximates the point value on the interface $x_{i+\frac{1}{2}}$ at
time $t^n$. The discrete value of the moments at time $t^n$ is stored
at the midpoints of the grid cell, see (\ref{eqn:Q1d}).

We compute the numerical solution of the coupled moment system for shear flow with an operator splitting method in which we separately approximate the different components of the coupled moment system. We use Strang splitting for solving the inhomogeneous diffusion equation as well as for solving the inhomogeneous system of moment equations. The steps of the algorithm for solving the coupled moment system for one time step are presented in \Cref{shearalgo}. The approach is comparable to the splitting method presented by Cheng and Knorr \cite{Cheng} for the Vlasov-Poisson equation. While in \cite{Cheng} a Poisson equation is considered, we have an inhomogeneous diffusion equation. 

\begin{algorithm}[H]
\caption{Operator splitting algorithm for solving the coupled moment system for shear flow.}\label{shearalgo}
\begin{itemize}
  \item[1.] $\frac{1}{2}\Delta t $ step on \hspace{4mm} $\partial_tQ(x,t)=\phi(Q(x,t))$. 
  \item[2.]   $\frac{1}{4}\Delta t $ step on \hspace{4mm} $\partial_t w(x,t)=\delta(\bar{\rho}-\rho)$. \hspace{4.5cm} \raisebox{1mm}{$\tikzmark{listing-2-end}$}
  
  \item[3.]   $\frac{1}{2}\Delta t $ step on \hspace{2mm} $\partial_t
    w(x,t)
    -\partial_{xx}w(x,t)=0$. Calculate $\partial_x w(x,t)$.
  \item[4.]    $\frac{1}{4}\Delta t $ step on \hspace{4mm} $\partial_t
    w(x,t)=\delta(\bar{\rho}-\rho)$. \hspace{4.5cm}\raisebox{-1mm}{$\tikzmark{listing-4-end}$}
  
   \item[5.]  $\Delta t $ step on \hspace{6mm} $\partial_tQ(x,t)+A\partial_xQ(x,t)=0$. 
   \item[6.]  $\frac{1}{4}\Delta t $ step on \hspace{4mm} $\partial_t w(x,t)=\delta(\bar{\rho}-\rho)$. \hspace{4.5cm} \raisebox{1mm}{$\tikzmark{listing-6-end}$}
    \item[7.] $\frac{1}{2}\Delta t $ step on \hspace{4mm} $\partial_t
      w(x,t)
      -\partial_{xx}w(x,t)=0$. Calculate $\partial_x w(x,t)$. 
   \item[8.]  $\frac{1}{4}\Delta t $ step on \hspace{4mm} $\partial_t w(x,t)=\delta(\bar{\rho}-\rho)$. \hspace{4.5cm} \raisebox{-1mm}{$\tikzmark{listing-8-end}$}
   \item[9.]  $\frac{1}{2}\Delta t $ step on \hspace{4mm} $\partial_tQ(x,t)=\phi(Q(x,t))$. 
\end{itemize}

\AddNote{listing-2-end}{listing-4-end}{listing-2-end}{ first half time step of Strang splitting for the flow equation}
\AddNote{listing-6-end}{listing-8-end}{listing-6-end}{ second half time step of Strang splitting for the flow equation}
\end{algorithm} 

The system of ordinary differential equations resulting from the source term of the moment system is solved with the classical Runge-Kutta method. For each time step, the update of the discrete velocity field $w(x,t)$ is computed with the Crank-Nicolson method for periodic solutions. This solution is used to calculate
\begin{equation*}
\partial_x w \left(x_i,t^n\right)=\frac{w_{i+\frac{1}{2}}^n-w_{i-\frac{1}{2}}^n}{\Delta x}, \quad i=1,\ldots,m.
\end{equation*}
We calculate the solution of the homogeneous system of moment equations with the high-resolution Wave Propagation Algorithm by LeVeque, described in \autoref{subsec:wave1d}. We use the test case that was already considered in \cite{Dahm} to study the accuracy of this approach.
\begin{example}
\label{exa:accuracy}
We study the one-dimensional moment system (\ref{eqn:inhomogen}) coupled to (\ref{eqn:diffusion}) with initial data on the interval $[0,100]$ of the form 
\begin{equation*}
\begin{aligned}
\rho(x,0)&=\exp\left(-(x-50)^2\right), \\
C_i(x,0)&=S_i(x,0)=0, \quad i=1,...,N, \\
w(x,0)&=0
\end{aligned}
\end{equation*}
and periodic boundary conditions. The parameters are set to
$D_r=0.01$, $Re = 1$ and $\delta=1$. We compute the solution of $\rho$ at time $t=30$.  
\end{example}
In \Cref{EOC1.1}, we present a convergence study for the problem in
\cref{exa:accuracy} for different values of $N$. As there is no
analytical solution for the coupled problem for shear flow, we use a
highly resolved numerical solution of the coupled problem calculated
on a very fine grid with 8192 grid cells as a reference solution. We
compare the highly resolved solution of $\rho$ with the numerical
solution of $\rho$ on coarse grids for different values of $N$. As the
grids are chosen in the way that all grid points on coarser grids are
also grid points on the fine grid, we can compare the numerical
solutions of $\rho$ on coincident grid points. In the first test, the highly resolved solution and the coarse solution use the same number of moment equations. We show the $L_\infty$-error and the experimental order of convergence (EOC), computed by comparing the error on two different grids
\begin{equation*}
EOC=\frac{\log\left(\|\rho_{n}-\rho_{n}^{ex}\|_\infty\Big/\|\rho_{2n}-\rho_{2n}^{ex}\|_\infty\right)}{\log(2)}. 
\end{equation*}
 $\rho_{n}$ denotes the numerical solution computed on a coarse grid
 with $n$ grid cells in $x$ at time $t=30$. $\rho_{n}^{ex}$ is the
 reference solution which is computed on a fine grid in $x$ at time
 $t=30$ and subsequently projected onto the grid with $n$ cells. In all computations for \Cref{EOC1.1}, we discretised the coupled problem for shear flow with the methods presented in \Cref{shearalgo}. The results in \Cref{EOC1.1} confirm second order convergence rates. 
 
 \begin{table}[H]
\centering
\begin{tabular}{|l|ll|ll|ll|ll|}
\hline
                       &         N=1              &                      &         N=2              &  &   N=3               
             &    &   N=10 &   \\\hline
\multicolumn{1}{|l|}{grid} & \multicolumn{1}{l|}{$L_{\infty}$-Error} &EOC & \multicolumn{1}{l|}{$L_{\infty}$-Error} &\multicolumn{1}{l|}{EOC}& \multicolumn{1}{l|}{$L_{\infty}$-Error} &\multicolumn{1}{l|}{EOC} & \multicolumn{1}{l|}{$L_{\infty}$-Error} &\multicolumn{1}{l|}{EOC}\\ \hline

\multicolumn{1}{|l|}{512} & \multicolumn{1}{l|}{$5.8162\cdot 10^{-3}$} & \multicolumn{1}{l|}{} & \multicolumn{1}{l|}{$2.1098\cdot 10^{-3}$}& \multicolumn{1}{l|}{} & \multicolumn{1}{l|}{$8.4068\cdot 10^{-4}$}  &&\multicolumn{1}{l|}{$1.0134\cdot 10^{-3}$}& \\

\multicolumn{1}{|l|}{1024} & \multicolumn{1}{l|}{$1.7213\cdot 10^{-3}$} & \multicolumn{1}{l|}{1.75} & \multicolumn{1}{l|}{$5.8559\cdot 10{-4}$}& \multicolumn{1}{l|}{1.85} & \multicolumn{1}{l|}{$2.3662\cdot 10^{-4}$} & \multicolumn{1}{l|}{1.83}&\multicolumn{1}{l|}{$2.9458\cdot 10^{-4}$}& 1.78\\

\multicolumn{1}{|l|}{2048} & \multicolumn{1}{l|}{$4.4172\cdot 10^{-4}$} & \multicolumn{1}{l|}{1.96} & \multicolumn{1}{l|}{$1.4176\cdot 10^{-4}$}& \multicolumn{1}{l|}{2.05} & \multicolumn{1}{l|}{$6.5280\cdot 10^{-5}$} & \multicolumn{1}{l|}{1.86}  &\multicolumn{1}{l|}{$7.5300\cdot 10^{-5}$} &1.97\\
  \hline
\end{tabular}
\caption{\label{EOC1.1} Accuracy study for the coupled problem for shear flow using $N=1,2,3,10$ moment equations. The reference solution uses the same number of moment equations as the coarse solution. The models  are coupled with \Cref{shearalgo}. }
\end{table}
Next, we calculate the reference solution on a highly resolved grid
with 8192 grid cells
using $N=20$ moment equations. We compare this reference solution with
the numerical solution on a coarse grid using $N=\{3,6,10,15\}$ moments in \Cref{EOC2.1}.
\begin{table}[H]
\centering
\begin{tabular}{|l|ll|ll|ll|ll|}
\hline
                       &         N=3              &                      &         N=6              &  &   N=10               
             &    &   N=15 &   \\\hline
\multicolumn{1}{|l|}{grid} & \multicolumn{1}{l|}{$L_{\infty}$-Error} &EOC & \multicolumn{1}{l|}{$L_{\infty}$-Error} &\multicolumn{1}{l|}{EOC}& \multicolumn{1}{l|}{$L_{\infty}$-Error} &\multicolumn{1}{l|}{EOC} & \multicolumn{1}{l|}{$L_{\infty}$-Error} &\multicolumn{1}{l|}{EOC} \\ \hline

\multicolumn{1}{|l|}{256} & \multicolumn{1}{l|}{$9,0513\cdot 10^{-3}$} & \multicolumn{1}{l|}{}  & \multicolumn{1}{l|}{$2.4774\cdot 10^{-3}$}  & \multicolumn{1}{l|}{$$} & \multicolumn{1}{l|}{$4.2754\cdot 10^{-3}$}  & \multicolumn{1}{l|}{$$}&  \multicolumn{1}{l|}{$4.4779\cdot 10^{-3}$}& \multicolumn{1}{l|}{} \\

\multicolumn{1}{|l|}{512} & \multicolumn{1}{l|}{$8.4361\cdot 10^{-3}$} & \multicolumn{1}{l|}{0.10}  & \multicolumn{1}{l|}{$6.9811\cdot 10^{-4}$}  & \multicolumn{1}{l|}{1.83} & \multicolumn{1}{l|}{$1.2723\cdot 10^{-3}$} & \multicolumn{1}{l|}{1.75}  & \multicolumn{1}{l|}{$1.3225\cdot 10^{-3}$} & \multicolumn{1}{l|}{1.76} \\

\multicolumn{1}{|l|}{1024} & \multicolumn{1}{l|}{$8.3212\cdot 10^{-3}$} & \multicolumn{1}{l|}{0.02} & \multicolumn{1}{l|}{$3.3731\cdot 10^{-4}$} & \multicolumn{1}{l|}{1.05} & \multicolumn{1}{l|}{$3.5854\cdot 10^{-4}$} & \multicolumn{1}{l|}{1.83}  & \multicolumn{1}{l|}{$3.3623\cdot 10^{-4}$}& \multicolumn{1}{l|}{1.98} \\

\multicolumn{1}{|l|}{2048} & \multicolumn{1}{l|}{$8.2830\cdot 10^{-3}$} & \multicolumn{1}{l|}{0.007}  & \multicolumn{1}{l|}{$3.3305\cdot 10^{-4}$} & \multicolumn{1}{l|}{0.02}  & \multicolumn{1}{l|}{$8.5690\cdot 10^{-5}$} & \multicolumn{1}{l|}{2.06}  & \multicolumn{1}{l|}{$8.8320\cdot 10^{-5}$} & \multicolumn{1}{l|}{1.93}\\
%
  \hline

\end{tabular}
\caption{\label{EOC2.1} Accuracy study for the coupled shear flow problem using $N=3,6,10,15$. The reference solution uses $N=20$ moment equations. The models are coupled with \Cref{shearalgo}. }
\end{table}
The accuracy study in \Cref{EOC2.1} illustrates the convergence due to
grid refinement as well as due to an increase of the number of
moments.
For $N=3$ the error due to an inadequate 
number of moment equations dominates the error on all grids
and we do not observe convergence as the mesh is refined. For $N=6$
the discretisation error dominates the error on relatively coarse
grids. On finer grids the error due to an insufficient value of $N$ dominates
the total error. Therefore, we only observe the expected second order
convergence rates on  the coarser grids. For $N\ge 10$ the
discretisation error dominates the total error on all considered grids
and we observe the expected second order convergence rates as the grid
is refined.
From the different values of the error on grids with 1024 or 2048
cells one can also observe how an increase of the number of moment
equations leads to a decrease of the error. 


In the following example, we use the discretisation presented in
\autoref{sec:numerical} and \autoref{sec:Bulk} to again solve the coupled
moment system for shear flow. Now we adapt the number of moment equations
locally in order to resolve the solution structure accurately and efficiently.
Based on the results of \autoref{sec:estimate}, we will use the magnitude
of the residuum as an error indicator and choose the number of moments accordingly.
\begin{example}
\label{exa:interface} 
We consider the moment system for shear flow (\ref{eqn:inhomogen}) coupled to the diffusion equation (\ref{eqn:diffusion}) with initial data on the interval $[0,100]$ of the form
\begin{equation*}
\label{eqn:generaliesedRP2}
\begin{array}{ccccc}
\rho(x,0)=\exp\left(-10(x-50)^2\right) \\[8pt]
w(x,0)=0. 
\end{array}
\end{equation*}
All other moments are initially set to zero.
We use the parameters $D_r=0.01$ and $\delta=1$ and periodic boundary conditions.
The solution is computed at time $t=50$ and the zeroth order moment
$\rho$ will be shown.
We compare this solution using different levels of detail with the solution of the coupled problem for shear flow using the same number of moment equations throughout the domain. 
\end{example}
In \autoref{fig:N=123}, the solution of the coupled moment system using the same number of moment equations throughout the domain is illustrated.

\begin{figure}[H]%
\captionsetup[subfloat]{labelformat=empty}
  \centering
 \subfloat[][]{\includegraphics[width=0.3\linewidth]{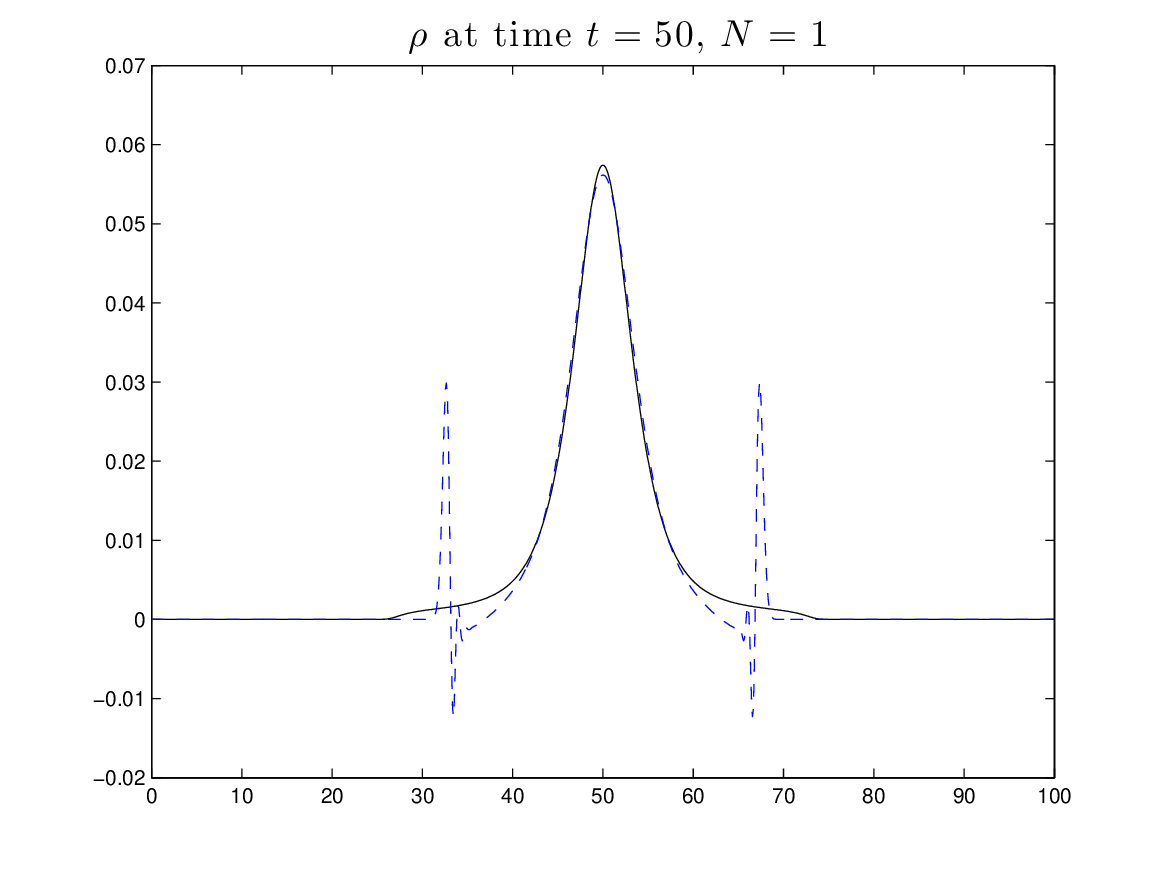}}
 \quad
 \subfloat[][]{\includegraphics[width=0.3\linewidth]{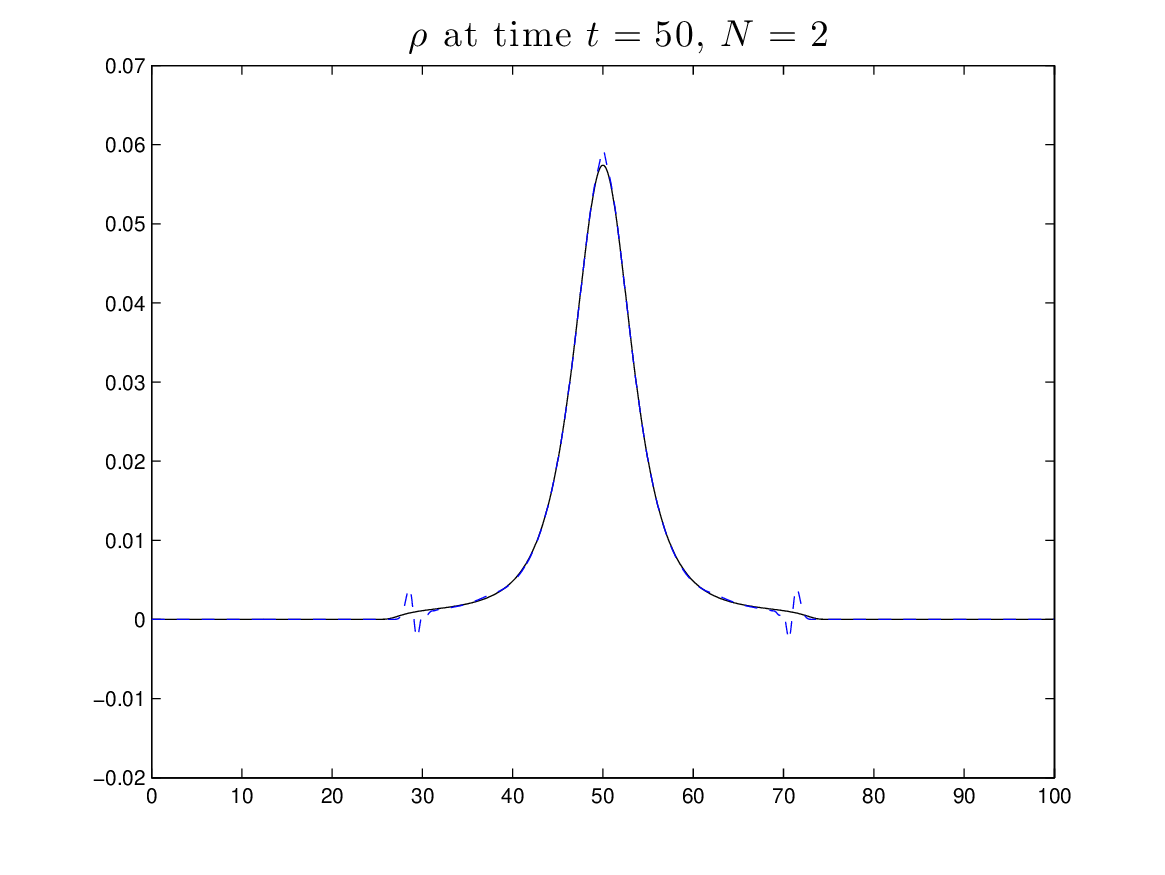}}
 \quad
 \subfloat[][]{\includegraphics[width=0.3\linewidth]{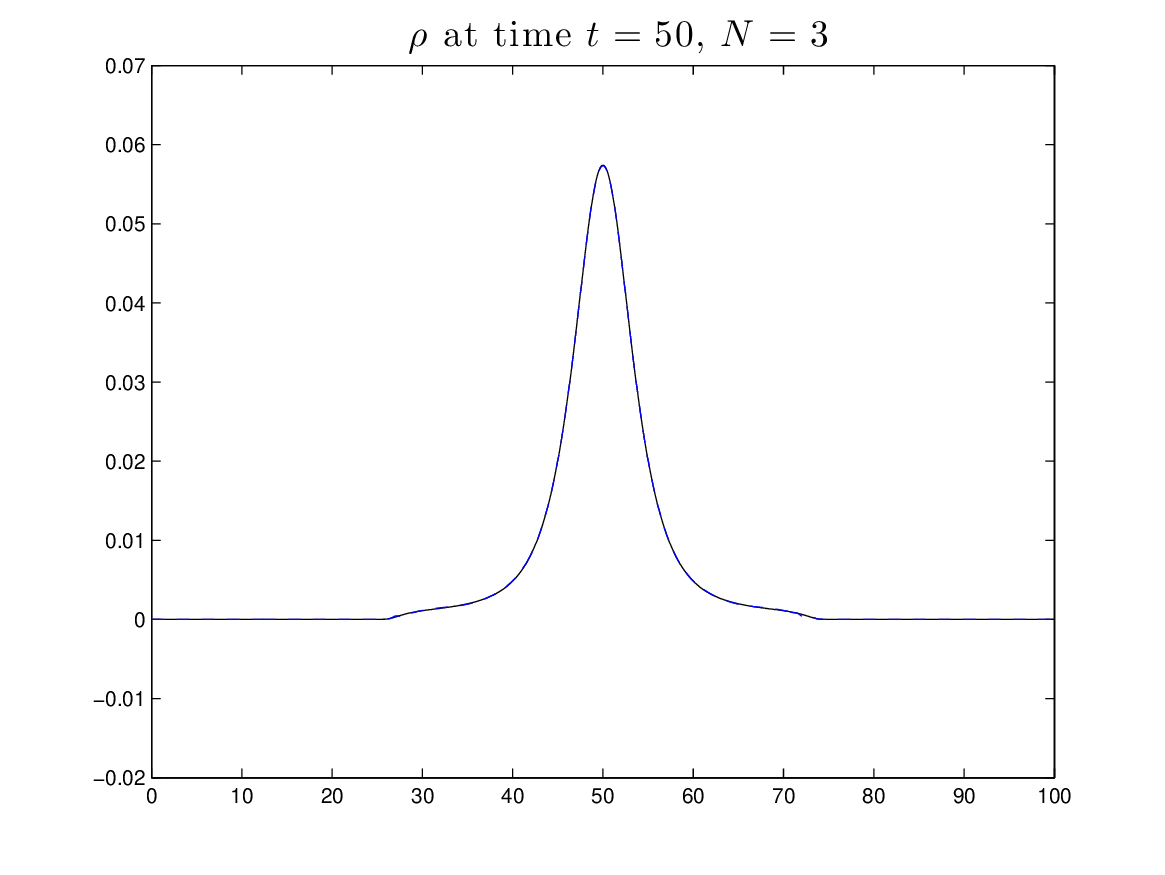}} \\
 \caption{Approximation of the coupled problem for shear flow as described in \cref{exa:interface}. The blue dashed dotted curve shows the density at time $t=50$ for different values of $N$. The black solid line is a reference solution.}
    \label{fig:N=123}%
\end{figure}

We plot the density $\rho$ at time $t=50$ for $N=1$, $N=2$ and $N=3$
as a blue dashed line. The black solid line is a reference solution
using $N=20$. Using $N=1$, i.e.\ only
the three moment equations for $\rho$, $C_1$ and $S_1$,
approximations of the coupled fluid-moment system lead to negative and
thus
unphysical values in density  for $ x \in (30,40) \cup
(60,70)$. In spatial regions of low density the solution of the moment
system using  $N=1$ moment equations approximates the highly resolved
solution very well. The solution of the coupled moment system using
$N=2$ moments still leads to negative values of density in the
intervals $[25,35]$ and $[65,75]$. Also in the area of the highest
density $[45,55]$, the moment system using $N=1$ or $N=2$ leads to
visible deviations
from the solution structure of the reference solution. The density
computed by the coupled moment system with $N=3$ moment equations
compares well with the reference solution and does not show any
unphysical values.

The analytical considerations of \cref{sec:estimate} suggest that the
quantities
\begin{equation*}
 |\hat{R}_{2N+2}|:= |\frac{1}{4} \partial_x {S}_N +\frac{N+1}{2}
 \partial_x w
 S_N|, \quad
 |\hat{R}_{2N+3}|:= |\frac{1}{4} \partial_x C_N + \frac{N+1}{2}
 \partial_x w C_N| 
\end{equation*}
can be used as error indicators. In Figures
\ref{fig:errorUnderResolved_R1} and \ref{fig:errorUnderResolved_R2} we plot
these quantities for the numerical solution at time $t=50$ using
$N=1,2,3$.

\begin{figure}[H]%
\captionsetup[subfloat]{labelformat=empty}
  \centering
 \subfloat[][]{\includegraphics[width=0.3\linewidth]{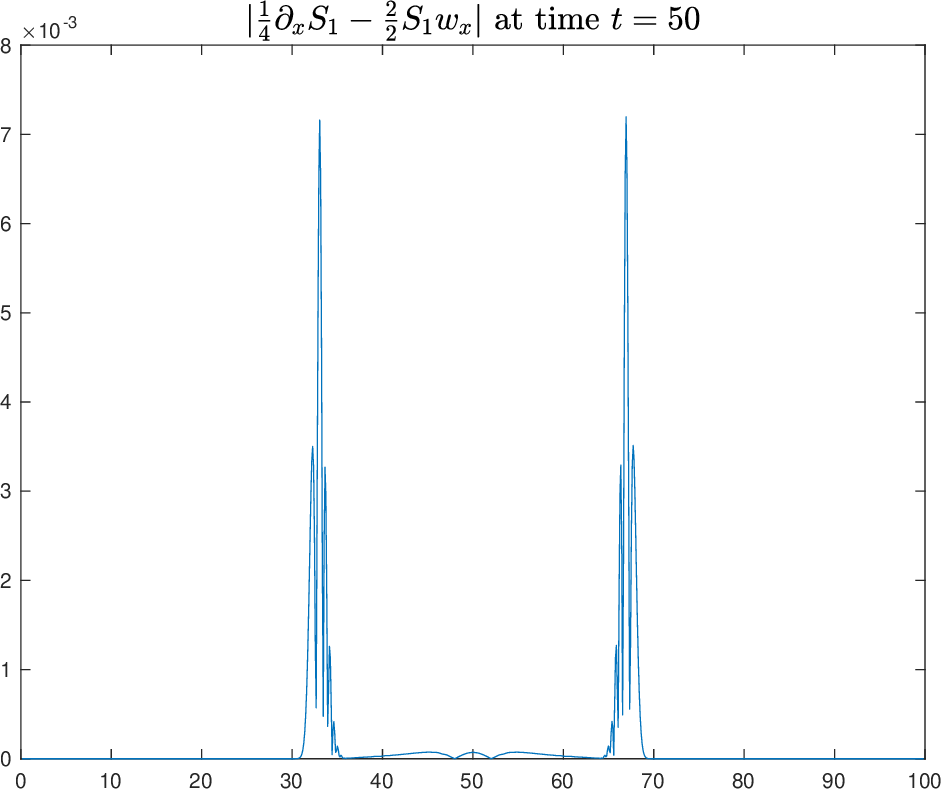}}
 \quad
 \subfloat[][]{\includegraphics[width=0.3\linewidth]{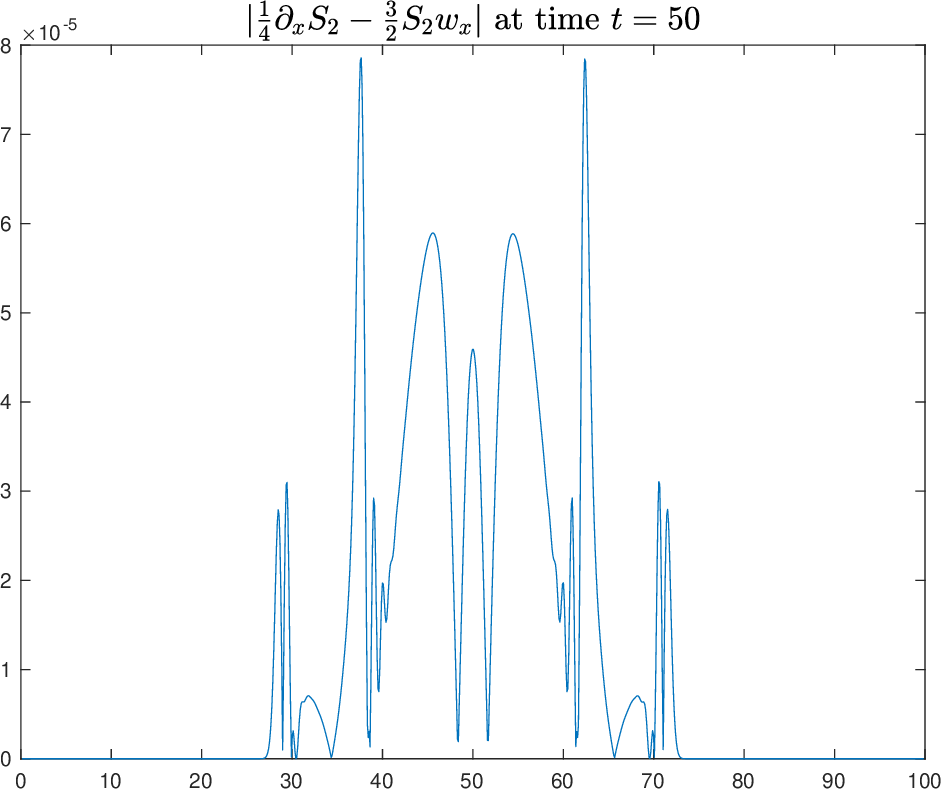}}
 \quad
 \subfloat[][]{\includegraphics[width=0.3\linewidth]{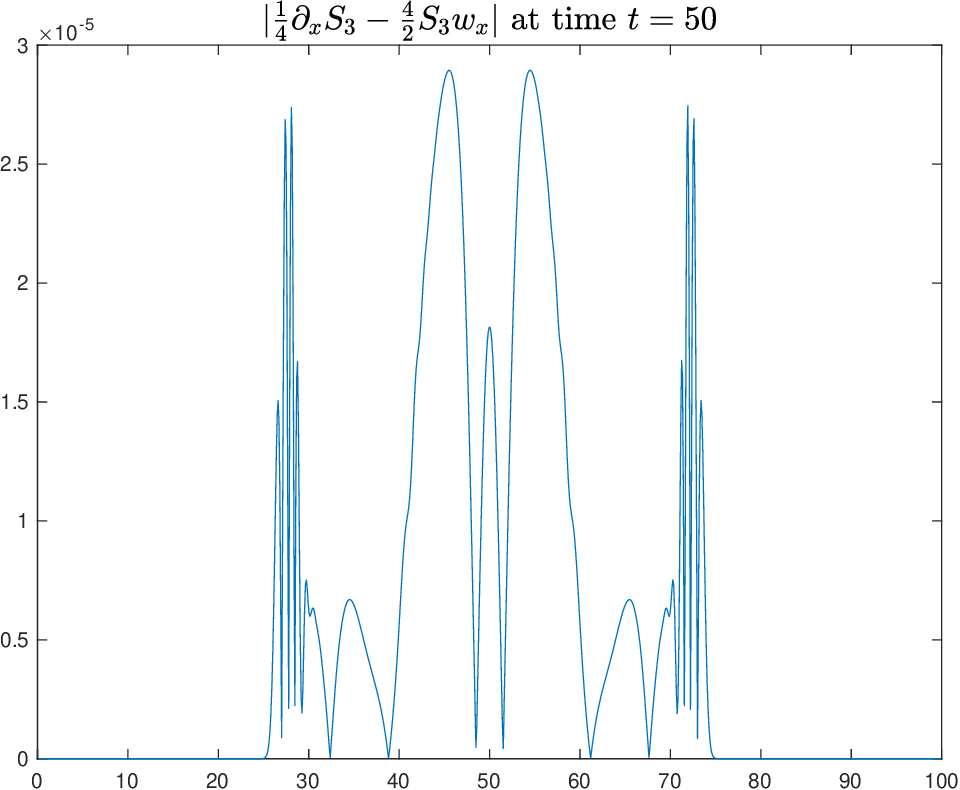}}
 \caption{Error indicators $|\hat{R}_{2N+2}|$ at time $t=50$ for
   $N=1$, $N=2$ and $N=3$. Note the different scaling of the $y$ axis.}
    \label{fig:errorUnderResolved_R1}%
  \end{figure}
  \begin{figure}[H]%
\captionsetup[subfloat]{labelformat=empty}
  \centering
 \subfloat[][]{\includegraphics[width=0.3\linewidth]{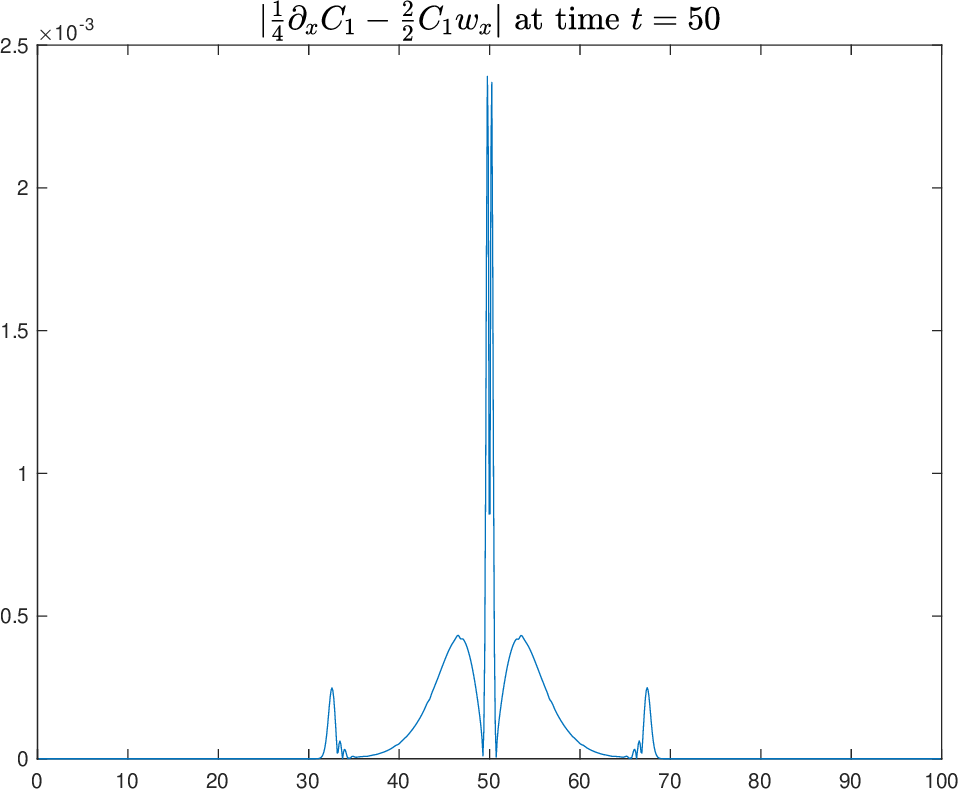}}
 \quad
 \subfloat[][]{\includegraphics[width=0.3\linewidth]{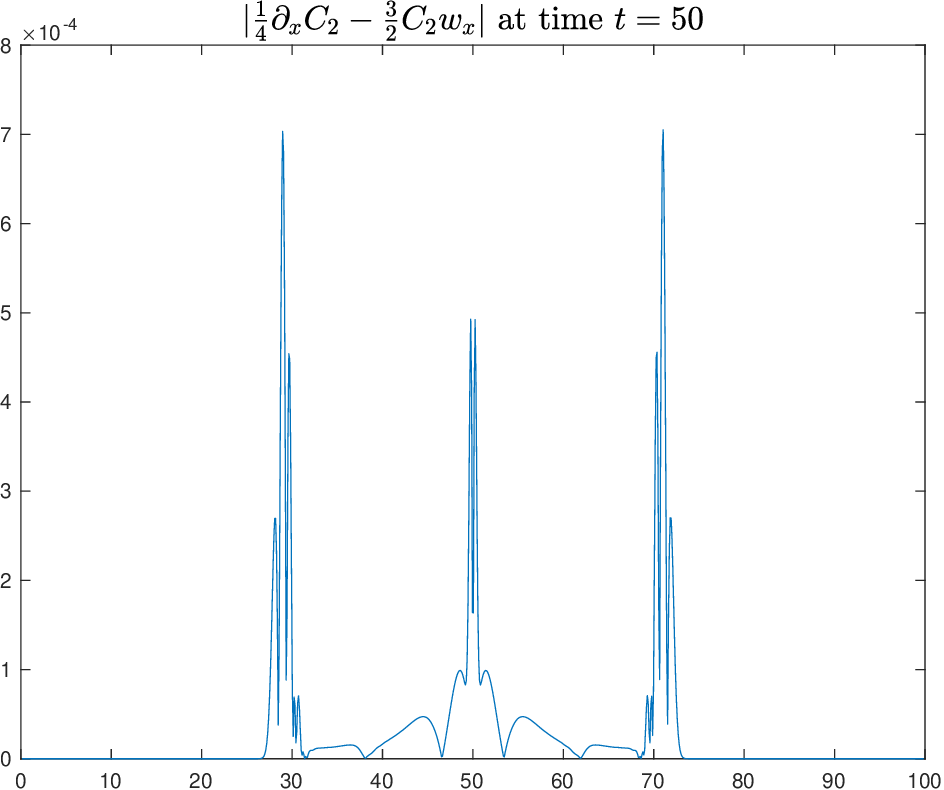}}
 \quad
 \subfloat[][]{\includegraphics[width=0.3\linewidth]{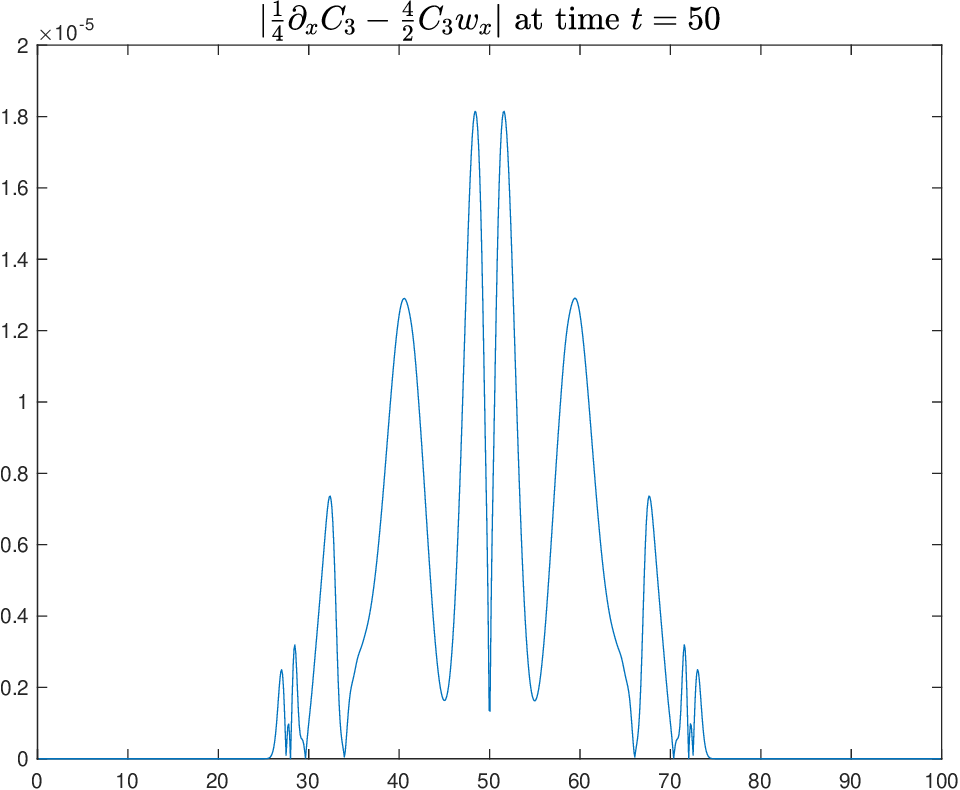}}
 \caption{Error indicators $|\hat{R}_{2N+3}|$ at time $t=50$ for
   $N=1$, $N=2$ and $N=3$. Note the different scaling of the $y$ axis.}
    \label{fig:errorUnderResolved_R2}%
  \end{figure}
For $N=1$, we can see that the magnitude of $|\hat{R}_4|$ has maximal
values for $ x \in (30,40) \cup (60,70)$.
Precisely in these intervals, the moment $\rho$ for $N=1$ has
unphysical values in
\autoref{fig:N=123}.
The error indicator
$|\hat{R}_5|$ has its maximum at the center. For $N=2$ both
$|\hat{R_{6}}|$ and $|\hat{R}_7|$ indicate the largest error within
the interval $[27,73]$ but the magnitude of the error indicators for
$N=2$ are more than an order of magnitude smaller than for $N=1$. For
$N=3$ the magnitude of the error indicators $|\hat{R}_8|$ and
$|\hat{R}_9|$ decrease further.  In this case   the solution structure of $\rho$ compares well with the reference solution.  

Finally, in \autoref{fig:errorResolved} we show the two components of the
error indicator for the highly resolved reference solution using
$N=20$.
Now both error indicators have values on the level of machine
precision. 
\begin{figure}[H]%
\captionsetup[subfloat]{labelformat=empty}
  \centering
 \subfloat[][]{\includegraphics[width=0.3\linewidth]{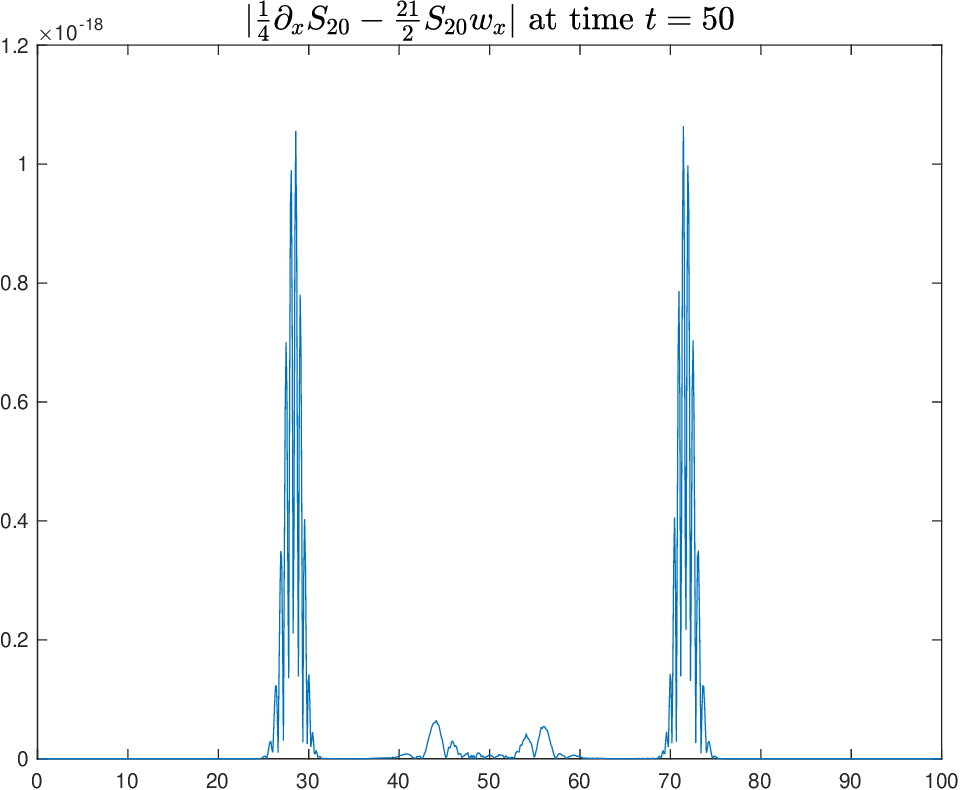}}
 \quad
 \subfloat[][]{\includegraphics[width=0.3\linewidth]{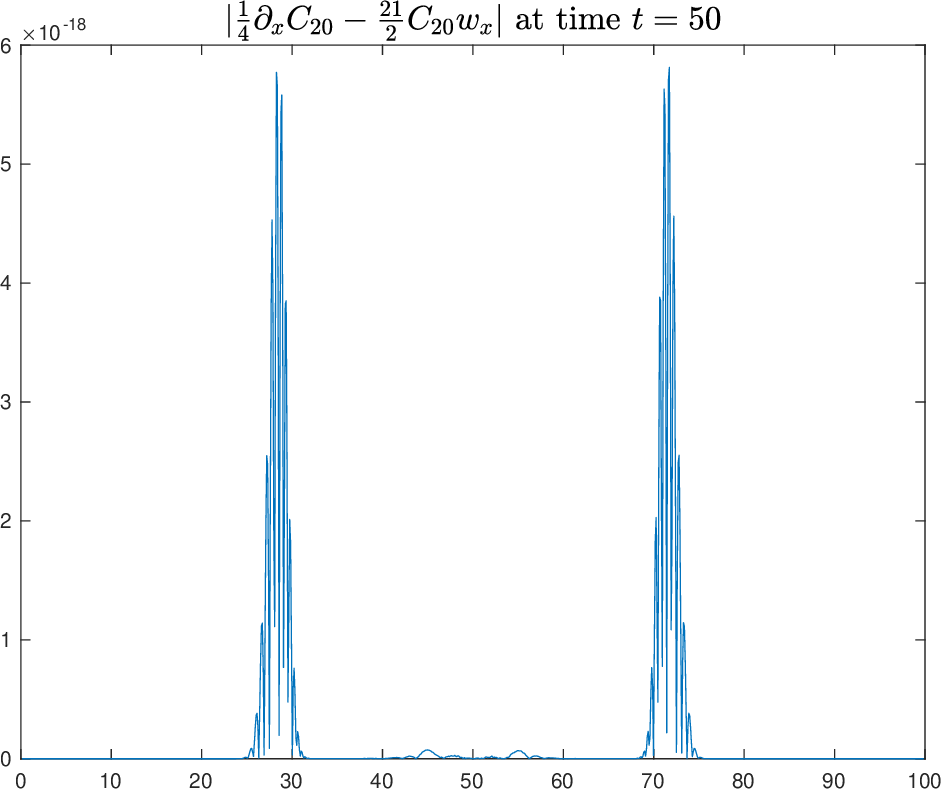}}
\caption{Error indicators for the resolved reference solution computed
at $t=50$.}
    \label{fig:errorResolved}%
  \end{figure}

Based on these observations, we choose the number of moment equations
used in the coupled moment system in the different regions of the
domain $[0,100]$ to compute an accurate and efficient approximation of
the solution of \autoref{exa:interface}.
We use $N=1$ for the intervals $[0,20]$ and $[80,100]$, $N=2$ for
$[20,25] \cup [75,80]$ and $N=3$ for $[25,75]$.
In \autoref{fig:Interface}, the first component $\rho$ of the solution
of the coupled moment system with different resolution in different
spatial regions and initial data as described in \autoref{exa:interface}
is shown as a blue dashed line.  We again compare this solution of the coupled moment system using different levels of detail with a highly resolved reference solution, which is given as a black solid line. The solution of the coupled moment system using different levels of resolution shows no unphysical values  and compares very well with the solution structure of the reference solution.
\begin{figure}[H]%
\captionsetup[subfigure]{}
  \centering
  \includegraphics[width=0.40\linewidth]{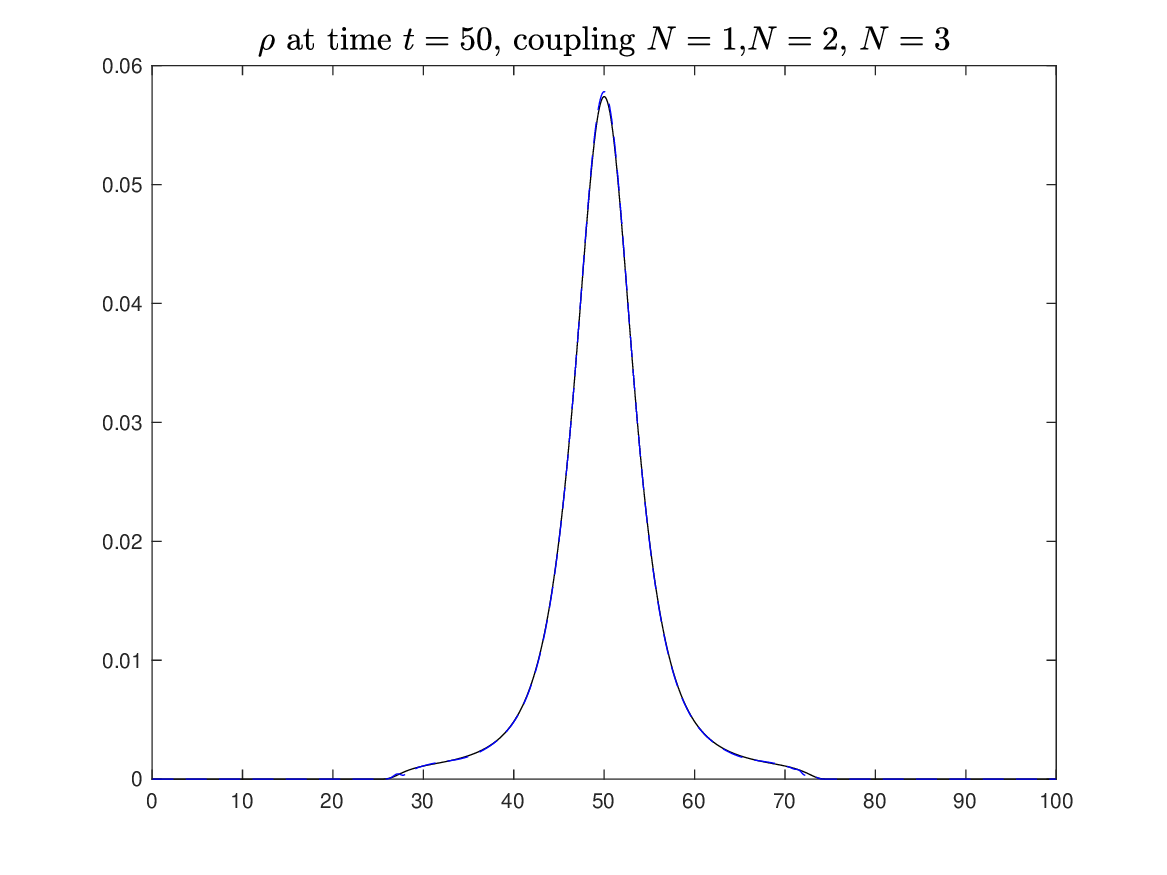}%
  \caption{Approximation of the coupled problem for shear flow as
    described in \autoref{exa:interface}. The blue dashed dotted curve
    shows the density at time $t=50$. We use $N=1$ for $ x \in [0,20]
    \cup  [80,100]$, $N=2$ for $x \in [20,25] \cup [75,80]$ and $N=3$
    for $x \in [25,75]$. The black solid line is a reference solution. }
    \label{fig:Interface}
  \end{figure}

 The comparison demonstrates that a local increase of the number of
 moment equations can avoid the unphysical solutions observed in the
 under-resolved case.
 The adaptive usage of
 moment systems of higher resolution leads to accurate approximations
 at lower computational costs.
The error indicators obtained from the residuum 
provide a useful selection criteria for choosing the number of
moments.

\subsection{Bulk-Coupling for Two-Dimensional Flow}	
\label{subsec:twodim}
We present the numerical discretisation of the two-dimensional moment
system (\ref{eqn:system2d}) coupled to the flow equation
(\ref{eqn:Navier}). The two-dimensional spatial domain
$\Omega:=[x_l,x_r]\times [z_l,z_r] $ is discretised as described in
\autoref{subsubsec:Wave2d}.

To apply the High-Resolution Wave Propagation Algorithm by LeVeque from \autoref{subsubsec:Wave2d}, the components of the discrete vector of moments $Q_{i,j}^n$ are defined as the average value over the $(i,j)$-th grid cell at time $t^n$, see (\ref{eqn:Q2d}). To solve the Navier-Stokes Equation with the projection method by Lee \cite{Lee}, the divergence-free velocity field $(u,w)$ in the two-dimensional Navier-Stokes equation is discretised on a staggered grid. While $(U_{i,j}^n,W_{i,j}^n)$ is defined at the cell center of $C_{i,j}$, the horizontal and vertical components of the discrete edge velocity field $(u_{i\pm1/2,j}^n,w_{i,j\pm1/2}^n)$ are defined at the midpoints of the interfaces $(x_{i\pm 1/2},z_j)$ and $(x_i,z_{j\pm 1/2})$ of the cell $C_{i,j}$. We compute the cell average $\pmb{U}_{i,j}^n = (U_{i,j}^n,W_{i,j}^n)$ over the $(i,j)$-th grid cell at time $t^n$ as
\begin{equation}
\pmb{U}_{i,j}^n \approx \frac{1}{\Delta x \Delta z}\int_{C_{i,j}}\pmb{u}\left(x,z,t^n\right) dx dz. 
\label{eqn:average}
\end{equation}
The discrete edge velocity is calculated by taking the average of the cell-centered values. For example, the left edge value of cell $C_{i,j}$ is
\begin{equation}
u_{i-1/2,j}^n = \frac{1}{2}(U_{i-1,j}^n+U_{i,j}^n). 
\label{eqn:edge}
\end{equation}
 The numerical solution of the coupled moment system for the
 two-dimensional flow problem is computed with the steps presented in
 \Cref{2dflow}.
 
 \begin{algorithm}[H]
\caption{Operator splitting algorithm for solving the coupled moment system for two-dimensional flow. }\label{2dflow}
\begin{itemize}
  \item[1.] $\frac{1}{2}\Delta t $ step on \hspace{2.5mm} $\partial_tQ(\pmb{x},t)=\phi(Q(\pmb{x},t))$. 
  \item[2.]   $\frac{1}{4}\Delta t $ step on \hspace{2.5mm} $w_t(\pmb{x},t)=-\frac{\delta}{Re}\rho$.   \hspace{8.5cm}\raisebox{1mm}{$\tikzmark{listing-2-end}$}
  
  \item[3.]   $\frac{1}{2}\Delta t $ step on \hspace{3mm} Navier-Stokes; Calculate $\partial_x u(\pmb{x},t)$, $\partial_z u(\pmb{x},t)$, $\partial_x w(\pmb{x},t)$, $\partial_z w(\pmb{x},t)$.
  \item[4.]    $\frac{1}{4}\Delta t $ step on \hspace{2.5mm}
    $\partial_t w(\pmb{x},t)=-\frac{\delta}{Re}\rho$. \hspace{7.5cm} \raisebox{-1mm}{$\tikzmark{listing-4-end}$}
  
   \item[5.]  $\Delta t $ step on \hspace{4.5mm} $\partial_tQ(\pmb{x},t)+A\partial_xQ(\pmb{x},t)+B\partial_zQ(\pmb{x},t)=0$.
   \item[6.]  $\frac{1}{4}\Delta t $ step on \hspace{2.5mm}
     $\partial_t w(\pmb{x},t)=-\frac{\delta}{Re}\rho$. \hspace{8.09cm} \raisebox{1mm}{$\tikzmark{listing-6-end}$}
    \item[7.] $\frac{1}{2}\Delta t $ step on \hspace{3mm} Navier-Stokes; Calculate $\partial_x u(\pmb{x},t)$, $\partial_z u(\pmb{x},t)$, $\partial_x w(\pmb{x},t)$, $\partial_z w(\pmb{x},t)$.
   \item[8.]  $\frac{1}{4}\Delta t $ step on \hspace{2.5mm}
     $\partial_t w(\pmb{x},t)=-\frac{\delta}{Re}\rho$.  \hspace{8.5cm}\raisebox{-1mm}{$\tikzmark{listing-8-end}$}
   \item[9.]  $\frac{1}{2}\Delta t $ step on \hspace{2.5mm} $\partial_tQ(\pmb{x},t)=\phi(Q(\pmb{x},t))$. 
\end{itemize}
\AddNote{listing-2-end}{listing-4-end}{listing-2-end}{ first half time step of Strang splitting for the flow equation}
\AddNote{listing-6-end}{listing-8-end}{listing-6-end}{ second half
  time step of Strang splitting for the flow equation}
\end{algorithm} 
In each time step, the system of ordinary differential equations resulting from the source term of the moment system is solved with the classical Runge-Kutta method. The Navier-Stokes equation is solved with the projection method by Long Lee \cite{Lee}. The solution is used to calculate the discrete derivatives
\begin{equation*}
\begin{array}{lll}
\partial_x u\left(x_i,y_j,t^n\right)&=&\dfrac{u_{i+\frac{1}{2},j}^n-u_{i-\frac{1}{2},j}^n}{\Delta x},\quad i=1,\ldots,m, \quad j=1,\ldots,m, \\[12pt]
\partial_z u\left(x_i,y_j,t^n\right)&=&\dfrac{u_{i,j+\frac{1}{2}}^n-u_{i,j-\frac{1}{2}}^n}{\Delta z},\quad i=1,\ldots,m, \quad j=1,\ldots,m, \\[12pt]
\partial_x w\left(x_i,y_j,t^n\right)&=&\dfrac{w_{i+\frac{1}{2},j}^n-w_{i-\frac{1}{2},j}^n}{\Delta x}, \quad i=1,\ldots,m, \quad j=1,\ldots,m, \\[12pt]
\partial_z w\left(x_i,y_jt^n\right)&=&\dfrac{w_{i,j+\frac{1}{2}}^n-w_{i,j-\frac{1}{2}}^n}{\Delta z}, \quad i=1,\ldots,m, \quad j=1,\ldots,m. \\[12pt]
\end{array}
\end{equation*}
The solution of the homogeneous system of moment equations is calculated with the high-resolution Wave Propagation Algorithm by LeVeque described in \autoref{subsubsec:Wave2d}. 

\subsection{Numerical Simulation for the Coupled Moment System in
  a two-dimensional Flow}
In the two-dimensional case we consider the sedimentation of a
droplet of rod-like particles. 
\begin{example}
We consider the two-dimensional moment system (\ref{eqn:system2d})
coupled to the flow equations (\ref{eqn:Navier}) on the domain
$[0,100] \times [0,100]$ with doubly periodic boundary conditions.
The initial values are set to
\begin{equation*}
  \begin{split}
u(x,y,0)  = v(x,y,0) & = 0\\    
\rho(x,y,0) & = \exp \left(-0.025 \left( (x-50)^2+(y-75)^2 \right)\right)\\
C_i(x,y,0) = S_i(x,y,0) & = 0, \quad i=1,\ldots, N
  \end{split}
\end{equation*}
We set the parameter values to $\delta = 1$, $Re = 1$ and vary the
rotational diffusion parameter to consider  $D_r=1$ and $D_r=0.1$.  
\end{example}
In  \autoref{fig:2d-1-N4} we show the sedimenting droplet at three
different times using $D_r=1$ and $N=4$. The initially circular
droplet deforms as it sediments.
\begin{figure}[H]%
\captionsetup[subfigure]{}
  \centering
  \includegraphics[width=0.3\linewidth]{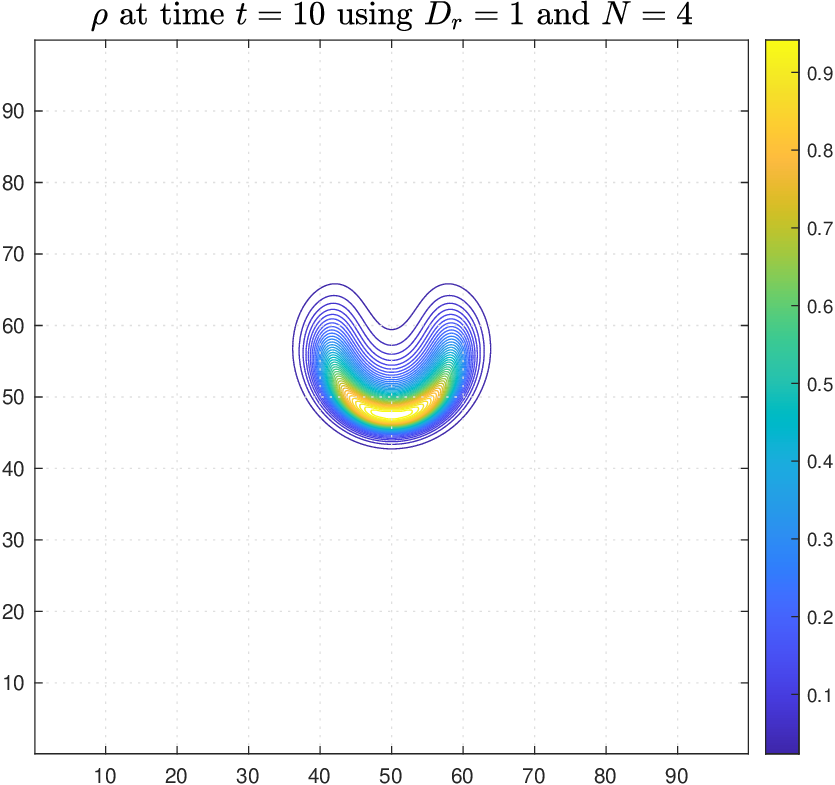}\hfil
  \includegraphics[width=0.3\linewidth]{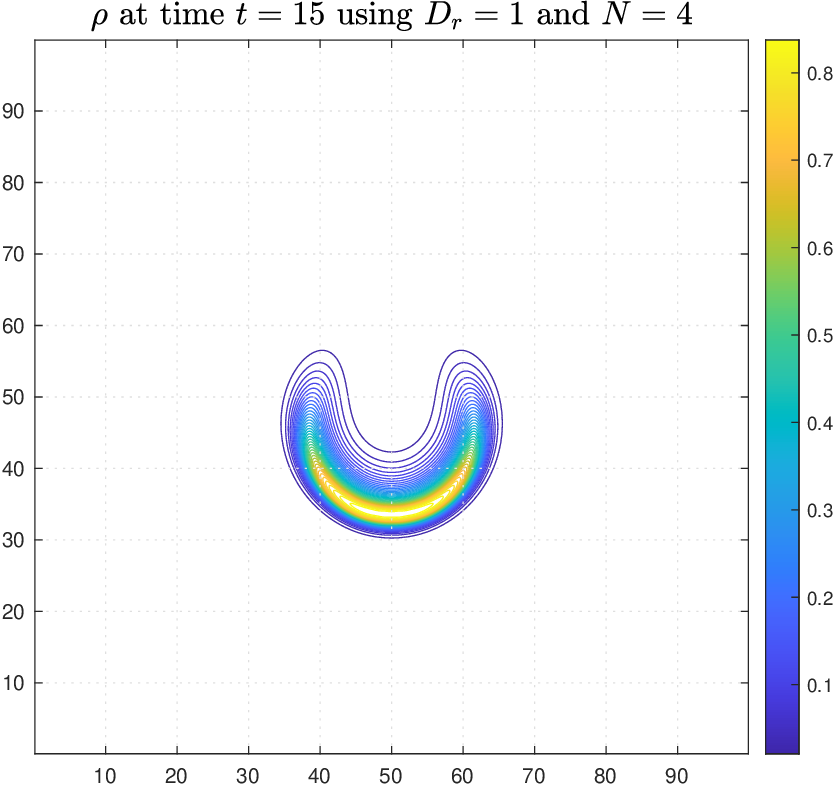}\hfil
  \includegraphics[width=0.3\linewidth]{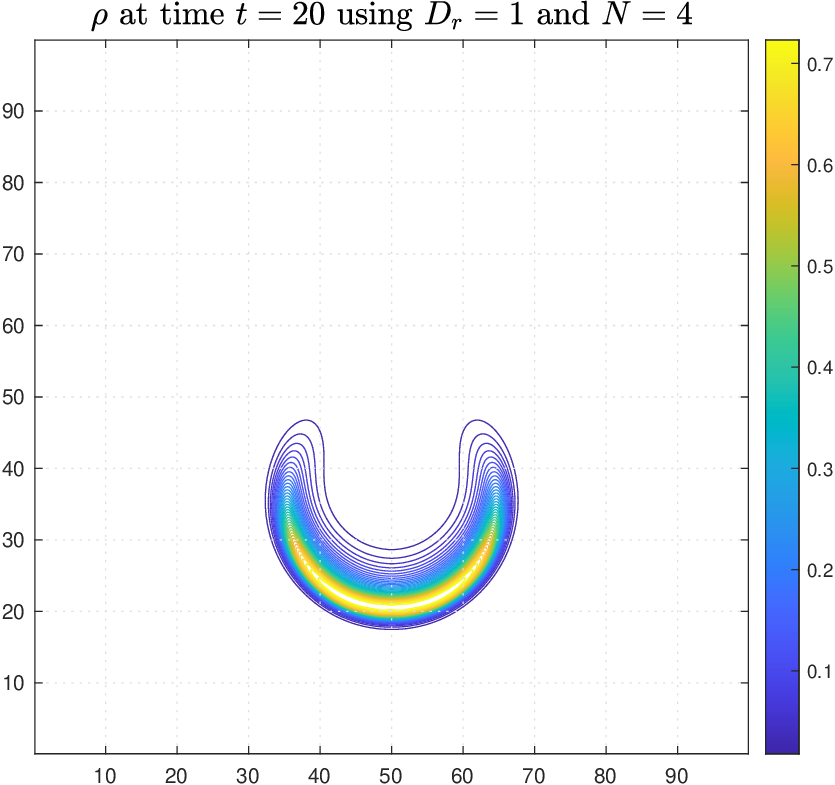}
  \caption{Approximation of the two-dimensional coupled problem using
    $D_r=1$ and $N=4$. Contour plots of density $\rho$ are shown at
    times $t=10, 15, 20$. }
    \label{fig:2d-1-N4}
  \end{figure}
The solution structure at 
time $t=20$ computed with fewer moment equations is shown in 
\autoref{fig:2d-1-N123}.
For $N=1$ differences in the solution structure are clearly
visible. For $N = 2, 3$ the solution structure compares well with those
observed for  $N=4$.  
\begin{figure}[H]%
\captionsetup[subfigure]{}
  \centering
  \includegraphics[width=0.3\linewidth]{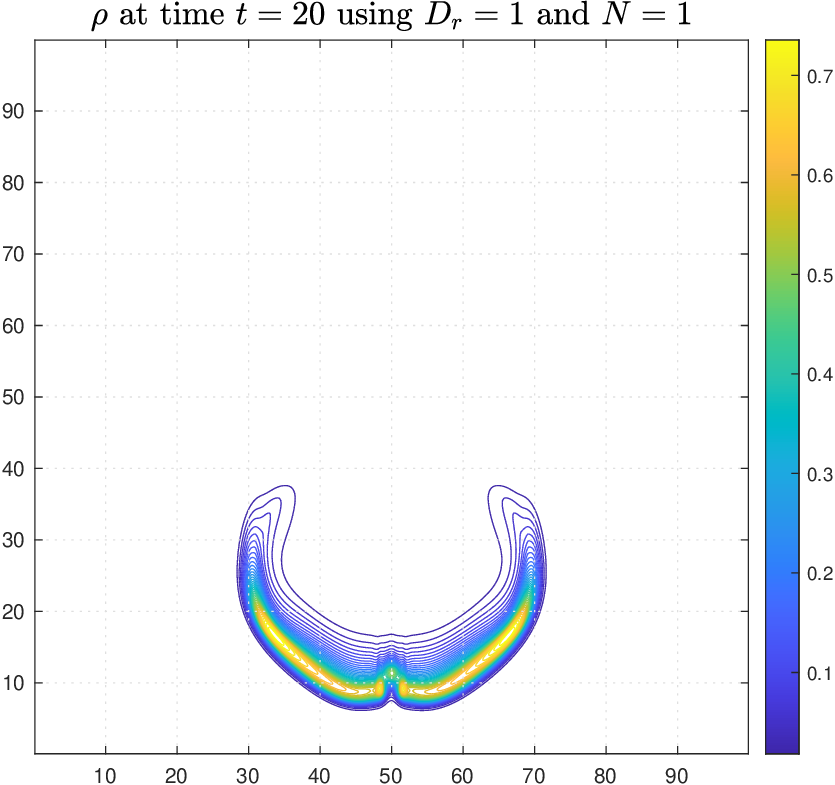}\hfil
  \includegraphics[width=0.3\linewidth]{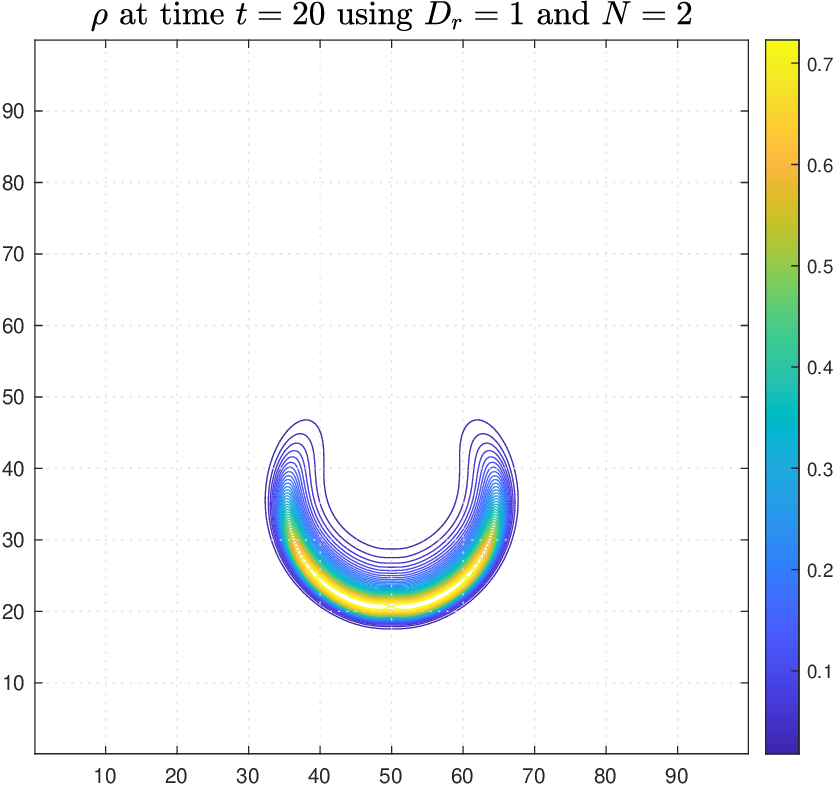}\hfil
  \includegraphics[width=0.3\linewidth]{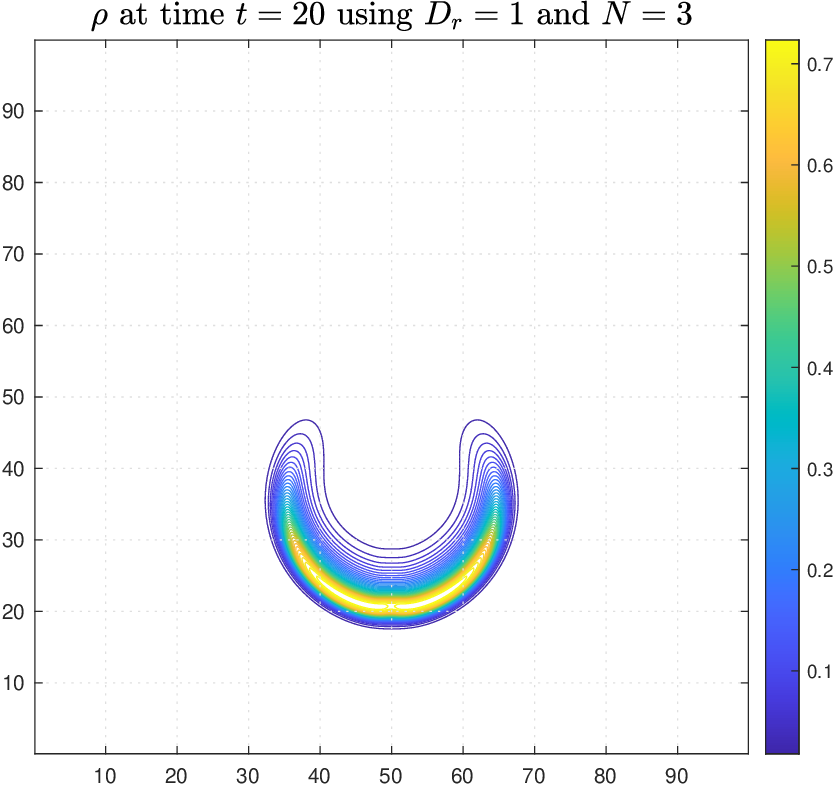}
  \caption{Approximation of the two-dimensional coupled problem at
    time $t=20$ using
    $D_r=1$ and from left to right $N=1,2,3$. }
    \label{fig:2d-1-N123}
  \end{figure}
  
Motivated by the error indicator derived for shear flow, we also
consider the quantities
\begin{equation*}
  \begin{split}
  |\hat{R}_{2N+2}| & := \left| -\frac{1}{4} \partial_x S_N - \frac{1}{4}
  \partial_z C_N - \frac{N+1}{2} (\partial_z w - \partial_x u) C_N - \frac{N+1}{2}
  (\partial_z u+\partial_x w)S_N \right|\\
|\hat{R}_{2N+3}| & := \left| \frac{1}{4} \partial_x C_N - \frac{1}{4}
  \partial_z S_N - \frac{N+1}{2} (\partial_z w-\partial_x u) S_N + \frac{N+1}{2}
  (\partial_z u+\partial_x w) C_N \right|.
\end{split}
\end{equation*}
In \autoref{fig:2d-R-1-N123} we show contour plots of
$|\hat{R}_{2N+2}|$ for $N=1,2,3$ at time $t=20$.
Contour plots of $|\hat{R}_{2N+3}|$ look similar and are therefore not
shown here.
The error indicator shows the expected behavior. In particular it
  becomes smaller as $N$ increases, indicating that
this quantity is well suited as  error indicator.
  \begin{figure}[H]%
\captionsetup[subfigure]{}
  \centering
  \includegraphics[width=0.3\linewidth]{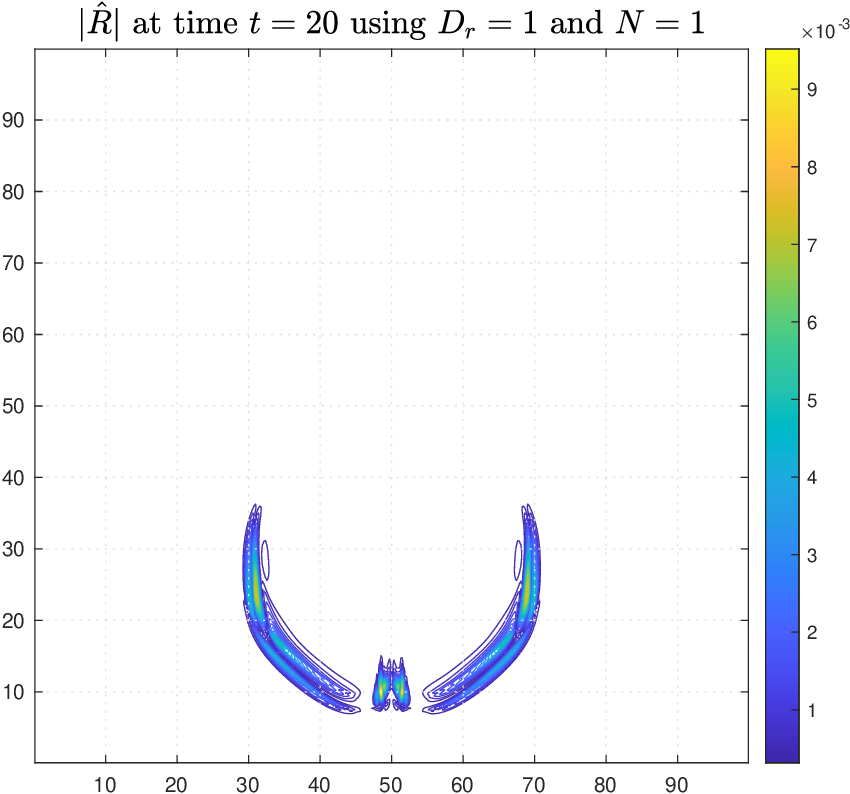}\hfil
  \includegraphics[width=0.3\linewidth]{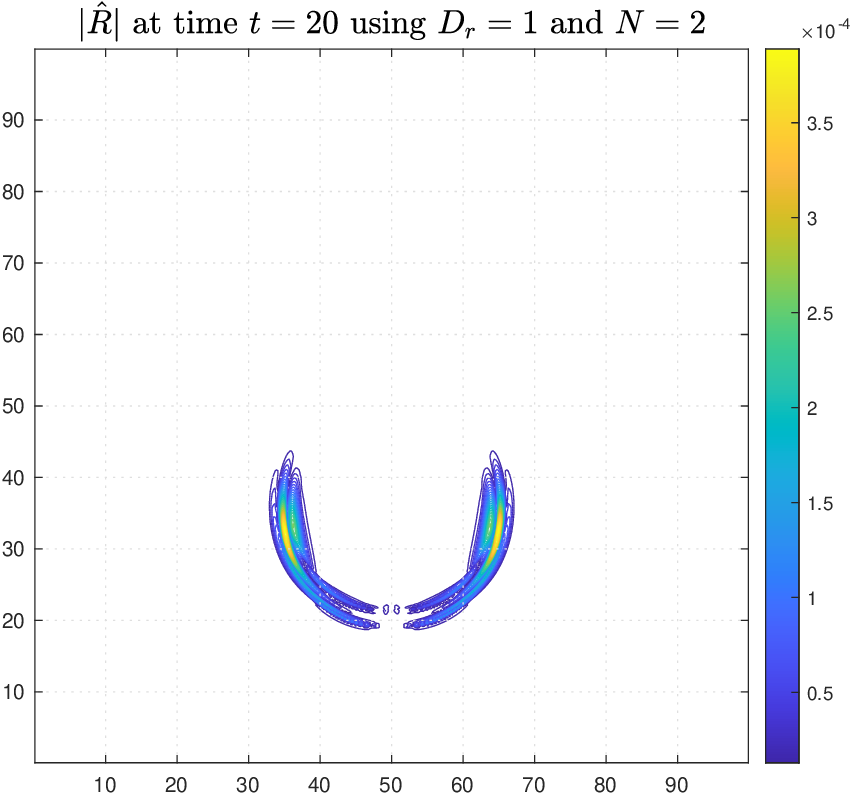}\hfil
  \includegraphics[width=0.3\linewidth]{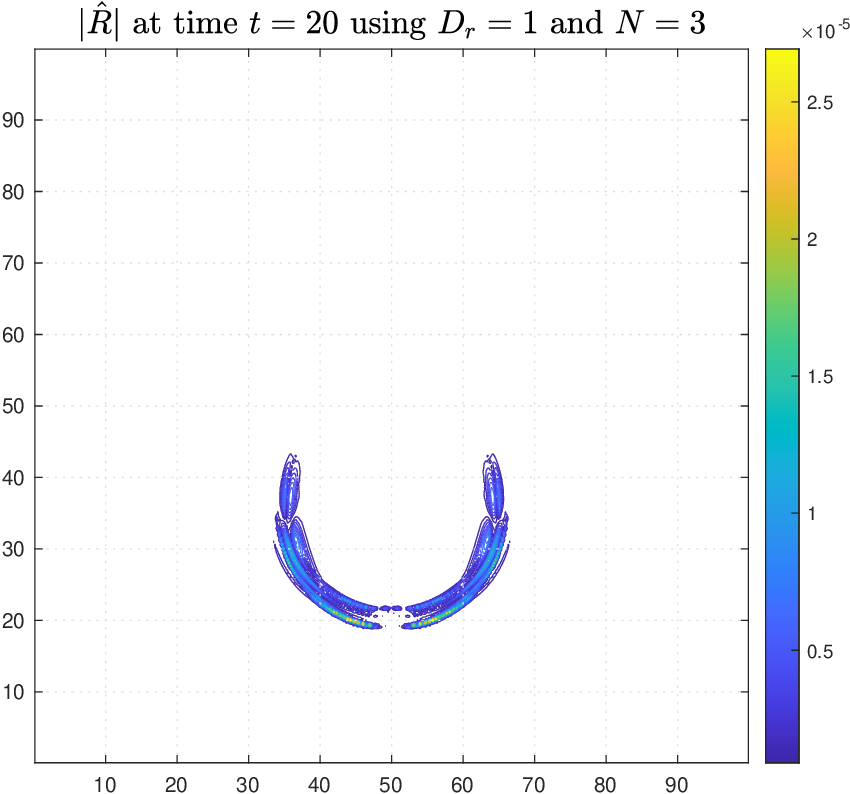}
  \caption{$|\hat{R}_{2N+2}|$ at time $t=20$ using
    $D_r=1$ and from left to right $N=1,2,3$. Note the different
    scales of the colorbar. }
    \label{fig:2d-R-1-N123}
  \end{figure}

In \autoref{fig:2d-01-N4} we show the sedimenting droplet at
different times for $D_r=0.1$ and $N=4$. In this case the droplet
starts to split into three smaller droplets with high density. 
\begin{figure}[H]%
\captionsetup[subfigure]{}
  \centering
  \includegraphics[width=0.3\linewidth]{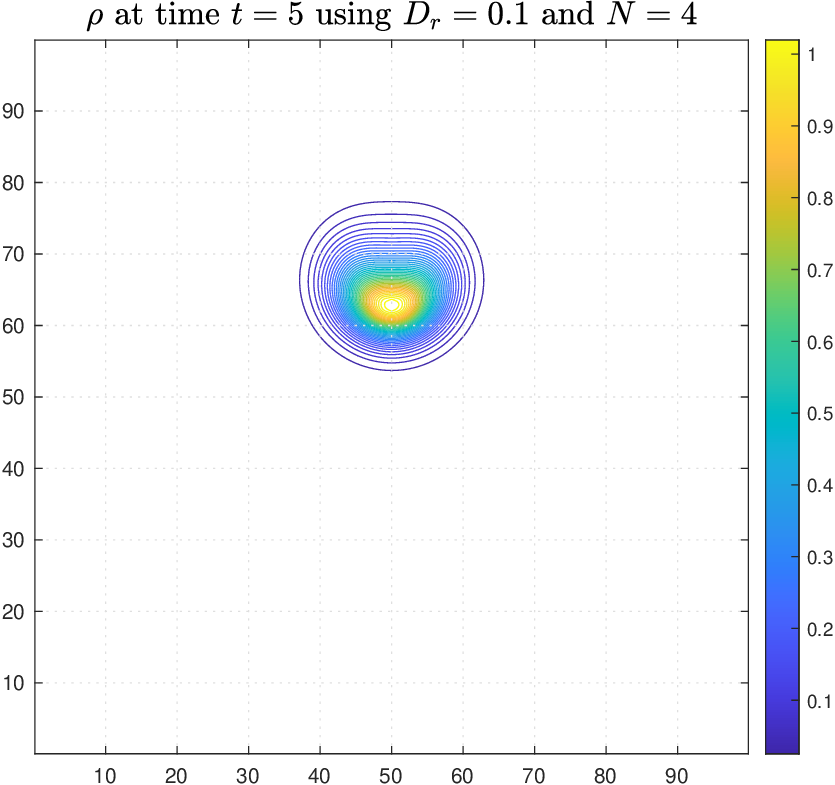}\hfil
  \includegraphics[width=0.3\linewidth]{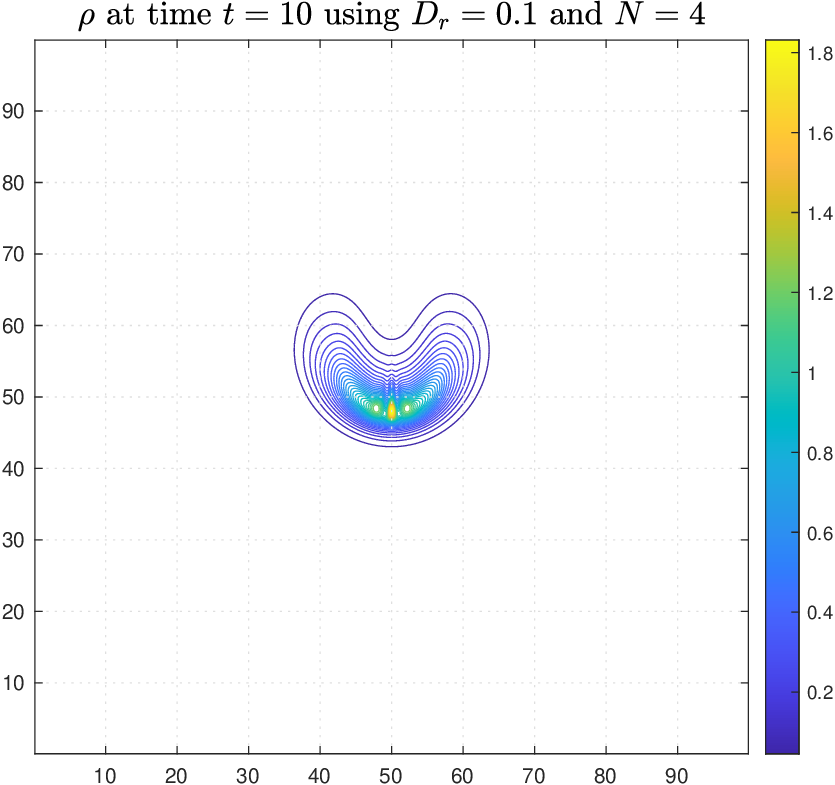}\hfil
  \includegraphics[width=0.3\linewidth]{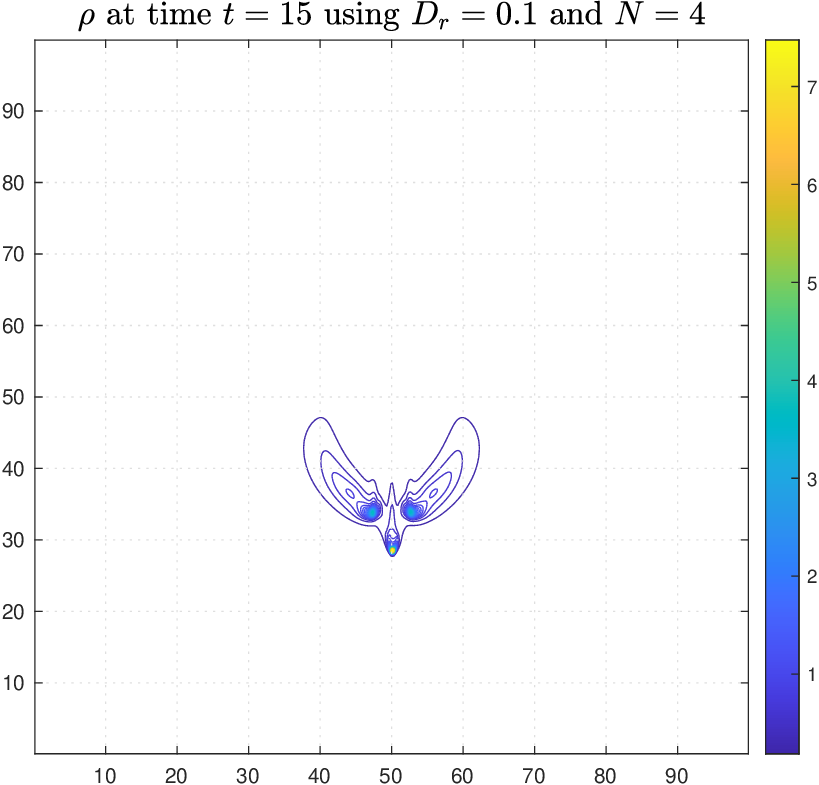}
  \caption{Approximation of the two-dimensional coupled problem using
    $D_r=0.1$ and $N=4$. Contour plots of density $\rho$ are shown at
     times $t=5, 10, 15$. }
    \label{fig:2d-01-N4}
  \end{figure}
In \autoref{fig:2d-01-N123} we show the solution at time $t=15$
computed using  $N=1, 2, 3$.
\begin{figure}[H]%
\captionsetup[subfigure]{}
  \centering
  \includegraphics[width=0.3\linewidth]{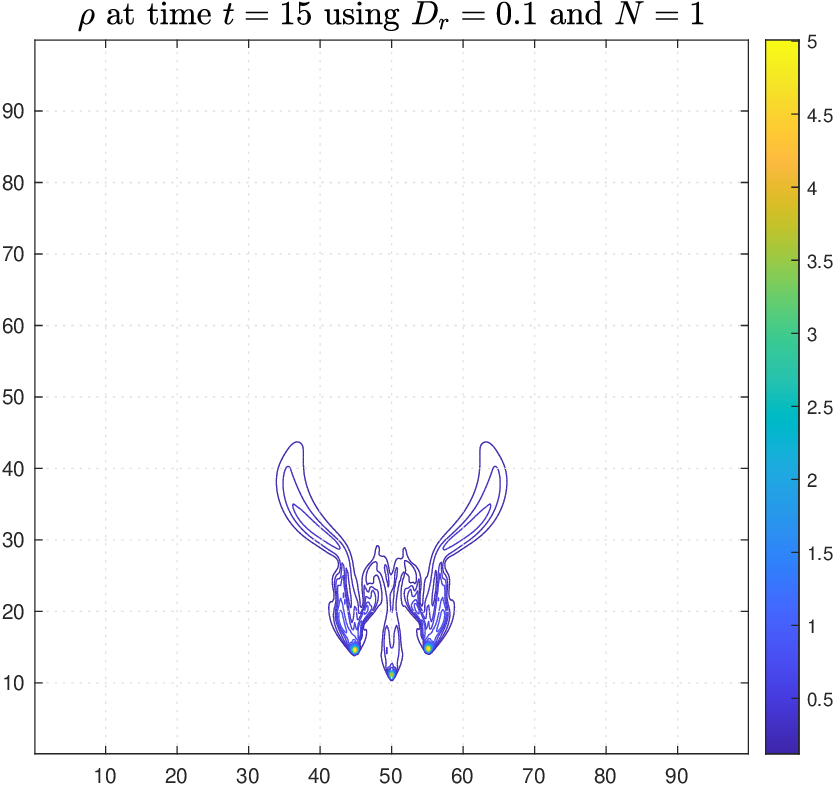}\hfil
  \includegraphics[width=0.3\linewidth]{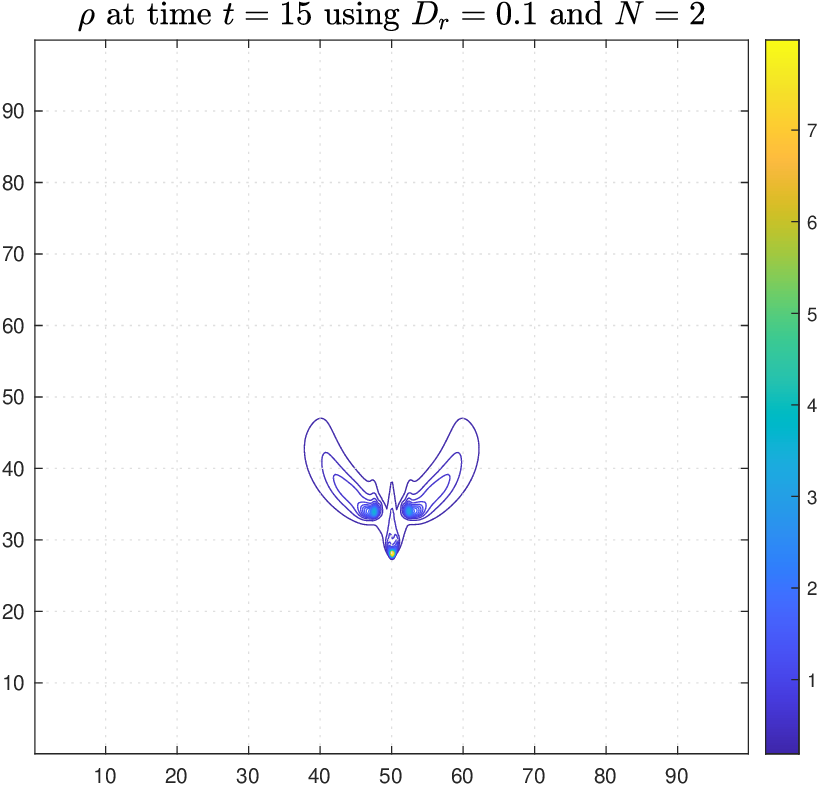}\hfil
  \includegraphics[width=0.3\linewidth]{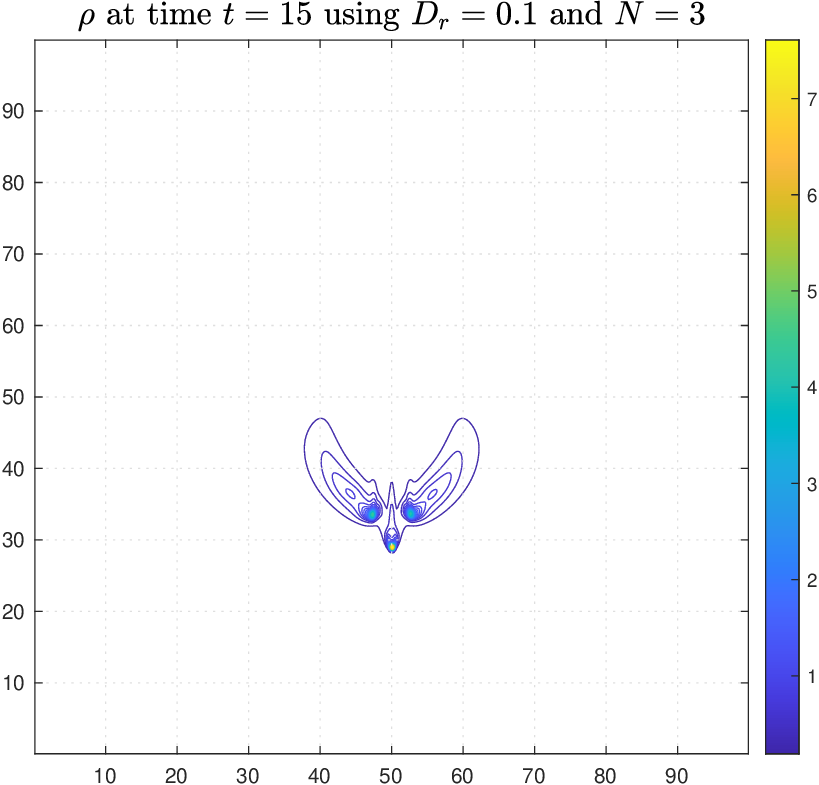}
  \caption{Approximation of the two-dimensional coupled problem at
    time $t=15$ using
    $D_r=0.1$ and $N=1,2,3$. }
    \label{fig:2d-01-N123}
  \end{figure}  
  For $N=1$ and $N=2$ we observe small negative values of density
  which are unphysical. In \autoref{fig:2d-R-01-N123} we show
  contour plots of the
  corresponding error indicators $|\hat{R}_{2N+2}|$. 
 \begin{figure}[H]%
\captionsetup[subfigure]{}
  \centering
  \includegraphics[width=0.3\linewidth]{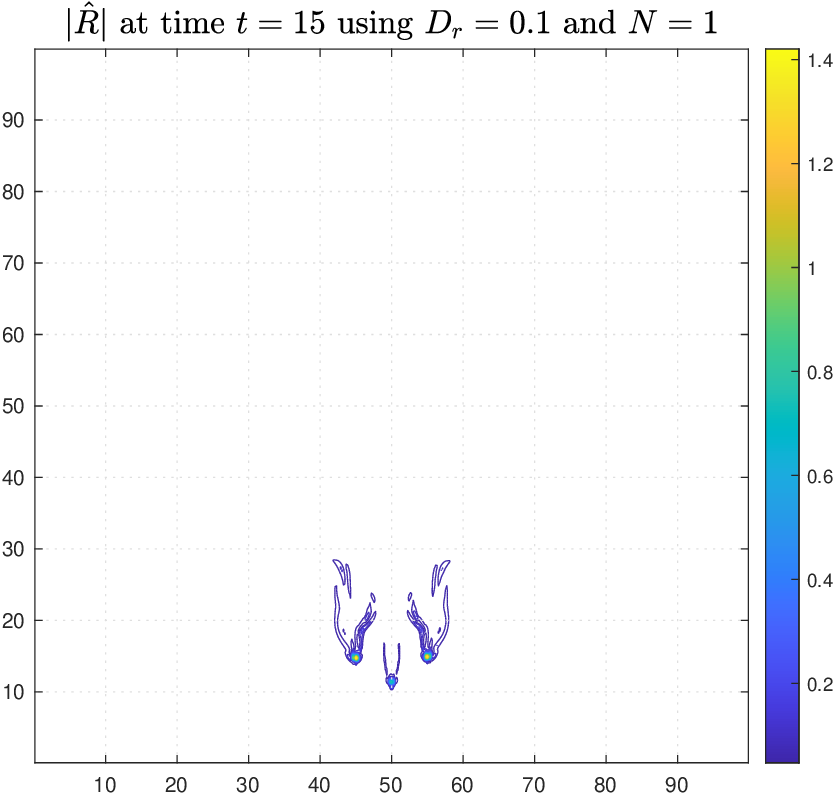}\hfil
  \includegraphics[width=0.3\linewidth]{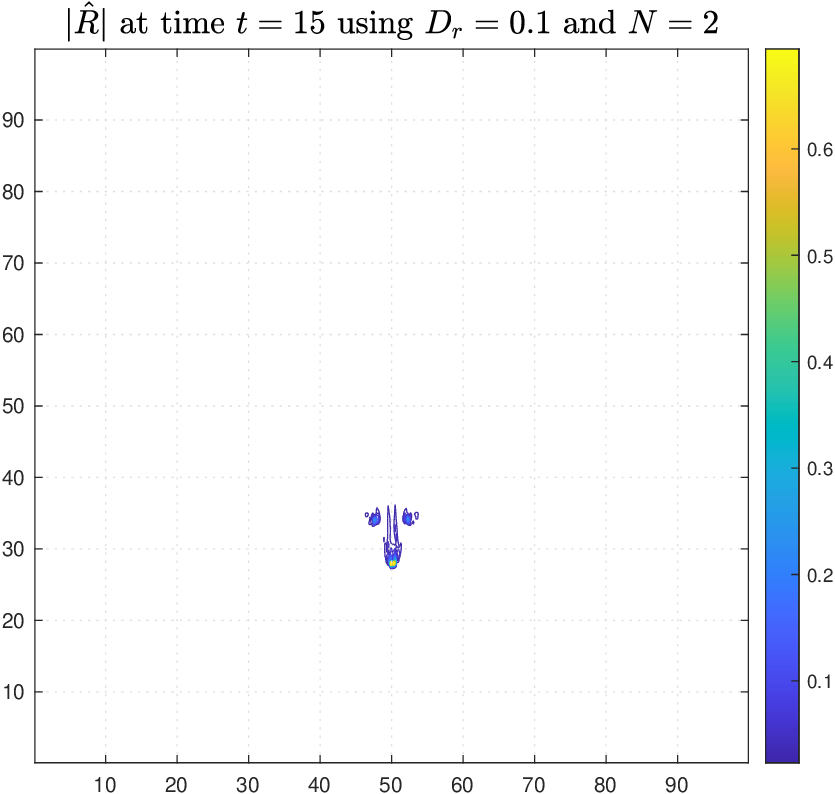}\hfil
  \includegraphics[width=0.3\linewidth]{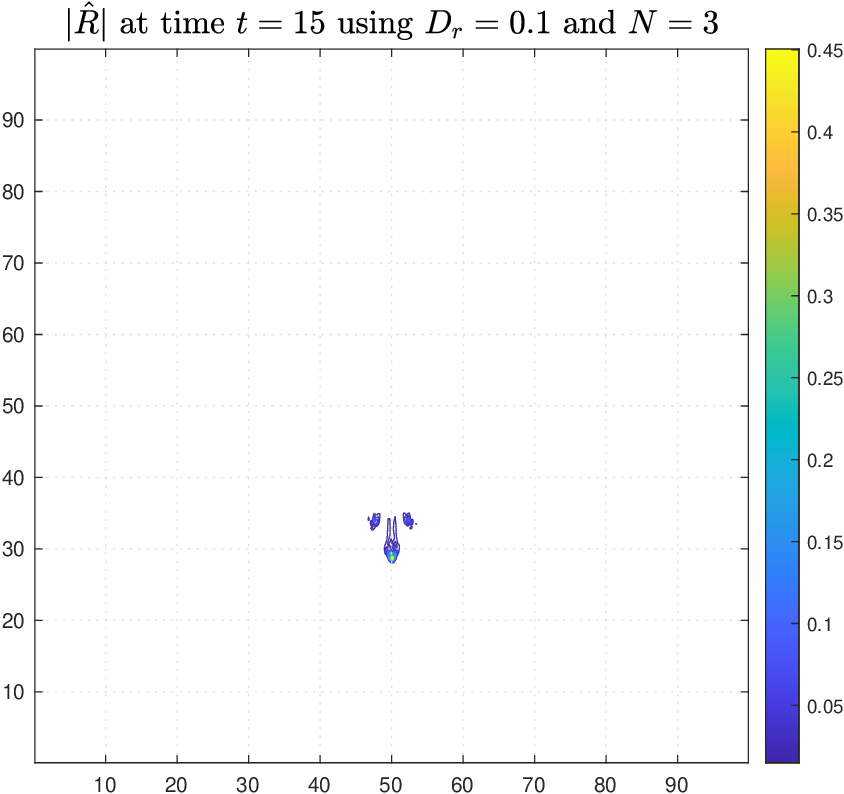}
  \caption{$|\hat{R}_{2N+2}|$ at time $t=15$ using
    $D_r=0.1$ and from left to right $N=1,2,3$. Note the different
    scales of the colorbar.}
    \label{fig:2d-R-01-N123}
  \end{figure}  
The error indicator predicts a relatively large error in regions where
the concentration is very large. In such regions an assumption
analogously to inequality (\ref{eqn:inequTheorem}) from 
\autoref{theorem:errorEstimate} might not even be satisfied and the use of
the considered quantity as error indicator might not be justified.
Furthermore, note that high concentrations of rod-like particles located at small
regions in space might arise as solutions of the coupled system
(\ref{eqn:allgemein}) but might not correspond to solution structures
observed in the sedimentation process. A reason  for this discrepancy
is that our coupled kinetic-fluid model was derived under the
assumption of a dilute suspension. In the concentrated regime
so-called excluded volume effects would have an influence on the
microscopic orientation and consequently on the solution structure of the
coupled model.
In the future we plan to include such effects
into the model equations. 

  All two-dimensional computations were performed on a grid with $512
  \times 512$ grid cells.

\section{Conclusions}
We presented a numerical discretisation of the coupled hyperbolic
moment systems which approximate a simplified multiscale model for
sedimentation in suspensions of rod-like particles. For the shear flow
problem, an experimental study confirmed second order convergence.
We adaptively adjusted the
level of detail of the model by coupling moment systems with different
numbers of moment equations. We derived a conservative high-resolution
finite volume method for solving the moment systems with different
resolution. A theoretically justified error indicator
  was introduced and used to determine regions in which an accurate
  approximation requires a higher number of moment equations.

A future goal  is the derivation of physically more
realistic, moment based models and efficient numerical methods that
approximate the dynamics of sedimenting rod-like particles dispersed
in a three-dimensional fluid.

\section*{Acknowledgments}
Funded by the Deutsche Forschungsgemeinschaft (DFG, German Research
Foundation) - SPP 2410 Hyperbolic Balance Laws in Fluid Mechanics:
Complexity, Scales, Randomness (CoScaRa), within the Project  ``A
posteriori error estimators for statistical solutions of barotropic
Navier-Stokes equations''  525877563  and FOR 5409
Structure-preserving Numerical Methods for Bulk and Interface Coupling
of Heterogeneous Models, within the Project ``Structure-Preserving
Methods for Complex Fluids''  463312734.

\bibliographystyle{abbrvurl}
\bibliography{refs}

\end{document}